\documentclass{amsart}
\allowdisplaybreaks[4]
\usepackage{color}

\newcommand{\Z}{\mathbf{Z}}
\newcommand{\R}{\mathbf{R}}
\newcommand{\C}{\mathbf{C}}
\newcommand{\sig}{\sigma}

\newcommand{{\ba}}{\bf a}
\newcommand{\ve}{\varepsilon}
\newcommand{\la}{\lambda}
\newcommand{\La}{\Lambda}
\newcommand{\ga}{\gamma}
\newcommand{\pa}{\partial}
\newcommand{\ra}{\rightarrow}
\newcommand{\Om}{\Omega}
\newcommand{\del}{\delta}
\newcommand{\Del}{\Delta}

\newcommand{\na}{\nabla}
\newcommand{\cd}{\cdot}

\newcommand{\al}{\alpha}

\newcommand{\be}{\begin{equation}}
\newcommand{\ee}{\end{equation}}

\newcommand{\om}{\omega}

\newtheorem{lem}{Lemma}{\bf}{\it}
\newtheorem{remark}{Remark}{\it}{\rm}

\newtheorem{theorem}{Theorem}
\newtheorem{proposition}{Proposition}
\newtheorem{corollary}{Corollary}
\newtheorem{hypothesis}{Hypothesis}
\numberwithin{theorem}{section}
\numberwithin{lem}{section}
\numberwithin{hypothesis}{section}
\numberwithin{equation}{section}
\numberwithin{proposition}{section}
\numberwithin{corollary}{section}
\title[Strong Convergence]{Strong Convergence to the Homogenized Limit of Parabolic Equations with Random Coefficients}

\author{Joseph G. Conlon \  and \  Arash Fahim}

\address{ (Joseph G. Conlon): University of Michigan, Department of Mathematics, Ann Arbor,
  MI 48109-1109}
\email{conlon@umich.edu}

\address{ (Arash Fahim): University of Michigan, Department of Mathematics, Ann Arbor,
  MI 48109-1109}
\email{fahimara@umich.edu}


\begin{document}

\maketitle

\begin{abstract}
This paper is concerned with the study of solutions to discrete parabolic equations in divergence form with random coefficients, and their convergence to solutions of a homogenized equation. It has previously been shown that  if the random environment is translational invariant and ergodic, 
then solutions of the random equation converge under diffusive scaling  to solutions of a homogenized parabolic PDE.  In this paper point-wise estimates are obtained on the difference between the averaged solution to the random equation and the solution to the homogenized equation for certain random environments which are strongly mixing. 
\end{abstract}

\section{Introduction.}
Let $(\Om,\mathcal{F},P)$ be a probability space and  denote by $\langle \  \cd \ \rangle$ expectation w.r. to the measure $P$.   We assume that the $d$ dimensional integer lattice $\Z^d$ acts on $\Om$ by space translation operators $\tau_{x,0}:\Om\ra\Om, \ x\in\Z^d$, which are measure preserving and satisfy the properties $\tau_{x,0}\tau_{y,0}=\tau_{x+y,0}, \ \tau_{0,0}= \ {\rm identity}, \ x,y\in\Z^d$.  We assume also that either the integers $\Z$ or the real line $\R$ acts on $\Om$ by time translation operators $\tau_{0,t}:\Om\ra\Om, $ where $t\in\Z$ in the former case and $t\in\R$ in the latter. In either case we assume that  for all $t,s$, one has $\tau_{0,t}\tau_{0,s}=\tau_{0,t+s}$,  and that the operators $\tau_{0,t}$ commute with the operators $\tau_{x,0}$, so we may set $\tau_{x,t}=\tau_{x,0}\tau_{0,t}= \tau_{0,t}\tau_{x,0}$. 

Consider a bounded measurable function  ${\bf a}:\Om\ra\R^{d(d+1)/2}$  from $\Om$ to the space of symmetric $d\times d$ matrices which satisfies the quadratic form inequality  
\begin{equation} \label{A1}
\la I_d \le {\bf a}(\om) \le \La I_d, \ \ \ \ \ \om\in\Om,
\end{equation}
where $I_d$ is the identity matrix in $d$ dimensions and $\La, \la$
are positive constants. In the case when $\Z$ acts on $\Om$ by operators $\tau_{0,t}$, we shall be interested in solutions $u(x,t,\om)$ to the discrete parabolic equation
\be \label{B1}
u(x,t+1,\om)-u(x,t,\om) \ = \ -\nabla^*{\bf a}(\tau_{x,t}\om)\nabla u(x,t,\om) \ , \quad x\in \Z^d, \ t\ge0, \ \om\in\Om,
\ee
with initial data
\be \label{C1}
u(x,0,\om) \ = \ h(x), \quad x\in\Z^d, \ \om\in\Om \ .
\ee
In the case when $\R$ acts on $\Om$ by operators $\tau_{0,t}$, we shall be interested in solutions $u(x,t,\om)$ to the corresponding continuous in time, discrete in space parabolic equation
\be \label{D1}
\frac{\pa u(x,t,\om)}{\pa t} \ = \ -\nabla^*{\bf a}(\tau_{x,t}\om)\nabla u(x,t,\om) \ , \quad x\in \Z^d, \ t\ge0, \ \om\in\Om,
\ee
with initial data (\ref{C1}).  In (\ref{B1}) and (\ref{D1}) we take  $\nabla$ to be the discrete gradient operator, which has adjoint $\nabla^*$. Thus $\nabla$ is a $d$ dimensional { \it column} operator and $\nabla^*$ a $d$ dimensional {\it row} operator, which act on  functions $\phi:\Z^d\ra\R$ by
\begin{eqnarray} \label{E1}
\na \phi(x) &=& \big( \na_1 \phi(x),... \ \na_d\phi(x) \big), \quad  \na_i \phi(x) = \phi (x + {\bf e}_i) - \phi(x),  \\
\na^* \phi(x) &=& \big( \na^*_1 \phi(x),... \ \na^*_d\phi(x) \big), \quad  \na^*_i \phi(x) = \phi (x - {\bf e}_i) - \phi(x). \nonumber
\end{eqnarray}
In (\ref{E1}) the vector  ${\bf e}_i \in \Z^d$ has 1 as the ith coordinate and 0 for the other coordinates, $1\le i \le  d$. 

One expects that if the translation operators $\tau_{x,t}$ are ergodic  on $\Om$ then solutions to the random equation (\ref{B1}) or (\ref{D1}) converge to solutions of a constant coefficient homogenized equation under diffusive scaling. Thus suppose $f:\R^d\ra\R$ is a $C^\infty$ function with compact support and for $\ve$ satisfying $0<\ve\le 1$ set $h(x)=f(\ve x), \ x\in\Z^d$, in (\ref{C1}), and let  $u_\ve(x,t,\om)$ denote the corresponding solution to (\ref{B1}) or (\ref{D1}) with this initial data.  It has been shown in \cite{loy}, just assuming  ergodicity of the translation operators,  that $u_\ve(x/\ve,t/\ve^2,\om)$ converges in probability as $\ve\ra 0$ to a function $u_{\rm hom}(x,t), \ x\in\R^d, \ t>0$, which is the solution to a constant coefficient parabolic PDE
\be \label{F1}
\frac{\pa u_{\rm hom}(x,t)}{\pa t} \ = \ -\na^* {\bf a}_{\rm hom}\na u_{\rm hom}(x,t) \ , \quad x\in\R^d, \ t>0,
\ee
with initial condition
\be \label{G1}
u_{\rm hom}(x,0) \ = \ f(x), \quad x\in\R^d \ .
\ee
The $d\times d$ symmetric matrix ${\bf a}_{\rm hom}$ in (\ref{F1}) satisfies the quadratic form inequality (\ref{A1}).   Similar results under various ergodic type assumptions on $\Om$ can be found in  \cite{bmp,cps, dkl,r}. In time-independent environments the corresponding results for elliptic equations in divergence form have been proven much earlier  -see \cite{k1,k2,pv,zko}.

In this paper we shall confine ourselves to studying the expectation $\langle \ u(x,t,\cdot) \ \rangle$ of the solution $u(x,t,\om)$ to (\ref{B1}) and(\ref{D1}) with initial condition (\ref{C1}). Our first theorem is consistent with the result of \cite{loy} already mentioned:
\begin{theorem}
Let  $f:\R^d\ra\R$ be a $C^\infty$ function of compact support and set $h(x)=f(\ve x), \ x\in\Z^d$ in (\ref{C1}). 
For the translation group $\tau_{x,t}, \ x\in\Z^d,t\in\Z,$ on $\Om$ assume that one of the operators $\tau_{\mathbf{e}_j,0}, \ j=1,..,d,$ or $\tau_{0,1}$ is ergodic on $(\Om,\mathcal{F},P)$. 
 Then if $4d\La\le 1$ the solution $u_\ve(x,t,\om)$ of (\ref{B1}) with initial data (\ref{C1}) has the property
\be \label{H1}
\lim_{\ve\ra 0}\sup_{x\in\ve\Z^d, t\in\ve^2\Z^+}| \langle \ u_\ve(x/\ve,t/\ve^2,\cdot) \ \rangle - u_{\rm hom}(x,t)| \ = \ 0.
\ee
For the translation group $\tau_{x,t}, \ x\in\Z^d,t\in\R,$ on $\Om$ assume that one of the operators $\tau_{\mathbf{e}_j,0}, \ j=1,..,d,$ or the continuous $1$ parameter group $\tau_{0,t}, \ t\in\R,$ is ergodic on $(\Om,\mathcal{F},P)$.  Then the solution $u_\ve(x,t,\om)$ of (\ref{D1}) with initial data (\ref{C1}) has the property
\be \label{I1}
\lim_{\ve\ra 0}\sup_{x\in\ve\Z^d, t>0}| \langle \ u_\ve(x/\ve,t/\ve^2,\cdot) \ \rangle - u_{\rm hom}(x,t)| \ = \ 0.
\ee
\end{theorem}

It has been shown in the case of homogenization of elliptic equations in divergence form with random coefficients,  that a rate of convergence in homogenization can be obtained provided the random environment satisfies some {\it quantitative strong mixing property}.  The first results in this direction were proven in the 1980's  by Yurinski \cite{y}, but there have been several papers more recently extending his work.  In particular, Caffarelli and Souganidis \cite{cs} have obtained rates of convergence results in homogenization of fully nonlinear PDE. In recent work of Gloria and Otto  \cite{go1}  an optimal rate of convergence result was obtained for linear elliptic equations in divergence form. Following an idea of Naddaf and Spencer  \cite{ns2},  they express the quantitative strong mixing assumption as a Poincar\'{e} inequality. This formulation of  the strong mixing assumption is very useful when the random environment is a Euclidean field theory with uniformly convex Lagrangian.

For the case of parabolic equations in divergence form with random coefficients, we were unable to find in the literature any results on rate of convergence in homogenization. Here we shall obtain a rate of convergence, but only for the averaged solution to the parabolic equation as in Theorem 1.1, and for two particular environments.  For the discrete time problem  (\ref{B1}), (\ref{C1}) we assume the environment is i.i.d. That is we assume the variables $\mathbf{a}(\tau_{x,t}\cdot), \ (x,t)\in\Z^{d+1},$ are   i.i.d.  For the continuous time problem  (\ref{C1}), (\ref{D1}) we assume the environment is the stationary process associated with a massive Euclidean field theory. 

This Euclidean field theory is determined by a potential $V : \R^d \ra \R$, which is a $C^2$ uniformly convex function, and a mass $m>0$.  Thus the second derivative ${\bf a}(\cdot)=V''(\cdot)$ of $V(\cdot)$ is assumed to satisfy the inequality (\ref{A1}). Consider functions 
$\phi : \Z^d\times\R \ra \R$ which we denote as $\phi(x,t)$ where $x$ lies on the integer lattice $\Z^d$ and $t$ on the real line $\R$.  Let $\Om$ be the space of
all such functions which have the property that for each $x\in\Z^d$ the function $t\ra\phi(x,t)$ on $\R$ is continuous, and $\mathcal{F}$ be the Borel algebra generated by finite 
dimensional rectangles 
$\{ \phi(\cdot,\cdot) \in \Om: \  |\phi(x_i,t_i) - a_i| < r_i, \ i=1,...,N\}$, where
$(x_i,t_i) \in \Z^d\times\R, \ a_i \in \R, \ r_i > 0, \ i=1,...,N, \ N \ge 1$.  The translation operators $\tau_{x,t}:\Om\ra\Om, \ (x,t)\in\Z^d\times\R$, are defined by $\tau_{x,t}\phi(z,s)=\phi(x+z,t+s), \ z\in\Z^d,s\in\R$. 

For any $d\ge 1$ and $m>0$ one can define \cite{c1,fs} a unique ergodic translation invariant probability 
measure $P$ on $(\Om, \mathcal{F})$ which depends on the function $V$ and $m$.  In this measure the variables $\phi(x,t), \ x\in\Z^d,t>0,$ conditioned on the variables $\phi(x,0), \ x\in\Z^d,$ are determined as solutions of the infinite dimensional stochastic differential  equation 
\be \label{J1}
d\phi(x,t) \ = \ -\frac{\pa}{\pa\phi(x,t)}\sum_{x'\in\Z^d} \frac{1}{2}\{V(\na\phi(x',t))+m^2\phi(x',t)^2/2\} \ dt +dB(x,t) \ , \quad x\in\Z^d, t>0, 
\ee
 where $B(x,\cdot), \ x\in\Z^d,$ are independent copies of Brownian motion. Formally the invariant measure for the Markov process (\ref{J1}) is the Euclidean field theory measure
\be \label{K1}
\exp \left[ - \sum_{x\in \Z^d} V\left( \na\phi(x)\right)+m^2\phi(x)^2/2 \right] \prod_{x\in \Z^d} d\phi(x)/{\rm normalization}.
\ee
Hence if the variables $\phi(x,0), \ x\in\Z^d,$ have distribution determined by (\ref{K1}), then $\phi(\cdot,t), \ t>0,$ is a stationary process and so can be extended to all $t\in\R$ to yield a measure $P$ on $(\Om,\mathcal{F})$. For this measure the translation operators $\tau_{x,t}, \ (x,t)\in\Z^d\times\R,$ form a group of  measure preserving transformations on $(\Om,\mathcal{F},P)$.  
\begin{theorem}
Let  $f:\R^d\ra\R$ be a $C^\infty$ function of compact support and set $h(x)=f(\ve x), \ x\in\Z^d$ in (\ref{C1}). 
 Then if $4d\La\le 1$ and the variables  $\mathbf{a}(\tau_{x,t}\cdot), \ (x,t)\in\Z^{d+1},$ are   i.i.d.,  the solution $u_\ve(x,t,\om)$ of (\ref{B1}) with initial data (\ref{C1}) has the property
\be \label{L1}
\sup_{x\in\ve\Z^d, t\in\ve^2\Z^+}| \langle \ u_\ve(x/\ve,t/\ve^2,\cdot) \ \rangle - u_{\rm hom}(x,t)| \ \le \ C\ve^\al, \quad 0<\ve\le 1, 
\ee
where $\al>0$ is a constant depending only on $d,\La/\la$ and  $C$ is a constant depending only on $d,\La,\la$ and the function $f(\cdot)$.  

Let $\tilde{{\bf a}}:\R\ra\R^{d(d+1)/2}$  be a $C^1$ function on $\R$ with values in the space of symmetric $d\times d$ matrices which satisfy the quadratic form inequality  (\ref{A1}). Let $(\Om, \mathcal{F}, P)$ be the probability space of  fields $\phi(\cdot,\cdot)$ determined by (\ref{J1}), (\ref{K1}) and set ${\bf a}(\cdot)$ in (\ref{D1}) to be ${\bf a}(\phi)=\tilde{{\bf a}}(\phi(0,0)), \ \phi\in\Om$.  
Suppose in addition that the derivative $D\tilde{{\bf a}}(\cdot)$ of $\tilde{{\bf a}}(\cdot)$ satisfies  the inequality $\|D\tilde{{\bf a}}(\cdot)\|_{\infty}\le \La_1$.  
Then the solution $u_\ve(x,t,\om)$ of (\ref{D1}) with initial data (\ref{C1}) has the property
\be \label{M1}
\sup_{x\in\ve\Z^d, t>0}| \langle \ u_\ve(x/\ve,t/\ve^2,\cdot) \ \rangle - u_{\rm hom}(x,t)| \ \le \  C\ve^\al, \quad 0<\ve\le1,
\ee
where $\al>0$ is a constant depending only on $d,\La/\la$ and  $C$ is a constant depending only on $d,\La,\la,m,\La_1$ and the function $f(\cdot)$.  
\end{theorem}
We consider what  Theorem 1.2  tells us about the expectation of the Green's function for  the equations (\ref{B1}) and (\ref{D1}).  By translation invariance of the measure we have that
\be \label{N1}
\langle \  u(x,t,\cdot) \ \rangle=\sum_{y\in\Z^d} G_{{\bf a}}(x-y,t)h(y), \quad x\in\Z^d,
\ee
where  $G_{{\bf a}}(x,t)$ is the expected value of the Green's function. Setting $h(x)=f(\ve x), \ x\in\Z^d$, then (\ref{N1}) may be written as
\be \label{O1}
\langle \ u_\ve(x/\ve,t/\ve^2,\cdot) \ \rangle=\int_{\ve Z^d} \ve^{-d}G_{{\bf a}}\left(\frac{x-z}{\ve},\frac{t}{\ve^2}\right) f(z)  \ dz, \quad x\in \ve \Z^d,
\ee
where integration over $\ve\Z^d$ is defined by
\be \label{P1}
\int_{\ve Z^d} g(z) \ dz \ =  \ \sum_{z\in\ve\Z^d} g(z) \ \ve^d.
\ee
Let $G_{{\bf a}_{\rm hom}}(x,t), \ x\in\R^d, \ t>0$, be the Greens function for the PDE (\ref{F1}). One easily sees that $G_{{\bf a}_{\rm hom}}(\cdot,\cdot)$ satisfies the scaling property
\be \label{Q1}
\ve^{-d} G_{{\bf a}_{\rm hom}}(x/\ve,t/\ve^2) \ = \  G_{{\bf a}_{\rm hom}}(x,t), \quad \ve>0, \ x\in\R^d, \ t>0.
\ee
Hence Theorem 1.2 implies that averages of $\ve^{-d}G_{{\bf a}}(x/\ve,t/\ve^2)-\ve^{-d} G_{{\bf a}_{\rm hom}}(x/\ve,t/\ve^2)$ with respect to $x\in\ve\Z^d$ are bounded by $C\ve^\al$ for some constant $C$.  Conversely Theorem 1.2 is implied by the {\it point-wise} estimate on Green's functions:
\be \label{R1}
\big|\ve^{-d} G_{{\bf a}}(x/\ve,t/\ve^2)-\ve^{-d}G_{{\bf a}_{\rm hom}}(x/\ve,t/\ve^2)\big| \ \le \ \frac{C\ve^\alpha}{[\La t+\ve^2]^{(d+\alpha)/2}} \exp\left[-\ga\min\left\{\frac{|x|}{\ve}, \ \frac{ |x|^2}{\La t+\ve^2}\right\}\right] \ ,
\ee
provided $\La t\ge \ve^2$ and $x\in\ve\Z^d$.

It is  clear that the inequality (\ref{R1}) for $\ve<1$ follows from the same inequality for $\ve=1$:
\be \label{S1}
\big| G_{{\bf a}}(x,t)-G_{{\bf a}_{\rm hom}}(x,t)\big| \ \le \ \frac{C}{[\La t+1]^{(d+\alpha)/2}} \exp\left[-\ga\min\left\{|x|, \ \frac{ |x|^2}{\La t+1}\right\}\right] \ ,
\ee
provided $\La t\ge 1$ and $x\in\Z^d$.
We shall prove such an inequality and also  similar inequalities for the derivatives of the expectation of the Green's function,
\be \label{T1}
\big| \na G_{{\bf a}}(x,t)-\na G_{{\bf a}_{\rm hom}}(x,t)\big| \ \le \ \frac{C}{[\La t+1]^{(d+1+\alpha)/2}} \exp\left[-\ga\min\left\{|x|, \ \frac{ |x|^2}{\La t+1}\right\}\right] \ , 
\ee
\be \label{U1}
\big|\na\na  G_{{\bf a}}(x,t)-\na\na G_{{\bf a}_{\rm hom}}(x,t)\big| \ \le \ \frac{C}{[\La t+1]^{(d+2+\alpha)/2}} \exp\left[-\ga\min\left\{|x|, \ \frac{ |x|^2}{\La t+1}\right\}\right] \ .  
\ee
\begin{theorem} Let $(\Om,\mathcal{F},P)$ and $\mathbf{a}(\om), \ \om\in\Om,$ be as in the statement of Theorem 1.2. Then for $d\ge 1$ there exists $\alpha,\ga>0$ depending only on $d$ and  $\La/\la$, such that (\ref{S1}), (\ref{T1}) and (\ref{U1}) hold for some positive constant $C$. In the discrete time case $C$ depends only on $d,\La,\la$, and in the continuous time case only on
$d,\la,\La,m,\La_1$.  
\end{theorem}
The proofs of Theorem 1.2 and Theorem 1.3 follow the same lines as the proofs of the corresponding results for elliptic equations in \cite{cs}.  One begins with a Fourier representation for the average of the solution to the random parabolic equation, which was obtained in \cite{c2}. Then for the i.i.d. environment the generalization by Jones \cite{j} of the Calderon-Zygmund theorem \cite{cz} to parabolic multipliers, together with some interpolation inequalities,  yields Theorem 1.2 and the inequalities (\ref{S1}), (\ref{T1}) of Theorem 1.3 in the discrete time case.  Similarly to \cite{cs} we need to use the result of Delmotte and Deuschel \cite{dd} on the H\"{o}lder continuity of the second difference $\na\na  G_{{\bf a}}(x,t)$ in order to prove (\ref{U1}).   In the continuous time case  we need in addition to prove some Poincar\'{e} inequalities for time dependent fields. To do this we follow the methodology of Gourcy-Wu \cite{gw} by using the Clark-Ocone formula \cite{nu}.

\section{Fourier Space Representation and Homogenization}
In this section we shall prove the homogenization result Theorem 1.1. The proof of this is based on a Fourier representation for the solutions of (\ref{B1}), (\ref{D1}) which was given in \cite{c2}. 

We begin by  summarizing relevant results from \cite{c2} for the discrete time equation (\ref{B1}). Thus we are assuming a probability space $(\Om,\mathcal{F},P)$ and a set of translation operators $\tau_{x,t}, \ x\in\Z^d,t\in\Z,$ acting on $\Om$.  For $\xi\in\R^d$ and  $\psi:\Om\ra\C$  a measurable function  we define the $\xi$ derivative of $\psi(\cdot)$ in the $j$ direction $\pa_{j,\xi}$, and its adjoint  $\pa^*_{j,\xi}$, by
\begin{eqnarray} \label{A2}
\pa_{j,\xi} \psi(\om) \ &=& \  e^{-i{\bf e}_j.\xi}\psi(\tau_{{\bf e}_j,0} \ \om)-\psi(\om),  \\
\pa^*_{j,\xi} \psi(\om) \ &=& \   e^{i{\bf e}_j.\xi}\psi(\tau_{-{\bf e}_j,0} \ \om)-\psi(\om). \nonumber
\end{eqnarray}
The  $d$ dimensional column $\xi$ derivative operator  $\pa_\xi$ is then given by  $\pa_\xi=(\pa_{1,\xi},....,\pa_{d,\xi})$. Similarly to  (\ref{E1}) its adjoint  $\pa_\xi^*$ is  given by the row operator $\pa_\xi^*=(\pa^*_{1,\xi},....,\pa^*_{d,\xi})$. Let  $P:L^2(\Om)\ra L^2(\Om)$ be the projection orthogonal to the constant function and $\eta\in\C$ with $\Re\eta>0$. Then there is a unique square integrable solution $\Phi(\xi,\eta,\om)$ to the equation
\be \label{B2}
e^\eta\Phi(\xi,\eta,\tau_{0,1}\om)-\Phi(\xi,\eta,\om) +P\pa_\xi^*{\bf a}(\om)\pa_\xi\Phi(\xi,\eta,\om)=-P\pa^*_\xi {\bf a}(\om), 
\ee
provided $4d\La\le 1$. Thus there is a unique row vector $\Phi(\xi,\eta,\om)=[\Phi_1(\xi,\eta,\om),....,\Phi_d(\xi,\eta,\om)]$ with $\Phi_j(\xi,\eta,\cdot)\in L^2(\Om), \ j=1,..,d,$ which satisfies (\ref{B2}). Let $q(\xi,\eta)=[q_{r,r'}(\xi,\eta)]$ be the $d\times d$ matrix function given in terms of the solution to (\ref{B2}) by the formula
\be \label{C2}
q(\xi,\eta) \ = \  \langle \ {\bf a}(\cdot) \ \rangle+ \langle \ {\bf a}(\cdot)\pa_\xi  \Phi(\xi,\eta,\cdot) \ \rangle \ .
\ee
We define the $d$ dimensional periodic column vector $e(\xi)\in\C^d$ to have $j$th entry given by the formula $e_j(\xi)=e^{-i{\bf e}_j\cdot\xi}-1, \  1\le j\le d$. It was shown in \cite{c2} that the solution $u(x,t,\om), \ x\in\Z^d, t=0,1,2,..,\om\in\Om$,  to the initial value problem (\ref{B1}), (\ref{C1}) has the representation 
\be \label{D2}
u(x,t,\om)  =  \frac{1}{(2\pi)^{d+1}}\int_{[-\pi,\pi]^{d+1}} 
\frac{\hat{h}(\xi)e^{-i\xi.x+\eta(t+1)}}{e^\eta-1+e(\xi)^*q(\xi,\eta)e(\xi)}\left[1+\Phi(\xi,\eta,\tau_{x,t}\om)e(\xi)\right] \ d[\Im\eta] \ d\xi  ,   
\ee
where $\hat{h}(\cdot)$ is the Fourier transform of $h(\cdot)$, 
\be \label{E2}
\hat{h}(\xi) \ = \ \sum_{x\in\Z^d} h(x)e^{i\xi\cdot x} \ .
\ee
It follows from (\ref{D2}) that the Fourier transform $\hat{G}_{{\bf a}}(\xi,\eta)$ of   the averaged Green's function $G_{{\bf a}}(\cdot,\cdot)$ for (\ref{B1}), (\ref{C1}) given by
\be \label{F*2}
\hat{G}_{{\bf a}}(\xi,\eta) \ = \  \sum_{t=0}^\infty\sum_{x\in\Z^d} G_{{\bf a}}(x,t)\exp[ix.\xi-\eta t] \ ,
\ee
has the representation  
\be \label{G*2}
\hat{G}_{{\bf a}}(\xi,\eta) \ = \  e^{\eta}/[e^\eta-1+e(\xi)^*q(\xi,\eta)e(\xi)] \ .
\ee

The solution to (\ref{D2}) can be generated by a convergent perturbation expansion. Let $\mathcal{H}(\Om)$ be the Hilbert space of measurable functions $\psi:\Om\ra\C^d$ with norm $\|\psi\|$ given by $\|\psi\|^2=\langle \ |\psi(\cdot)|^2 \ \rangle$. We define an operator $T_{\xi,\eta}$  on  $\mathcal{H}(\Om)$  as follows: For any $g\in \mathcal{H}$, let 
$\psi(\xi,\eta,\om)$ be  the solution to the equation
\be \label{F2}
\frac{1}{\La}\left[e^{\eta} \ \psi(\xi,\eta,\tau_{0,1}\om)-\psi(\xi,\eta,\om)\right]+\pa_\xi^*\pa_\xi\psi(\xi,\eta,\om)=\pa^*_\xi g(\om) \ .
\ee
Then $T_{\xi,\eta} g(\cdot)=\pa_\xi\psi(\xi,\eta,\cdot)$, or more explicitly
\be \label{G2}
T_{\xi,\eta} g(\om) \ = \   \La\sum_{t=0}^\infty e^{-\eta(t+1)}\sum_{x\in \Z^d} \left\{\nabla\nabla^* G_{\La}(x,t)\right\}^*\exp[-ix.\xi]  \ g(\tau_{x,-t-1}\om),
\ee 
where $G_\La(\cdot)$ is the solution to the initial value problem
\begin{multline} \label{H2}
G_\La(x,t+1)-G_\La(x,t)+\La\nabla^*\nabla G_\La(x,t) \ = \ 0, \quad x\in\Z^d, \ t\in\Z, \ t\ge 0, \\
G_\La(x,0) \ = \ \del(x), \quad x\in\Z^d \ .
\end{multline}
Equation (\ref{H2}) has a unique solution provided $4d\La\le 1$, and the function $G_\La(x,t)$ satisfies an inequality
\be \label{I2}
0\le G_\La(x,t)\le \frac{C_d}{[\La t+1]^{d/2}}\exp\left[-\frac{\min\left\{|x|, \ |x|^2/(\La t+1)\right\}}{C_d}\right] \ ,
\ee
where $C_d>0$ is a constant depending only on dimension $d$. 
The operator $T_{\xi,\eta} $ is  bounded on $\mathcal{H}(\Om)$  with $\|T_{\xi,\eta}\|\le 1$, provided $\xi\in\R^d, \ \Re\eta>0$. 
Now on setting ${\bf a}(\cdot)=\La[I_d-{\bf b}(\cdot)]$, one sees  that (\ref{B2}) is equivalent to the equation
\be \label{J2}
\pa_\xi\Phi(\xi,\eta,\cdot)=PT_{\xi,\eta}[{\bf b}(\cdot)\pa_\xi\Phi(\xi,\eta,\cdot)]
+PT_{\xi,\eta}[{\bf b}(\cdot)] \ .
\ee
Since  $\|T_{\xi,\eta}\|\le 1$ and $\|{\bf b}(\om)\|\le 1-\la/\La, \ \om\in\Om$,  the Neumann series for the solution to (\ref{J2}) converges in $\mathcal{H}(\Om)$. 

It will be useful later to express the operator $T_{\xi,\eta}$ in its Fourier representation. To do this we use the standard notation for the Fourier transform  of a function $h:\Z^{d+1}\ra\C$ which we denote by  $h(x,t), \ x\in\Z^d,t\in\Z$. Letting $\hat{h}(\zeta,\theta), \ \zeta\in[-\pi,\pi]^d, \ \theta\in[-\pi,\pi],$ be the Fourier transform of  $h(\cdot,\cdot)$, then 
\be \label{AE2}
\hat{h}(\zeta,\theta) \ = \ \sum_{x\in\Z^d} \sum_{t\in\Z} h(x,t)e^{ix.\zeta+it\cdot\theta} \ .
\ee
The Fourier inversion formula yields
\be \label{AF2}
h(x,t) \ = \  \frac{1}{(2\pi)^{d+1}}\int_{[-\pi,\pi]^{d+1}} \hat{h}(\zeta,\theta)e^{-ix.\zeta-it\theta} \ d\zeta \ d\theta, \quad x\in\Z^d, t\in\Z \ .
\ee
Now the action of the translation group $\tau_{x,0}, \ x\in\Z^d$, on $\Om$ can be described by a set $A_1,...,A_d$ of commuting self-adjoint operators on $L^2(\Om)$, so that
\be \label{AG2}
f(\tau_{x,0}\cdot) \ = \ \exp[ix.A] f(\cdot), \quad x\in\Z^d, \ f\in L^2(\Om),
\ee
where $A=(A_1,..,A_d)$.  Similarly the action of the translation group $\tau_{0,t}, \ t\in\Z$, on $\Om$ can be described by a  self-adjoint operator $B$ on $L^2(\Om)$ which commutes with $A_1,..,A_d$, so that
\be \label{AH2}
f(\tau_{0,t}\cdot) \ = \ \exp[-itB] f(\cdot), \quad t\in\Z, \ f\in L^2(\Om),
\ee
It follows then from (\ref{G2}),(\ref{AG2}), (\ref{AH2}) that
\be \label{AI2}
T_{\xi,\eta} g(\cdot) \ = \  \frac{\La e(\xi-A)e^*(\xi-A)}{e^{\eta-iB}-1+\La e(\xi-A)^*e(\xi-A)}  \ g(\cdot) \ .
\ee

The Neumann series for the solution to (\ref{J2})  yields  a convergent perturbation expansion for the function $q(\xi,\eta)$ of (\ref{C2}).  Thus   for $m=1,2...$, let the matrix function $h_m(\xi,\eta)$ be defined for $\Re\eta> 0, \ \xi\in \R^d$, by
\be \label{W2}
h_m(\xi,\eta)=\langle \ {\bf b}(\cdot)\left[ PT_{\xi,\eta}{\bf b}(\cdot)\right]^m \ \rangle \ ,
\ee
whence (\ref{C2}), (\ref{J2}) imply that
\be \label{X2}
q(\xi,\eta) \ = \  \langle \ {\bf a}(\cdot) \ \rangle-\La \sum_{m=1}^\infty h_m(\xi,\eta) \ .
\ee
It is easy to see that the function $q(\xi,\eta)$ is $C^\infty$ for $\xi\in\R^d, \ \Re\eta>0$.  As in \cite{c2,cs} we can extend this result as follows:
\begin{proposition}
Suppose that $4d\La\le 1$ and any of the translation operators  $\tau_{{\bf e}_j,0}, \ 1\le j\le d,$ or $\tau_{0,1}$ is ergodic  on $\Om$.  Then the limit $\lim_{(\xi,\eta)\ra(0,0)} q(\xi,\eta)=q(0,0)$ exists. If any of the translation operators   is weak mixing \cite{p} on $\Om$ then  $q(\xi,\eta),  \  \xi\in\R^d,\ \Re\eta>0$, extends to a continuous function on $\xi\in\R^d, \ \Re\eta\ge 0$. 
\end{proposition}
\begin{proof}
We follow the same argument as in Lemma 2.5 of \cite{c2} and Proposition 2.1 of \cite{cs}.
\end{proof}
\begin{remark}
Note that the projection operator $P$ in equation  (\ref{B2}) plays a critical role in establishing  continuity. For a constant function $g(\cdot)\equiv v\in\C^d$, one has
\be \label{R2}
T_{\xi,\eta} g(\cdot) \ = \  [e(\xi)^*v]e(\xi) \big/ [(e^\eta-1)/\La+ e(\xi)^*e(\xi)] \ ,
\ee
which does not extend to a continuous function of $(\xi,\eta)$ on the set $\xi\in\R^d, \ \Re\eta\ge 0$.
\end{remark}

Next we show that the function $q(\xi,\eta)$ with domain $\xi\in\R^d, \  \Re\eta>0$, can be extended to complex $\xi=\Re\xi+i\Im\xi\in\C^d$ with small imaginary part.
\begin{lem}
The $C^\infty$ operator valued function $(\xi,\eta)\ra T_{\xi,\eta}$ with domain $\{(\xi,\eta):\xi\in\R^d, \ \Re\eta>0\}$ and range  the space of bounded linear operators $\mathcal{B}[\mathcal{H}(\Om)]$ on  $\mathcal{H}(\Om)$, has an analytic continuation to a region $\{(\xi,\eta)\in\C^{d+1}: 0<\Re\eta<\La, \ |\Im\xi|<C_1\sqrt{\Re\eta/\La}\}$, where $C_1$ is a constant depending only on $d$. For $(\xi,\eta)$ in this region the norm of $T_{\xi,\eta}$ satisfies the inequality $\|T_{\xi,\eta}\|\le 1+C_2 |\Im\xi|^2/[
\Re\eta/\La]$, where the constant $C_2$ depends only on $d$.  
\end{lem} 
\begin{proof}
That there is an analytic continuation to the region $\{\xi\in\C^d: 0<\Re\eta<\La, \ |\Im\xi|<C_1\sqrt{\Re\eta/\La}\}$ is a consequence of the fact that $|\na\na^*G_\La(x,t)|$ is bounded by $(\La t+1)^{-1}$ times the RHS of (\ref{I2}).  For $(\xi,\eta)$ in this region one has that
\be  \label{S2}
\Re[-\eta(t+1)-ix\cdot\xi] \ \le  \sup_{|\theta|<1}[-\theta^2\La (t+1)+C_1\theta |x|] 
\le \ \min\{C_1|x|, \ C_1^2|x|^2/4\La(t+1)\} \ .
\ee
Hence using the representation (\ref{G2}) for $T_{\xi,\eta}$, we see  that the analytic continuation extends to any region $\{\xi\in\C^d: 0<\Re\eta<\La, \ |\Im\xi|<C_1\sqrt{\Re\eta/\La}\}$ provided $C_1$ satisfies the inequalities $C_1<C_d^{-1}, \ C_1^2<4C_d^{-1}$, where $C_d$ is the constant in (\ref{I2}). 

The bound on $\|T_{\xi,\eta}\|$ can be obtained from (\ref{F2}).  Thus on multiplying (\ref{F2}) by $\bar{\psi}(\xi,\eta,\tau_{0,1}\om)$ we see that
\be \label{T2}
\frac{e^{\Re\eta}}{\La} \langle \ |\psi(\xi,\eta,\cdot)|^2 \ \rangle \ \le \  \frac{1}{\La}|\langle \ \bar{\psi}(\xi,\eta,\tau_{0,1}\cdot)[I-\La\pa_\xi^*\pa_\xi]\psi(\xi,\eta,\cdot) \ \rangle| + |\langle \ \bar{\psi}(\xi,\eta,\tau_{0,1}\cdot)\pa^*_\xi g(\cdot) \ \rangle| \ .
\ee
Since $4d\La\le 1$ it follows that for $\xi\in\R^d$  the operator $I-\La\pa^*_\xi\pa_\xi$ is symmetric  non-negative definite. Hence if $\xi\in\R^d$ one has that
\begin{multline} \label{U2}
|\langle \ \bar{\psi}(\xi,\eta,\tau_{0,1}\cdot)[I-\La\pa_\xi^*\pa_\xi]\psi(\xi,\eta,\cdot) \ \rangle| \ \le  \\
\frac{1}{2}\langle \ \bar{\psi}(\xi,\eta,\tau_{0,1}\cdot)[I-\La\pa_\xi^*\pa_\xi]\psi(\xi,\eta,\tau_{0,1}\cdot) \ \rangle+
\frac{1}{2}\langle \ \bar{\psi}(\xi,\eta,\cdot)[I-\La\pa_\xi^*\pa_\xi]\psi(\xi,\eta,\cdot) \ \rangle \ .
\end{multline}
Similarly one has that
\be \label{V2}
|\langle \ \bar{\psi}(\xi,\eta,\tau_{0,1}\cdot)\pa^*_\xi g(\cdot) \ \rangle|  \ \le
\frac{1}{2}\langle \ \bar{\psi}(\xi,\eta,\tau_{0,1}\cdot)[\pa_\xi^*\pa_\xi]\psi(\xi,\eta,\tau_{0,1}\cdot) \ \rangle
+\frac{1}{2}\|g(\cdot)\|^2 \ .
\ee
We conclude from (\ref{T2})-(\ref{V2}) that $\|T_{\xi,\eta}\|\le 1$ provided $\xi\in\R^d$ and $\Re\eta>0$. This argument can then be extended as in Lemma 2.1 of \cite{cs} to $\xi\in\C^d$.
\end{proof}
\begin{corollary}
The $d\times d$ matrix function $q(\xi,\eta)$ with domain $\{(\xi,\eta): \xi\in\R^d, \ \Re\eta>0\}$, has an analytic continuation to a region $\{\xi\in\C^d: 0<\Re\eta<\La, \ |\Im\xi|<C_1\sqrt{\la\Re\eta/\La^2}\}$, where $C_1$ is a constant depending only on $d$. There is a constant $C_2$ depending only on $d$ such that for $\xi$ in this region, 
\be \label{Y2}
\|q(\xi,\eta)-q(\Re\xi,\eta)\| \ \le \ \frac{C_2\La^2}{\la} \frac{ |\Im\xi|}{\sqrt{\Re\eta/\La}} \ .
\ee
\end{corollary}
\begin{proof} The fact that $q(\xi,\eta)$ has an analytic continuation follows from the representations (\ref{W2}), (\ref{X2}), Lemma 2.1 and the matrix norm bound $\|{\bf b}(\om)\|\le 1-\la/\La, \ \om\in\Om$. On summing the perturbation series (\ref{X2}), we conclude that for $\xi$ satisfying  $ |\Im\xi|<C_1\sqrt{\la\Re\eta/\La^2}$, then $\|q(\xi,\eta)\|\le  C_2\La^2/\la$ for a constant $C_2$ depending only on $d$, provided $C_1$ is chosen sufficiently small, depending only on $d$. By arguing as in Lemma 2.1 we also see that there are positive constants $C_1, C_2$ such  that 
\be \label{Z2}
\|T_{\xi,\eta}-T_{\Re\xi,\eta}\| \ \le \  C_2 |\Im\xi|/\sqrt{\Re\eta/\La} \ , \quad \xi\in\C^d, \  |\Im\xi|<C_1\sqrt{\Re\eta/\La} \ .
\ee
The inequality (\ref{Y2}) follows from (\ref{Z2}).
\end{proof}
It follows from Corollary 2.1 that for $\xi\in\C^d, \ \eta\in\C$ with fixed $\Im\xi\in\R^d, \ \Re\eta>0$ satisfying  $0<\Re\eta<\La, \ |\Im\xi|<C_1\sqrt{\la\Re\eta/\La^2}$, the periodic matrix function $(\Re\xi,\Im\eta)\ra q(\xi,\eta)$ on $\R^{d+1}$ with fundamental region $[-\pi,\pi]^{d+1}$, is bounded.  
\begin{corollary} There exist positive constants $C_1,C_2$ depending only on $d$ and $\La/\la$ such that  
\be \label{AD2}
|e^\eta-1+e(\bar{\xi})^*q(\xi,\eta)e(\xi)| \ \ge \ C_2[ \ |\eta|+ \La|e(\Re\xi)|^2] \ ,
\ee
provided $0<\Re\eta<\La, \  |\Im\xi|<C_1\sqrt{\Re\eta/\La}$.
\end{corollary} 
\begin{proof}
The inequality (\ref{AD2}) follows from Corollary 2.1 and  Lemma 2.7, Lemma 2.8 of \cite{c2}.
\end{proof}
\begin{proof}[Proof of Theorem 1.1-discrete time case]
Taking $h(x)=f(\ve x),$ we have from (\ref{D2}) that
\be \label{H*2}
\langle \ u_\ve(x/\ve,t/\ve^2,\cdot) \ \rangle  = 
  \frac{1}{(2\pi)^{d+1}}\int_{[-\pi/\ve,\pi/\ve]^{d}}\int_{-\pi/\ve^2}^{\pi/\ve^2} 
\frac{\ve^2\hat{f}_\ve(\xi)e^{-i\xi.x+\eta(t+\ve^2)}}{e^{\ve^2\eta}-1+e(\ve\xi)^*q(\ve\xi,\ve^2\eta)e(\ve\xi)} \ d[\Im\eta] \ d\xi \  ,   
\ee
where
\be \label{I*2}
\hat{f}_\ve(\xi) \ = \ \sum_{y\in\ve\Z^d} \ve^d f(y) e^{iy\cdot\xi} \ .
\ee
We also have that
\be \label{J*2}
u_{\rm hom}(x,t) \ = \ \frac{1}{(2\pi)^{d+1}}\int_{\R^d} \int_{\R}
\frac{\hat{f}(\xi)e^{-i\xi.x+\eta t}}{\eta+\xi^*q(0,0)\xi} \ d[\Im\eta] \ d\xi ,
\ee
where $\hat{f}(\cdot)$ is the Fourier transform of $f(\cdot)$,
\be \label{K*2}
\hat{f}(\xi) \ = \ \int_{\R^d} f(y) e^{iy\cdot\xi} \ dy \ , \quad \xi\in\R^d.
\ee
Since $f:\R^d\ra\R$ is $C^\infty$ of compact support it follows from (\ref{I*2}), (\ref{K*2}) that 
\be \label{L*2}
\sup_{0<\ve\le 1,\xi\in[-\pi/\ve,\pi/\ve]^d} |\hat{f}_\ve(\xi)|(1+|\xi|^2)^N<\infty, \quad \sup_{0<\ve\le 1,\xi\in[-\pi/\ve,\pi/\ve]^d}|\hat{f}_\ve(\xi)-\hat{f}(\xi)|/\ve^2[1+|\xi|^2]<\infty \ , 
\ee
where $N$ in (\ref{L*2}) can be arbitrarily large. 

We first observe from (\ref{L*2}) and  Lemma 2.9, Lemma 2.10 of \cite{c2} that
\be \label{M*2}
\int_{|\xi|>1/\sqrt{\La\ve}}\left|\int_{-\pi/\ve^2}^{\pi/\ve^2}
\frac{\ve^2\hat{f}_\ve(\xi)e^{-i\xi.x+\eta(t+\ve^2)}}{e^{\ve^2\eta}-1+e(\ve\xi)^*q(\ve\xi,\ve^2\eta)e(\ve\xi)} \ d[\Im\eta] \right| \ d\xi \ \le \  C\ve \ ,
\ee
for a constant $C$ depending only on the function $f(\cdot)$ and $d,\la,\La$.   Since the function $q(\xi,\eta)$ is continuous at $(\xi,\eta)=(0,0)$, we similarly see there exists for any $\del>0$ an $\ve(\del)>0$ depending only on $\del,d,\la,\La,$  such that if $\ve\le\ve(\del)$, then
\be \label{N*2}
\left|\int_{-\pi/\ve^2}^{\pi/\ve^2}
\frac{\ve^2e^{\eta(t+\ve^2)}}{e^{\ve^2\eta}-1+e(\ve\xi)^*q(\ve\xi,\ve^2\eta)e(\ve\xi)}-\frac{\ve^2e^{\eta(t+\ve^2)}}{e^{\ve^2\eta}-1+e(\ve\xi)^*q(0,0)e(\ve\xi)} \ d[\Im\eta] \right|  \ \le \  \del \ 
\ee
for all $\xi\in[-\pi/\ve,\pi/\ve]^d$ such that $|\xi|\le 1/\sqrt{\La\ve}$. It follows from  (\ref{L*2}), (\ref{M*2}), (\ref{N*2}) that  for any $\del>0$ there exists $\ve(\del)>0$ depending only on  $\del,d,\la,\La,$  and the function $f(\cdot)$, such that
\be \label{O*2}
\sup_{x\in\ve\Z^d, t\in\ve^2\Z^+} \left| \langle \ u_\ve(x/\ve,t/\ve^2,\cdot) \ \rangle-\int_{[-\pi/\ve,\pi/\ve]^d }\int_{-\pi/\ve^2}^{\pi/\ve^2}\frac{\ve^2\hat{f}(\xi)e^{-i\xi\cdot x+\eta(t+\ve^2)}}{e^{\ve^2\eta}-1+e(\ve\xi)^*q(0,0)e(\ve\xi)} \ d[\Im\eta] \ d\xi\right| \ \le \ \del 
\ee
provided $\ve\le\ve(\del)$.   If we use the identities
\be  \label{P*2}
\frac{1}{2\pi}\int_{-\pi/\ve^2}^{\pi/\ve^2}\frac{\ve^2e^{\eta(t+\ve^2)}}{e^{\ve^2\eta}-1+e(\ve\xi)^*q(0,0)e(\ve\xi)} \ d[\Im\eta]  \ = \  \left[1-e(\ve\xi)^*q(0,0)e(\ve\xi)\right]^{t/\ve^2},
\ee
\be \label{Q*2}
\frac{1}{2\pi}\int_\R \frac{e^{\eta t}}{\eta+\xi^*q(0,0)\xi} \ d[\Im\eta] \  = \  
\exp\left[-\{\xi^*q(0,0)\xi\} t\right]  \ ,
\ee
we can conclude from (\ref{J*2}),(\ref{O*2}) that for any $\del>0$ there exists $\ve(\del)>0$ depending only on  $\del,d,\la,\La,$  and the function $f(\cdot)$, such that
\be \label{R*2}
\sup_{x\in\ve\Z^d, t\in\ve^2\Z^+} \left| \langle \ u_\ve(x/\ve,t/\ve^2,\cdot) \ \rangle-u_{\rm hom}(x,t)\right| \ \le \ \del  \ 
\ee
provided $\ve\le\ve(\del)$. 
\end{proof}
\begin{remark}
It is easy to see that if  $\tau_{\mathbf{e}_j,0}$ for some $j, \ 1\le j\le d,$ or $\tau_{0,1}$ acts ergodically on $\Om$ then
\be \label{S*2}
\lim_{(\xi,\eta)\ra (0,0)} \left[e^{\Re\eta}-1\right]\|\Phi(\xi,\eta,\cdot)\|^2 \ = \ 0.
\ee
In the case of an elliptic equation with random coefficients,  the limit corresponding to (\ref{S*2})  implies that the solution to the random equation converges in distribution to the solution of the homogenized equation \cite{cs}.  This is not the case for the parabolic problem due to the fact that integrand in (\ref{H*2}), when multiplied by an arbitrary bounded function of $\eta$,  can have a logarithmic divergence upon integration   with respect to $\Im\eta$. 
\end{remark}

 In the continuous time case there is a similar development to the above. The solution $u(x,t,\om)$ to (\ref{C1}), (\ref{D1}) has the representation
 \be \label{T*2}
u(x,t,\om)  =  \frac{1}{(2\pi)^{d+1}}\int_{[-\pi,\pi]^{d}} \int_{-\infty}^\infty
\frac{\hat{h}(\xi)e^{-i\xi.x+\eta t}}{\eta+e(\xi)^*q(\xi,\eta)e(\xi)}\left[1+\Phi(\xi,\eta,\tau_{x,t}\om)e(\xi)\right] \ d[\Im\eta]   \ d\xi,   
\ee
where  now the d dimensional row vector $\Phi(\xi,\eta,\om)$ is the solution to the equation
\be \label{AL2}
\eta\Phi(\xi,\eta,\om)+\pa\Phi(\xi,\eta,\om) +P\pa_\xi^*{\bf a}(\om)\pa_\xi\Phi(\xi,\eta,\om)=-P\pa^*_\xi {\bf a}(\om)  \ . 
\ee
In (\ref{AL2}) the operator $\pa$ is the infinitesimal generator of the time translation group $\tau_{0,t}, \ t\in\R$. 
The $d\times d$ matrix function $q(\xi,\eta)$ in (\ref{T*2}) is given in terms of the solution to (\ref{AL2}) by the formula (\ref{C2}). It follows from (\ref{T*2}) that the Fourier transform $\hat{G}_{{\bf a}}(\xi,\eta)$ of  the averaged Green's function $G_{{\bf a}}(\cdot,\cdot)$ for (\ref{D1}) defined by
\be \label{AJ2}
\hat{G}_{{\bf a}}(\xi,\eta) \ = \  \int_{0}^\infty dt\sum_{x\in\Z^d} G_{{\bf a}}(x,t)\exp[ix.\xi-\eta t] \ ,
\ee
has the representation  
\be \label{AK2}
\hat{G}_{{\bf a}}(\xi,\eta) \ = \  1/[\eta+e(\xi)^*q(\xi,\eta)e(\xi)] \ , \quad \xi\in\R^d, \ \Re\eta>0.
\ee
Let $G(x,t), \ x\in\Z^d, \ t>0,$ be the solution to the initial value problem
\begin{eqnarray} \label{AM2}
\frac{\pa G(x,t)}{\pa t} +\na^*\na G(x,t) \ &=& \ 0, \quad x\in\Z^d, \ t>0, \\
G(x,0) \ &=& \ \del(x), \quad x\in\Z^d \ . \nonumber
\end{eqnarray}
Then the equation (\ref{AL2}) is equivalent to (\ref{J2}) where the operator $T_{\xi,\eta}$ is given by the formula
\be \label{AN2}
T_{\xi,\eta} g(\om) \ = \   \La\int_{0}^\infty e^{-\eta t} \ dt\sum_{x\in \Z^d} \left\{\nabla\nabla^* G_{\La}(x,t)\right\}^*\exp[-ix.\xi]  \ g(\tau_{x,-t}\om) \ ,
\ee 
with $G_\La(x,t)=G(x,\La t), \ x\in\Z^d,t>0$.   Note that in the continuous time case there is no restriction on the value of $\La>0$. The operator $T_{\xi,\eta} $ of (\ref{AN2}) is  bounded on $\mathcal{H}(\Om)$  with $\|T_{\xi,\eta}\|\le 1$, provided $\xi\in\R^d, \ \Re\eta>0$, and hence the Neumann series for the solution of (\ref{J2}) converges in $\mathcal{H}(\Om)$. 

As in the discrete time case it will be useful later to express the operator $T_{\xi,\eta}$ in its Fourier representation. To do this we use the standard notation for the Fourier transform  of a function $h:\Z^d\times\R\ra\C$ which we denote by  $h(x,t), \ x\in\Z^d,t\in\R$. Letting $\hat{h}(\zeta,\theta), \ \zeta\in[-\pi,\pi]^d, \ \theta\in\R,$ be the Fourier transform of  $h(\cdot,\cdot)$, then 
\be \label{AO2}
\hat{h}(\zeta,\theta) \ = \ \int_{-\infty}^\infty  \ dt \sum_{x\in\Z^d} h(x,t)e^{ix.\zeta+it\theta} \ .
\ee
The Fourier inversion formula yields
\be \label{AP2}
h(x,t) \ = \  \frac{1}{(2\pi)^{d+1}}\int_{-\infty}^\infty  \int_{[-\pi,\pi]^{d}} \hat{h}(\zeta,\theta)e^{-ix.\zeta-it\theta} \ d\zeta \ d\theta, \quad x\in\Z^d, t\in\R \ .
\ee
Now the action of the translation group $\tau_{x,0}, \ x\in\Z^d$, on $\Om$ can be described by a set $A_1,...,A_d$ of commuting self-adjoint operators on $L^2(\Om)$, so that
\be \label{AQ2}
f(\tau_{x,0}\cdot) \ = \ \exp[ix.A] f(\cdot), \quad x\in\Z^d, \ f\in L^2(\Om),
\ee
where $A=(A_1,..,A_d)$.  Similarly the action of the translation group $\tau_{0,t}, \ t\in\R$, on $\Om$ can be described by a  self-adjoint operator $B$ on $L^2(\Om)$ which commutes with $A_1,..,A_d$, so that
\be \label{AR2}
f(\tau_{0,t}\cdot) \ = \ \exp[-itB] f(\cdot), \quad t\in\R, \ f\in L^2(\Om) \ ,
\ee
whence the infinitesimal generator $\pa$ in (\ref{AL2}) is given by $\pa=-iB$. 
It follows now from (\ref{AN2}),(\ref{AQ2}), (\ref{AR2}) that
\be \label{AS2}
T_{\xi,\eta} g(\cdot) \ = \  \frac{\La e(\xi-A)e^*(\xi-A)}{\eta-iB+\La e(\xi-A)^*e(\xi-A)}  \ g(\cdot) \ .
\ee

The Neumann series for the solution to (\ref{J2})-with the operator $T_{\xi,\eta}$ given now by (\ref{AS2})-  yields  a convergent perturbation expansion (\ref{W2}), (\ref{X2}) for the function $q(\xi,\eta)$.  It is easy to see that the analogues of Proposition 2.1, Lemma 2.1 and Corollary 2.1 continue to hold  for the continuous time case.   In the continuous time analogue of Corollary 2.2 the inequality (\ref{AD2}) is replaced by   
\be \label{AV2}
|\eta+e(\bar{\xi})^*q(\xi,\eta)e(\xi)| \ \ge \ C[ \ |\eta|+ \La|e(\Re\xi)|^2] \ . 
\ee
The inequality (\ref{AV2}) follows from  Lemma 5.3  of \cite{c2}.
\begin{proof}[Proof of Theorem 1.1-continuous time case] We proceed as in the discrete time case replacing (\ref{D2}) by (\ref{T*2}) and using Lemma 5.4, Lemma 5.5 of \cite{c2} in place of Lemma 2.9, Lemma 2.10 of \cite{c2}. 
\end{proof}

\vspace{.1in}

\section{Rate of Convergence in Homogenization}
In this section we shall prove Theorem 1.2 under the assumption that the solutions $\Phi(\xi,\eta,\om)$ of (\ref{B2}), (\ref{AL2}) satisfy a certain property which we describe below.  In $\S5$ we shall show that this property holds  for the independent variable environment, and in $\S6$ for the massive field theory environment.  We first consider the discrete time case, whence  $\Phi(\xi,\eta,\om)$ is a solution to (\ref{B2}).

For $1\le p\le\infty$ let $L^p(\Z^{d+1},\C^d\otimes\C^d)$ be the Banach space of $d\times d$ matrix valued functions $g:\Z^{d+1}\ra\C^d\otimes\C^d$ with norm $\|g\|_p$ defined by
\be \label{A6}
\|g\|_p^p  =  \sup_{v\in\C^d:|v|=1}\sum_{(x,t)\in\Z^{d+1}} |g(x,t)v|^p \ {\rm if  \ }p<\infty, \qquad \|g\|_\infty= \sup_{v\in\C^d:|v|=1}\left[ \ \sup_{(x,t)\in\Z^{d+1}} |g(x,t)v| \ \right] \ ,
\ee
where $|g(x,t)v|$ is the Euclidean norm of the vector $g(x,t)v\in\C^d$. We assume the following:
\begin{hypothesis}
There exists $p_0(\La/\la)>1$ depending only on $d,\La/\la$  and a constant $C$ such that for $1\le p\le p_0(\La/\la)$, 
\be \label{B6}
\|P\sum_{(x,t)\in\Z^{d+1}} g(x,t)\mathbf{b}(\tau_{x,-t}\cdot)[v+\pa_\xi\Phi(\xi,\eta,\tau_{x,-t}\cdot)v]\| \ \le \  C\|g\|_p|v|
\ee
for all $\xi\in\R^d, \ \eta\in\C$ with $0<\Re\eta<\La$, and $g\in L^p(\Z^{d+1},\C^d\otimes\C^d), \ v\in\C^d$. 
\end{hypothesis}
\begin{remark}
Note from Lemma 2.3 of \cite{c2} that since $\|\pa_\xi\Phi(\xi,\eta,\cdot)v\|^2\le \La|v|^2/\la$ for $\xi\in\R^d, \ \Re\eta>0,$  the inequality (\ref{B6}) holds for $p=1$.  Hence if (\ref{B6}) holds for $p=p_0(\La/\la)$, by the Riesz convexity theorem \cite{sw} it also holds for any $p$ satisfying $1\le p\le p_0(\La/\la)$. 
\end{remark}
We show that  if Hypothesis 3.1 holds then the function $q(\xi,\eta)$ defined by (\ref{C2}) is H\"{o}lder continuous with exponent depending on $d, \La/\la$. 
\begin{lem}
Assume Hypothesis 3.1 holds. Then there exists $\al>0$ depending only on $d,\La/\la$  and a constant $C_\al$ such that the $d\times d $ matrix function $q(\xi,\eta)$ of (\ref{C2}) satisfies the inequality
\be \label{C6}
\|q(\xi',\eta')-q(\xi,\eta)\| \ \le \ C_\al\La[ \ |\xi'-\xi|^\al+|(\eta'-\eta)/\La|^{\al/2} \ ]
\ee
for all $\xi',\xi\in\R^d, \ 0<\Re\eta',\Re\eta< \La$. 
\end{lem}
\begin{proof}
It follows from (\ref{W2}) that
 \be \label{D6}
 h_k(\xi',\eta')-h_k(\xi,\eta) \ = \  \sum_{j=1}^{k} \langle \ \mathbf{b}(\cdot)  \  [PT_{\xi',\eta'}\mathbf{b}(\cdot)]^{j-1} \ P[T_{\xi',\eta'}-T_{\xi,\eta}]\mathbf{b}(\cdot) \  [PT_{\xi,\eta}\mathbf{b}(\cdot)]^{k-j} \ \rangle \ .
 \ee
 Hence we conclude from  (\ref{J2}), (\ref{X2}) and (\ref{D6})  upon using the inequality $\|T_{\xi',\eta'}\|\le 1$ that  for $v\in\C^d$,
\be \label{E6}
\|[q(\xi',\eta')-q(\xi,\eta)]v\| \ \le \  (\La^2/\la)\| P[T_{\xi',\eta'}-T_{\xi,\eta}]\mathbf{b}(\cdot) \  [v+\pa_\xi\Phi(\xi,\eta,\cdot)v]\| \ .
\ee
From (\ref{G2}) we see that the RHS of (\ref{E6}) is the same as the LHS of (\ref{B6}) with the function $g(\cdot,\cdot)$ given by the formula
 \be \label{F6}
 g(x,t) \ = \ \La[\na\na^*G_\La(x,t-1)]^*[ \ e^{-\eta' t-ix\cdot\xi'}-e^{-\eta t-ix\cdot\xi} \ ] \ .
 \ee
 It is easy to see that for $0<\al\le 1,$ the function $g(\cdot,\cdot)$ is in $L^p(\Z^{d+1},\C^d\otimes\C^d)$ with $p>(d+2)/(d+2-\al),$ and with $\|g(\cdot,\cdot)\|_p$ satisfying the inequality
 \be \label{G6}
 \|g(\cdot,\cdot)\|_p \ \le \ C_p\La^{1-1/p}[ \ |\xi'-\xi|^\al+|(\eta'-\eta)/\La|^{\al/2} \ ] \  ,
 \ee
 where the constant $C_p$ depends only on $d,p$. The H\"{o}lder continuity (\ref{C6}) for sufficiently small $\al>0$ follows from  (\ref{E6}) and (\ref{G6}).
 \end{proof}
 \begin{proof}[Proof of Theorem 1.2-discrete time case] We follow the proof of Theorem 1.1 using the H\"{o}lder continuity of the function $q(\cdot,\cdot)$.
 \end{proof}
 
 For the continuous time case we prove Theorem 1.2  assuming a hypothesis analogous to Hypothesis 3.1. For $1\le p\le\infty$ let $L^p(\Z^d\times\R,\C^d\otimes\C^d)$ be the Banach space of $d\times d$ matrix valued functions $g:\Z^d\times\R\ra\C^d\otimes\C^d$ with norm $\|g\|_p$ defined by
\be \label{H6}
\|g\|_p^p  =  \sup_{v\in\C^d:|v|=1}\sum_{x\in\Z^d}\int_{-\infty}^\infty dt \  |g(x,t)v|^p \ {\rm if  \ }p<\infty, \qquad \|g\|_\infty= \sup_{v\in\C^d:|v|=1}\left[ \ \sup_{(x,t)\in\Z^d\times\R} |g(x,t)v| \ \right] \ ,
\ee
where $|g(x,t)v|$ is the Euclidean norm of the vector $g(x,t)v\in\C^d$. 
 \begin{hypothesis}
 There exists $p_0(\La/\la)>1$ depending only on $d,\La/\la$  and a constant $C$ such that for $1\le p\le p_0(\La/\la)$, 
\be \label{I6}
\|P\sum_{x\in\Z^d}\int_{-\infty}^\infty dt \  g(x,t)\mathbf{b}(\tau_{x,-t}\cdot)[v+\pa_\xi\Phi(\xi,\eta,\tau_{x,-t}\cdot)v]\| \ \le \  C\|g\|_p|v|
\ee
for all $\xi\in\R^d, \ \eta\in\C$ with $0<\Re\eta<\La$, and $g\in L^p(\Z^d\times\R,\C^d\otimes\C^d), \ v\in\C^d$. 
\end{hypothesis}
It is easy to see that Hypothesis 3.2 implies the H\"{o}lder continuity of the matrix function $q(\cdot,\cdot)$ defined by (\ref{C2}), (\ref{AL2}).  We conclude that Theorem 1.2 holds for the continuous time case also. 
 
\vspace{.2in}

\section{Fluctuations of averaged Green's functions}
In this section we shall prove Theorem 1.3 under the assumption that the solutions $\Phi(\xi,\eta,\om)$ of (\ref{B2}), (\ref{AL2}) satisfy stronger versions of Hypothesis 3.1 and 3.2 of $\S3$. Thus in the discrete time case our hypothesis is: 
\begin{hypothesis}
Let $T_{\xi,\eta}$ be the operator (\ref{G2}) on the Hilbert space $\mathcal{H}(\Om)$ and  $T^*_{\xi,\eta}$denote its adjoint.  Then for $ k\ge 1, \ p_2=p_3=\cdots=p_k=1,$  and $S_{\xi,\eta}=T_{\xi,\eta}$ or $S_{\xi,\eta}=T^*_{\xi,\eta}$, there exists $p_0(\La/\la)>1$ depending only on $d,\La/\la$  and a constant $C(k)$ such that
\begin{multline}  \label{A*3}
\left\|\sum_{(x_1,t_1),...(x_k,t_k)\in\Z^{d+1}} \left\{\prod_{j=1}^k g_j(x_j,t_j)\tau_{x_j,-t_j}P\mathbf{b}(\cdot)[I-PS_{\xi,\eta}\mathbf{b}(\cdot)]^{-1}\right\}v\right\| \\
 \le \  C(k)\prod_{j=1}^k \|g_j\|_{p_j}|v| \quad  {\rm for \ }g_j\in L^{p_j}(\Z^{d+1},\C^d\otimes\C^d), \ j=1,..,k, \ v\in\C^d,
\end{multline}
provided $1\le p_1\le p_0(\La/\la)$ and  $\xi\in\C^d, \ \eta\in\C$ satisfy  $0<\Re\eta<\La,  \ |\Im\xi|\le C_1\sqrt{\Re\eta/\La}$, with $C_1$ depending only on $d,\La/\la$.  
\end{hypothesis}
\begin{remark}
Note that from (\ref{J2}) and Lemma 2.1 we see that the inequality (\ref{A*3}) holds for $p_1=1$. Hence if (\ref{A*3}) holds for $p_1=p_0(\La/\la)$, by the Riesz convexity theorem \cite{sw} it also holds for any $p_1$ satisfying $1\le p_1\le p_0(\La/\la)$. 
\end{remark}
We  define spaces  $L^p([-\pi,\pi]^{d+1}\times\Om,\C^d\otimes\C^d)$ of $d\times d$ matrix valued functions $g:[-\pi,\pi]^{d+1}\times\Om\ra \C^d\otimes\C^d$ with norm $\|g\|_p$ defined by
\begin{multline} \label{B*3}
\|g\|_p^p=  \sup_{v\in\C^d:|v|=1}\frac{1}{(2\pi)^{d+1}} \int_{[-\pi,\pi]^{d+1}} \langle \ |g(\xi,\Im\eta,\cdot)v|^2 \ \rangle ^{p/2}  \ d[\Im\eta] \ d\xi\quad {\rm if \ }p<\infty, \\
\quad \|g\|_\infty= \sup_{v\in\C^d:|v|=1}\left[ \ \sup_{(\xi,\Im\eta)\in[-\pi,\pi]^{d+1}, \ }  \langle \ |g(\xi,\Im\eta,\cdot)v|^2 \ \rangle ^{1/2} \ \right]  \ .
\end{multline}
We consider $\xi\in\C^d,\eta\in\C$ with $\xi$ having fixed imaginary part,  $\eta$ having fixed positive real part, and satisfying the conditions of Hypothesis 4.1. 
For $k\ge 1$ we define a multilinear operator $T_{k,\Im\xi,\Re\eta}$  from a sequence $[g_1,g_2,..,g_k]$ of $k$ functions $g_j:\Z^{d+1}\ra\C^d\otimes \C^d, \ j=1,..,k$, to periodic functions $T_{k,\Im\xi,\Re\eta}[g_1,g_2,..g_k]:[-\pi,\pi]^{d+1}\times\Om\ra \C^d\otimes\C^d$ by
\begin{multline}  \label{C*3}
T_{k,\Im\xi,\Re\eta}[g_1,g_2,...,g_k](\Re\xi,\Im\eta,\cdot) \ = \\
   \sum_{(x_1,t_1),..(x_k,t_k) \in\Z^{d+1}} \prod_{j=1}^k g_j(x_j,t_j)e^{-i(x_j.\Re\xi+t\Im\eta)}\tau_{x_j,-t_j}P\mathbf{b}(\cdot)[I-PT_{\xi,\eta}\mathbf{b}(\cdot)]^{-1} \ .
\end{multline}
We similarly  define multilinear operators $\tilde{T}_{k,\Im\xi,\Re\eta}$  by replacing $T_{\xi,\eta}$ in (\ref{C*3}) with $T^*_{\xi,\eta}$. 
For $p$ satisfying $1\le p\le \infty$  let $p'$ be the conjugate of $p$, so $1/p+1/p'=1$. In \cite{c2} the following  result was obtained:
\begin{lem}
Suppose $2\le q\le\infty$ and $p_1,...,p_k$ with $1\le p_1,...,p_k\le 2$ satisfy the identity 
\be \label{D*3}
\frac{1}{p_1'} +\frac{1}{p_2'}+\cdots +\frac{1}{p_k'} \ = \ \frac{1}{q} \ ,
\ee
and for $j=1,..,k$, the functions $g_j\in L^{p_j}(\Z^{d+1},\C^d\otimes\C^d)$. Then there exist positive constants $C_1,C_2$ depending only on $d,\La/\la$ such that if $0<\Re\eta<\La, \ |\Im\xi|<C_1\sqrt{\Re\eta/\La},$  the function $S_{k,\Im\xi,\Re\eta}[g_1,g_2,..g_k]=T_{k,\Im\xi,\Re\eta}[g_1,g_2,..g_k]$ or $S_{k,\Im\xi,\Re\eta}[g_1,g_2,..g_k]=\tilde{T}_{k,\Im\xi,\Re\eta}[g_1,g_2,..g_k]$  is in $L^q([-\pi,\pi]^{d+1}\times\Om,\C^d\otimes\C^d )$ and
\be \label{E*3}
\| \ S_{k,\Im\xi,\Re\eta}[g_1,g_2,..g_k] \ \|_q \  \le 
  C_2^k\prod_{j=1}^k \|g_j\|_{p_j} \ .
\ee
\end{lem}
 If we assume Hypothesis 4.1 we can improve Lemma  4.1 as  follows:
 \begin{lem}
 Suppose  Hypothesis 4.1 holds with  $p_0(\La/\la)\le 2$, and  $q, p_1,...,p_k$ with $2\le q\le\infty, \ 1\le p_1,...,p_k\le 2$ satisfy the inequality 
 \be \label{F*3}
\frac{1}{q} \ \le \  \frac{1}{p_1'} +\frac{1}{p_2'}+\cdots +\frac{1}{p_k'} \ \le \ \frac{1}{q} + \left[1-\frac{1}{p_0(\La/\la)}\right]\left[1-\frac{2}{q}\right] \ .
 \ee
 Then there exists a positive constant $C_1$ depending only on $d,\La/\la$ such that if $0<\Re\eta<\La, \ |\Im\xi|<C_1\sqrt{\Re\eta/\La},$  the function $S_{k,\Im\xi,\Re\eta}[g_1,g_2,..g_k]=T_{k,\Im\xi,\Re\eta}[g_1,g_2,..g_k]$ or $S_{k,\Im\xi,\Re\eta}[g_1,g_2,..g_k]=\tilde{T}_{k,\Im\xi,\Re\eta}[g_1,g_2,..g_k]$    is in $L^q([-\pi,\pi]^{d+1}\times\Om,\C^d\otimes\C^d )$ and
\be \label{G*3}
\| \ S_{k,\Im\xi,\Re\eta}[g_1,g_2,..g_k] \ \|_q \  \le 
  C(k)\prod_{j=1}^k \|g_j\|_{p_j} \ ,
\ee
for some constant $C(k)$. 
 \end{lem}
 \begin{proof}
 We assume first that $p_2=p_3=\cdots=p_k=1$, in which case  Hypothesis 4.1 and Lemma 4.1 imply respectively that (\ref{G*3}) holds for $1/p'_1\le 1-1/p_0(\La/\la), \ q=\infty,$ and for $p_1=q=2$. The Riesz convexity theorem then implies that (\ref{G*3}) holds if $p'_1,q$ satisfy (\ref{F*3}) with $p_2=p_3=\cdots=p_k=1$. Next assume for induction that we have proved (\ref{G*3}) in the case when (\ref{F*3}) holds with  $p_{r+1}=p_{r+2}=\cdots=p_k=1$ for some $r\ge 1$.  Hence (\ref{G*3}) holds for $1/p'_1+\cdots1/p'_r\le 1-1/p_0(\La/\la),  \ p_{r+1}=1, \ q=\infty$, where  
 the functions $g_{r+1},..,g_k$ are fixed with $p_{r+2}=p_{r+3}=\cdots=p_k=1$. From Lemma 4.1 we  see that  (\ref{G*3}) also holds for $1/p'_1+\cdots1/p'_{r+1}=1/2, \ q=2,$ with the same functions  $g_{r+1},..,g_k$. Now we fix the functions $g_1,..,g_r,g_{r+2},..,g_k$ with $p_{r+1}=p_{r+2}=\cdots=p_k=1$  and $1/p'_1+\cdots1/p'_r\le 1-1/p_0(\La/\la)\le 1/2$.  Applying the Riesz convexity theorem to the functions $g_{r+1}$, we conclude that  (\ref{G*3}) holds  if $p_1,..,p_{r+1}$ satisfies (\ref{F*3}) with $p_{r+2}=p_{r+3}=\cdots=p_k=1$. 
 \end{proof}
 For $1\le p<\infty$ let  $L^p_w([-\pi,\pi]^{d+1})$ be the space of functions $g:[-\pi,\pi]^{d+1}\ra\C$ which are weakly $p$ integrable. The norm $\|g\|_{p,w}$ of $g$ is defined to be the minimum number satisfying the inequality
 \be \label{H*3}
 (2\pi)^{-(d+1)}{\rm meas}\{(\xi,\Im\eta)\in[-\pi,\pi]^{d+1} \ : |g(\xi,\Im\eta)|>z \  \} \ \le \ \|g\|_{p,w}^p/z^p \quad {\rm for \ all \  } z>0.
 \ee
\begin{proposition}
Assume Hypothesis 4.1 holds, $4d\La\le 1$ and $m$ is a  positive integer.  Then there exist positive constants $C_1$ and  $\al\le1$ depending only on $d$ and $\La/\la$,  such that
\begin{multline} \label{A3}
\|q(\xi',\eta')-q(\xi,\eta)\|\le C\La \left[ \ |\xi'-\xi|^\alpha +|(\eta'-\eta)/\La|^{\al/2} \ \right]  \ , \\
{\rm for \ } 0<\Re\eta, \Re\eta'\le \La, \quad \xi',\xi\in\C^d \ {\rm with \ } |\Im\xi|+|\Im\xi'|\le C_1\sqrt{\Re\eta/\La} \ ,
\end{multline}
where $C$ is a constant. 

If  $\xi\in\C^d, \ \eta\in\C$ with fixed $\Im\xi\in\R^d, \ \Re\eta>0$ satisfying  $0<\Re\eta<\La, \ |\Im\xi|<C_1\sqrt{\Re\eta/\La}$, and $m<1+d/2$, the function 
\be \label{AB2}
(\Re\xi,\Im\eta)\ra\frac{\pa^m q_{r,r'}(\xi,\eta)}{\pa\eta^m}, \quad  (\Re\xi,\Im\eta)\in[-\pi,\pi]^{d+1} \ ,
\ee
is in the space $L^p_w([-\pi,\pi]^{d+1})$ with $p=(1+d/2)/(m-\al/2)$ and its norm is bounded by $C\La^{1-m+1/p} $  for some constant $C$.

If $m$ is the largest integer strictly less than $1+d/2$ and $0\le\del<1+d/2-m,$ then for any $\rho\in\R$ satisfying $|\rho|\le 1$, the function 
\be \label{AC2}
(\Re\xi,\Im\eta)\ra\frac{\pa^m}{\pa\eta^m}\left[ \ q_{r,r'}(\xi,\eta+i\rho)-q_{r,r'}(\xi,\eta) \ \right]/|\rho|^{\del}
\ee
is in the space $L^p_w([-\pi,\pi]^{d+1})$ with $p=(1+d/2)/(m+\del-\al/2)$ and its norm is bounded by $ C_p\La^{1- m-\del+1/p}$, where the constant $C_p$ can diverge as $p\ra 1$.
\end{proposition}
\begin{proof}
The H\"{o}lder continuity (\ref{A3}) of the function $q(\cdot,\cdot)$ has already been proven in Lemma 3.1. We first  prove that the derivative (\ref{AB2}) with $m=1$ is in  $L^p_w([-\pi,\pi]^{d+1})$ with $p=(1+d/2)/(1-\al/2)$ for some $\al>0$ depending only on $d,\La/\la$.  
Observe from  (\ref{W2}) and  (\ref{X2})   that 
\be \label{M*3}
\left( \frac{\pa}{\pa \eta} \right)q(\xi,\eta)=  -\La\langle \ \mathbf{b}(\cdot)[I-PT_{\xi,\eta}\mathbf{b}(\cdot)]^{-1}\left\{\frac{\pa}{\pa\eta} T_{\xi,\eta}\right\} P\mathbf{b}(\cdot)[I-PT_{\xi,\eta}\mathbf{b}(\cdot)]^{-1}  \ \rangle \ .
\ee
Denoting by $[\cdot,\cdot]$ the inner product for $\mathcal{H}(\Om)$,  we therefore have for $v_1,v_2\in\C^d$  that
\be \label{AI4}
\left( \frac{\pa}{\pa \eta} \right)v^*_1q(\xi,\eta)v_2=  -\La\left[ \  \tilde{T}_{1,\Im\xi,\Re\eta} \ g(\Re\xi,\Im\eta,\cdot)v_1, \  T_{1,\Im\xi,\Re\eta} \ h(\Re\xi,\Im\eta,\cdot)v_2 \ \right] \ 
\ee
for  certain $d\times d$ matrix valued functions $g(x,t), h(x,t), \ x\in\Z^d,t\in\Z$. The functions $g(\cdot,\cdot), h(\cdot,\cdot)$ are determined from their Fourier transforms (\ref{AE2})   by the  formula
\be \label{AJ4}
\hat{g}(\zeta,\theta)^*\hat{h}(\zeta,\theta)  \ = \  - \frac{\La e(i\Im\xi-\zeta)e(-i\Im\xi-\zeta)^*}{\left[e^{\Re\eta-i\theta}-1+\La e(-i\Im\xi-\zeta)^*e(i\Im\xi-\zeta)\right]^2} \ ,
\ee
which follows from (\ref{AI2}).
We take  $\hat{h}(\cdot,\cdot)$ to be given by the formula 
\be \label{AK4}
\hat{h}(\zeta,\theta) \ = \ \left(\frac{\La}{d}\right)^{1/2}\frac{\mathbf{1}_d \  e(-i\Im\xi-\zeta)^*}{\left[e^{\Re\eta-i\theta}-1+\La e(-i\Im\xi-\zeta)^*e(i\Im\xi-\zeta)\right]} \ ,
\ee
where $\mathbf{1}_d$ is the $d$ dimensional column vector with all entries equal to $1$.  From (\ref{H2}) and (\ref{AK4}) it follows that
\be \label{AL4}
h(x,t) \ = \ \left(\frac{\La}{d}\right)^{1/2}\mathbf{1}_d \ \{ \na G_\La(x,t-1)\}^* e^{x\cdot\Im\xi-t\Re\eta} \  \ {\rm if \ } t\ge 1, \quad h(x,t)=0 \ {\rm otherwise}.
\ee 
Assuming $0<\Re\eta<\La, \ |\Im\xi|\le C\sqrt{\Re\eta/\La},$ for sufficiently small positive constant $C$ depending only on $d$, it follows from (\ref{AL4}) that there is a constant $C_d$ depending only on $d$ such that $|h(x,t)|$ is bounded above by $\{\La/(\La t+1\}^{1/2}$   times the RHS of (\ref{I2}).  It follows that $h(\cdot,\cdot)$ is in $L^p_w(\Z^{d+1})$ with $p=(d+2)/(d+1)$  and $\|h\|_{p,w}\le C\La^{1/2-1/p}$ for a constant $C$ depending only on $d$. 

Observe now that by the Hunt interpolation theorem \cite{rs} the inequality (\ref{G*3}) also holds for the operator $T_{1,\Im\xi,\Re\eta}$ as a mapping from $L^{p_1}_w(\Z^{d+1})$ to $L_w^q([-\pi,\pi]^{d+1}\times\Om,\C^d\otimes\C^d)$. Hence  $T_{1,\Im\xi,\Re\eta}h$ is in $L_w^q([-\pi,\pi]^{d+1}\times\Om,\C^d\otimes\C^d)$ provided $q$ satisfies the inequality in (\ref{F*3}) with $p_1=(d+2)/(d+1)$. Evidently we can choose $q$ so that $q/2>1+d/2$.  Since we can make an exactly similar argument for the function $g(x,t)$ and  $\tilde{T}_{1,\Im\xi,\Re\eta}g$, we conclude from (\ref{AI4}) that $\pa q_{r,r'}(\xi,\eta)/\pa\eta$ is in the space $L^{q/2}_w([-\pi,\pi]^{d+1})$  with norm bounded by $\La^{2-2/p}$ times a constant. We have proved for $m=1$ that the derivative (\ref{AB2}) is in the appropriate weak $L^p$ space.   We proceed similarly to estimate the higher derivatives (\ref{AB2}) and the fractional derivative (\ref{AC2}). 
\end{proof}
\begin{remark}
Proposition 4.1 with $\al=0$ was proven in \cite{c2}. In that case the constant $C$  in the statement of the Proposition depends only on $d,\La/\la$. 
\end{remark}
Proposition 4.1 enables us to compare the averaged Green's function $G_\mathbf{a}(x,t), \ x\in\Z^d,t\in\Z^+$ for (\ref{B1}), (\ref{C1})  to the lattice Green's function $G^{\rm lattice}_{\mathbf{a}_{\rm hom}}(x,t), \ x\in\Z^d, \ t\in\Z^+$ defined by
\be \label{I*3}
G^{\rm lattice}_{\mathbf{a}_{\rm hom}}(x,t) \ = \ \frac{1}{(2\pi)^{d+1}}\int_{[-\pi,\pi]^{d+1}} 
\frac{e^{-i\xi.x+\eta(t+1)}}{e^\eta-1+e(\xi)^*q(0,0)e(\xi)} \ d[\Im\eta] \ d\xi \ .
\ee
\begin{theorem}
Assume Hypothesis 4.1 holds and $4d\La\le 1$.  Then there exist positive constants $\al,\ga,$ with  $\al\le 1$, depending only on $d,\La/\la$ and a constant $C$ such that for $x\in\Z^d, \ t\in\Z,t\ge0,$ 
\be \label{C3}
|G_{{\bf a}}(x,t)-G^{\rm lattice}_{{\bf a}_{\rm hom}}(x,t)|  \le  \frac{C}{[\La t+1]^{(d+\alpha)/2}} \exp\left[-\ga\min\left\{|x|, \ \frac{ |x|^2}{\La t+1}\right\}\right] \ ,
 \ee
\be \label{D3}
 |\na G_{{\bf a}}(x,t)-\na G^{\rm lattice}_{{\bf a}_{\rm hom}}(x,t)|  \le \frac{C}{[\La t+1]^{(d+1+\alpha)/2}} \exp\left[-\ga\min\left\{|x|, \ \frac{ |x|^2}{\La t+1}\right\}\right] \ .
 \ee
If $\del$ satisfies $0<\del\le 1$ then there exists $\al,\ga>0$ depending only on $d,\La/\la,\del$  and a constant $C_\del$ such that the following inequality holds:
\begin{multline} \label{E3}
\big| \ [\na G_{{\bf a}}(x',t)-\na G^{\rm lattice}_{{\bf a}_{\rm hom}}(x',t)]-[\na G_{{\bf a}}(x,t)-\na G^{\rm lattice}_{{\bf a}_{\rm hom}}(x,t)] \ \big|  \\
\le \ |x'-x|^{1-\del} \frac{C_\del}{[\La t+1]^{(d+2-\del+\alpha)/2}} \exp\left[-\ga\min\left\{|x|, \ \frac{ |x|^2}{\La t+1}\right\}\right] \ 
, \quad   \ x',x\in\Z^d, \ 
1/2\le (|x'|+1)/(|x|+1)\le 2. 
\end{multline}
The constant $\al$ in (\ref{E3}) must satisfy  $\al<\del$. 
\end{theorem} 
\begin{proof} 
From (\ref{G*2}), (\ref{I*3}) and Corollary 2.1  there is a constant $C$ depending only on $\La/\la$ such that for $a\in\R^d$ with $|a|\le 1,$ 
\be \label{F3}
G_{{\bf a}}(x,t)-G^{\rm lattice}_{{\bf a}_{\rm hom}}(x,t) \ = \ \frac{\exp[a.x/C+\La|a|^2(t+1)]}{(2\pi)^{d+1}}\int_{[-\pi,\pi]^{d+1}} 
e^{-i\xi.x+i\Im\eta(t+1)} f_a(\xi,\Im\eta) \ d\xi \ d[\Im\eta] \ , 
\ee
where the function $f_a(\xi,\Im\eta)$ is given by the formula 
\begin{multline} \label{G3}
f_a(\xi,\Im\eta) \ =   \ \frac{e(\xi-ia/C)^*\{q(0,0)-q(\xi+ia/C,\La|a|^2+i\Im\eta)\}e(\xi+ia/C)}
 {\left[\exp[\La|a|^2+i\Im\eta]-1+e(\xi-ia/C)^*q(0,0)e(\xi+ia/C)\right] } \\
\times \ \frac{1}{\left[\exp[\La|a|^2+i\Im\eta]-1+e(\xi-ia/C)^*q(\xi+ia/C,\La|a|^2+i\Im\eta)e(\xi+ia/C)\right]}  \ .
\end{multline}
The exponential decay in the inequalities (\ref{C3})-(\ref{E3}) is obtained by choosing $a$ in (\ref{F3}) to be given by
\be \label{H3}
a=- x/(C+1)(\La t+1) \ {\rm if \ } |x|\le \La t+1, \quad a= -x/(C+1)|x| \ {\rm if \ } |x|\ge \La t+1 \ .
\ee
It follows from (\ref{AD2}), Proposition 4.1 and Corollary 2.1  that there is a positive constant $C_1$  such that  the function in (\ref{G3}) is bounded by 
\be \label{M3}
|f_a(\xi,\Im\eta)| \ \le \ \frac{C_1[|e(\xi)|^2+|a|^2]}{\La[|\Im\eta|/\La+|e(\xi)|^2+|a|^2]^{2-\al/2}} \ , \quad {\rm for \ }  (\xi,\Im\eta)\in[-\pi,\pi]^{d+1}.
\ee

To complete the proof of the theorem we need to obtain the polynomial decay in $[\La t+1]$ in (\ref{C3})-(\ref{E3}), whence we may assume that $\La t\ge 1$. We divide the torus $[-\pi,\pi]^{d+1}$ into various regions, the first of which is
\be \label{N3}
E_{0,0} \ = \ \{ (\xi,\Im\eta)\in[-\pi,\pi]^{d+1} \ : \  \La t|e(\xi)|^2\le 1, \ |\Im\eta|\le 1/t  \ \} \ .
\ee
It follows then from (\ref{M3})  that there is a constant $C_2$ such that
\be \label{O3}
\int_{E_{0,0}} |f_a(\xi,\Im\eta)|  \ d\xi  \ d[\Im\eta] \ \le \  C_2/[\La t+1]^{(d+\al)/2} \ .
\ee

Next we consider for $k=1,2,..,$ regions
\be \label{P3}
E_{0,k} \ = \ \{ (\xi,\Im\eta)\in[-\pi,\pi]^{d+1} \ : \  \La t|e(\xi)|^2\le 1, \ \ 2^{k-1}/t< |\Im\eta|\le 2^k/t \ \} \ .
\ee
From (\ref{M3}) we see that if $|a|\le 2/\La t$  there is a constant $C_3$ such that
\be \label{Q3}
\left| \ \int_{E_{0,k}} 
e^{-i\xi.x+i\Im\eta(t+1)} f_a(\xi,\Im\eta) \ d\xi \ d[\Im\eta]  \ \right| \  \le \  C_32^{-k(1-\al/2)}/[\La t+1]^{(d+\al)/2} \ .
\ee
In general $a=O(1)$, so we need to take advantage of the oscillatory nature of the integral in (\ref{Q3}). Let $\rho=\pi/(t+1)$ so that $e^{i\rho(t+1)}=-1$, and $E^\rho_{0,k}=\{ \ (\xi,\Im\eta):(\xi,\Im\eta+\rho)\in E_{0,k} \ \}$.  Then the LHS of (\ref{Q3}) is bounded by
\begin{multline} \label{R3}
\frac{1}{2}\int_{E_{0,k}\cap E^\rho_{0,k}} |f_a(\xi,\Im\eta)-f_a(\xi,\Im\eta+\rho)| \ d\xi \ d[\Im\eta]+\\
\frac{1}{2}\int_{E_{0,k}-E^\rho_{0,k}} |f_a(\xi,\Im\eta)| \ d\xi \ d[\Im\eta]+
\frac{1}{2}\int_{E^\rho_{0,k}-E_{0,k}} |f_a(\xi,\Im\eta+\rho)| \ d\xi \ d[\Im\eta] \ .
\end{multline}
It follows again from (\ref{M3}) that the last two integrals on the RHS of (\ref{R3}) are bounded by the RHS of (\ref{Q3}). In order to bound the first integral we observe from the H\"{o}lder continuity (\ref{A3}) of the function $q(\cdot,\cdot)$ that there are constants $C_4,C_5$ and
\begin{multline} \label{S3}
 |f_a(\xi,\Im\eta)-f_a(\xi,\Im\eta+\rho)|  \ \le \ \frac{C_4[|e(\xi)|^2+|a|^2](\rho/\La)^{\al/2}}{\La[|\Im\eta|/\La+|e(\xi)|^2+|a|^2]^2} \\
 + \frac{C_5[|e(\xi)|^2+|a|^2](\rho/\La)}{\La[|\Im\eta|/\La+|e(\xi)|^2+|a|^2]^{3-\al/2}}, \quad {\rm for \ }  (\xi,\Im\eta)\in E_{0,k}\cap E^\rho_{0.k} \ .
\end{multline}
Since we are assuming $|a|\ge 2/\La t$ it follows from (\ref{S3}) that
\be \label{T3}
\sum_{k\ge 1}\int_{E_{0,k}\cap E^\rho_{0,k}} |f_a(\xi,\Im\eta)-f_a(\xi,\Im\eta+\rho)| \ d\xi \ d[\Im\eta] \ \le \ 
C_6/[\La t+1]^{(d+\al)/2} \ 
\ee
for some constant  $C_6$. We  therefore conclude from (\ref{O3})-(\ref{T3})  that there is a constant $C_7$ and
\be \label{U3}
\sum_{k\ge 0} \left| \ \int_{E_{0,k}} 
e^{-i\xi.x+i\Im\eta(t+1)} f_a(\xi,\Im\eta) \ d\xi \ d[\Im\eta]  \ \right| \  \le \  C_7/[\La t+1]^{(d+\al)/2} \ .
\ee

The inequality (\ref{U3}) can also be derived by using the fact from Theorem 3.1 that the derivative $\pa q(\xi+ia/C,\La|a|^2+i\Im\eta)/\pa [\Im\eta]$ is in the space $L^p_w([-\pi,\pi]^{d+1})$  with $p=(1+d/2)/(1-\al/2)$. Thus we observe that
\begin{multline} \label{V3}
\left| \ \int_{E_{0,k}} 
e^{-i\xi.x+i\Im\eta(t+1)} f_a(\xi,\Im\eta) \ d\xi \ d[\Im\eta] \ \right|  \ \le \  \frac{1}{t+1} \  \int_{ \pa E_{0,k}} | f_a(\xi,\Im\eta)| \ d\xi   \\
+\frac{1}{t+1}  \int_{E_{0,k}} 
\left| \ \frac{\pa f_a(\xi,\Im\eta)}{\pa[\Im\eta]} \ \right| \ d\xi \ d[\Im\eta]   \ ,
\end{multline}
where $\pa E_{0,k}$ is the union of sets $\{(\xi,\Im\eta): \La t|e(\xi)|^2\le 1, \ \Im\eta={\rm constant}\}$ with the constant given by $\pm 2^k/t$ or $\pm2^{k-1}/t$. It follows from (\ref{M3}) that the first integral on the RHS of (\ref{V3}) is bounded by the RHS of (\ref{Q3}). To bound the second integral we use the inequality
\begin{multline} \label{W3}
\left| \ \frac{\pa f_a(\xi,\Im\eta)}{\pa[\Im\eta]} \ \right|  \ \le \ 
 \frac{C_8[|e(\xi)|^2+|a|^2]}{\La^2[|\Im\eta|/\La+|e(\xi)|^2+|a|^2]^{3-\al/2}} \\
+ \frac{C_9[|e(\xi)|^2+|a|^2]}{\La^2[|\Im\eta|/\La+|e(\xi)|^2+|a|^2]^{2}}\left|\frac{\pa q(\xi+ia/C,\La|a|^2+i\Im\eta)}{\pa[\Im\eta]}\right| \ ,
\end{multline}
where $C_8,C_9$ are constants.  We can bound the integral of the first term on the RHS of (\ref{W3}) just as we did with the second term on the RHS of (\ref{S3}). To bound the integral of the second term we use the well known fact that if $f\in L^p_w([-\pi,\pi]^{d+1})$ with $1<p<\infty$, then for any measurable set $F$, one has
\be \label{X3}
\int_F |f| \ \le \ C_p\|f\|_{p,w} m(F)^{1-1/p} \ ,
\ee
where the constant $C_p$ depends only on $p$.  Taking $p=(1+d/2)/(1-\al/2)$ we conclude from Proposition 4.1 that $1/(t+1)$ times the integral over $E_{0,k}$ of the second term on the RHS of (\ref{W3}) is bounded by 
\be \label{Y3}
\frac{C_{10}[1/\La t+|a|^2]}{\La t[2^k/\La t+|a|^2]^{2}} \frac{2^{k(1-1/p)}}{[\La t+1]^{(d+\al)/2}}
\ee
for some constant $C_{10}$. Summing (\ref{Y3}) over $k\ge 1$ we obtain the inequality (\ref{U3}) again. 

For $r\ge 1, \ k\ge 0,$ let $E_{r,k}$ be defined by
\be \label{Z3} 
E_{r,k} =  \{ (\xi,\Im\eta)\in[-\pi,\pi]^{d+1} \ : \  2^{r-1}<\La t|e(\xi)|^2\le 2^r, \ \  2^{k-1}/t<|\Im\eta|\le 2^k/t \ \} \ , \ k\ge1,
\ee
\be \nonumber
 E_{r,0} = \{ (\xi,\Im\eta)\in[-\pi,\pi]^{d+1} \ : \  2^{r-1}<\La t|e(\xi)|^2\le 2^r, \ \  |\Im\eta|\le 1/t \ \}  \ .
\ee
Then we have that
\begin{multline}  \label{AA3}
\sum_{k=0}^\infty \int_{E_{r,k}} 
e^{-i\xi.x+i\Im\eta(t+1)} f_a(\xi,\Im\eta) \ d\xi \ d[\Im\eta]   \ = \\
\frac{i^m}{(t+1)^m}\sum_{k=0}^\infty \int_{E_{r,k}} 
e^{-i\xi.x+i\Im\eta(t+1)} \frac{\pa^m f_a(\xi,\Im\eta)}{\pa[\Im\eta]^m} \ d\xi \ d[\Im\eta]    \ .
\end{multline}
Just as in (\ref{W3}) we see from Proposition 4.1  that
\be \label{AB3}
\frac{\pa^m f_a(\xi,\Im\eta)}{\pa[\Im\eta]^m} \ = \ \frac{|e(\xi)|^2+|a|^2}{\La^2[|\Im\eta|/\La+|e(\xi)|^2+|a|^2]^{2}} \ g_{a,m}(\xi,\Im\eta) \ ,
\ee
where for $m<1+d/2$ the function $g_{a,m}(\cdot,\cdot)$ is in  $L^p_w([-\pi,\pi]^{d+1})$ with $p=(1+d/2)/(m-\al/2)$. Thus there is a constant $C_{11}$ such that
\be \label{AC3}
\int_F |g_{a,m}(\xi,\Im\eta)| \ d\xi \ d[{\Im\eta}]  \ \le \ C_{11} \La^{1-m+1/p} m(F)^{1-1/p}\ , \quad F\subset [-\pi,\pi]^{d+1} \ .
\ee
It follows from (\ref{AB3}), (\ref{AC3}) that
\be \label{AD3}
\frac{1}{(t+1)^m}\int_{E_{r,k}} \left| \  \frac{\pa^m f_a(\xi,\Im\eta)}{\pa[\Im\eta]^m} \ \right| \ d\xi \ d[\Im\eta]    \ \le \ 
\frac{C_{12}}{[\La t+1]^{(d+\al)/2}} \frac{2^{(rd/2+k)(1-1/p)}}{2^r+2^k} \ 
\ee
for some constant  $C_{12}$.  Observe that
\be \label{AE3}
\sum_{k=0}^\infty\sum_{r=1}^\infty \frac{2^{(rd/2+k)(1-1/p)}}{2^r+2^k} \ < \ \infty \ 
\ee
provided $m$ satisfies the inequality $m>(d+\al)/2$. If $d$ is odd then there is an integer $m$ satisfying $d/2<m<1+d/2$, whence (\ref{U3}), (\ref{AA3}), and (\ref{AD3}) imply that (\ref{C3}) holds for some $\al>0$. 

In the case when $d$ is even we note from (\ref{AA3}) that
\begin{multline}  \label{AF3}
\sum_{k=0}^\infty \int_{E_{r,k}} 
e^{-i\xi.x+i\Im\eta(t+1)} f_a(\xi,\Im\eta) \ d\xi \ d[\Im\eta]   \ = \\
\frac{i^m}{2(t+1)^m}\sum_{k=0}^\infty \int_{E_{r,k}} 
e^{-i\xi.x+i\Im\eta(t+1)} \left[ \ \frac{\pa^m f_a(\xi,\Im\eta)}{\pa[\Im\eta]^m} - \frac{\pa^m f_a(\xi,\Im\eta+\rho)}{\pa[\Im\eta]^m}  \ \right]\ d\xi \ d[\Im\eta]    \ ,
\end{multline}
where $m$ is the largest integer satisfying $m<1+d/2$ and $\rho=\pi/(t+1)$.  Similarly to (\ref{AB3}) we have that
\be \label{AG3}
\frac{1}{|\rho|^{\del}}\left[\frac{\pa^m f_a(\xi,\Im\eta)}{\pa[\Im\eta]^m}-\frac{\pa^m f_a(\xi,\Im\eta+\rho)}{\pa[\Im\eta]^m}\right] \ = \ \frac{|e(\xi)|^2+|a|^2}{\La^2[|\Im\eta|/\La+|e(\xi)|^2+|a|^2]^{2}} \ g_{a,\del}(\xi,\Im\eta) \ ,
\ee
where for $0\le\del<1+d/2-m$ the function $g_{a,\del}(\cdot,\cdot)$ satisfies an inequality (\ref{X3}) with $p=(1+d/2)/(m+\del-\al/2)$. Hence as in (\ref{AD3}) we conclude that
\be \label{AH3}
\frac{1}{(t+1)^m}\int_{E_{r,k}} \left| \  \frac{\pa^m f_a(\xi,\Im\eta)}{\pa[\Im\eta]^m}-\frac{\pa^m f_a(\xi,\Im\eta+\rho)}{\pa[\Im\eta]^m} \ \right| \ d\xi \ d[\Im\eta]    \ \le \ 
\frac{C_{13}}{[\La t+1]^{(d+\al)/2}} \frac{2^{(rd/2+k)(1-1/p)}}{2^r+2^k} \ ,
\ee
where $C_{13}$ also depends on $\del$ as well as $\La,d,\La/\la$.  Now (\ref{C3}) for some $\al>0$  follows from (\ref{U3}), (\ref{AF3}), and (\ref{AH3}) by choosing $\del$ in (\ref{AH3}) so that $0<\del<1$. 

In order to prove (\ref{D3}) we follow the previous argument, replacing the function $f_a(\xi,\Im\eta)$ by the function $e(\xi)f_a(\xi,\Im\eta)$.  To prove (\ref{E3}) we use the inequality
\be \label{AI3}
|e^{i\xi\cdot(x-x')}-1|  \le \ 10|x-x'|^{1-\del}|e(\xi)|^{1-\del}  \ ,
 \ee
 and replace the function  $f_a(\xi,\Im\eta)$ by the function $|e(\xi)|^{2-\del}f_a(\xi,\Im\eta)$ in the argument to prove (\ref{C3}).  
\end{proof}
\begin{remark}
In the case when $\al=0$  the constant $C$ in (\ref{C3}), (\ref{D3}) depends only on $d,\La/\la$. For $\al>0$ the constant $C$ also depends on the constant in the inequality (\ref{A*3}) of Hypothesis 4.1. 
\end{remark}
The inequalities (\ref{S1}), (\ref{T1}) of Theorem 1.3 are a  consequence now of Theorem 4.1 and the following result which compares the lattice Green's function  $G^{\rm lattice}_{\mathbf{a}_{\rm hom}}(x,t)$ to the Green's function $G_{\mathbf{a}_{\rm hom}}(x,t)$ for the PDE (\ref{F1}): 
\begin{lem}
Assuming $4d\La\le 1$,  then there exist positive constants $\ga,C$  depending only on $d,\La/\la$  such that for $x\in\Z^d, \ t\in\Z$ with $\La t\ge 1$,  
\be \label{J*3}
|G_{\mathbf{a}_{\rm hom}}(x,t)-G^{\rm lattice}_{{\bf a}_{\rm hom}}(x,t)|  \le  \frac{C}{[\La t+1]^{(d+1)/2}} \exp\left[-\ga\min\left\{|x|, \ \frac{ |x|^2}{\La t+1}\right\}\right] \ ,
 \ee
\be \label{K*3}
 |\na G_{\mathbf{a}_{\rm hom}}(x,t)-\na G^{\rm lattice}_{{\bf a}_{\rm hom}}(x,t)|  \le \frac{C}{[\La t+1]^{(d+2)/2}} \exp\left[-\ga\min\left\{|x|, \ \frac{ |x|^2}{\La t+1}\right\}\right] \ ,
 \ee
 \be \label{L*3}
 |\na\na G_{\mathbf{a}_{\rm hom}}(x,t)-\na\na G^{\rm lattice}_{{\bf a}_{\rm hom}}(x,t)|  \le \frac{C}{[\La t+1]^{(d+3)/2}} \exp\left[-\ga\min\left\{|x|, \ \frac{ |x|^2}{\La t+1}\right\}\right] \ .
 \ee
\end{lem}
\begin{proof}
Taking $\mathbf{a}_{\rm hom}=q(0,0)$ in (\ref{F1}), we see from (\ref{I*3})  that  $G^{\rm lattice}_{{\bf a}_{\rm hom}}(\cdot,\cdot)$ is the Green's function for the discrete parabolic equation  corresponding to (\ref{F1}),
\be \label{N*3}
u(x,t+1,\om)-u(x,t,\om) \ = \ -\nabla^*\mathbf{a}_{\rm hom}\nabla u(x,t,\om) \ , \quad x\in \Z^d, \ t=0,1,2....
\ee
 To prove the theorem we follow a standard method of numerical analysis for estimating error between the solution of a continuous problem and its approximating discrete problems. The method is to regard the solution of the continuous problem as an approximate solution to the discrete problem. An alternative approach based on comparison of the Fourier representation  (\ref{I*3})  of the lattice  Green's function $G^{\rm lattice}_{\mathbf{a}_{\rm hom}}(\cdot,\cdot)$ to the Fourier representation of the continuous Green's function  $ G_{\mathbf{a}_{\rm hom}}(\cdot,\cdot)$  is pursued in \cite{mr} for the case of elliptic equations. 
 
 Let $f:\R^d\ra\R$ be a nonnegative $C^\infty$ function with support contained in the ball $\{x\in\R^d: |x|<1\}$ and $u(x,t)=u_{\rm hom}(x,t)$ be the solution to the initial value problem (\ref{F1}), (\ref{G1}).  With $\na_x,\na_x^*$ denoting the discrete operators (\ref{E1}), we have  that
\begin{multline} \label{O*3}
u(x+z,t+1)-u(x+z,t)+\na_x^* {\bf a}_{\rm hom}\na_x u(x+z,t) \ = \\
 u(x+z,t+1)-u(x+z,t)+{\rm Trace}[ {\bf a}_{\rm hom}A(x+z,t)]  \ ,  \quad x\in \Z^d, \ z\in\R^d, \ \ t=0,1,..,
\end{multline}
where the $d\times d$ matrix $A(y,t)=[A_{i,j}(y,t)], \ y\in\R^d,t>0$ is given by the formula
\be \label{P*3}
A_{i,j}(y,t) \ = \ u(y,t)+u(y+\mathbf{e}_j-\mathbf{e}_i,t)-u(y+\mathbf{e}_j,t)-u(y-\mathbf{e}_i,t) 
= \  -E\left[ \ \frac{\pa^2 u(y+Y_{i,j},t)}{\pa  y_i\pa y_j}\right] \ ,
\ee
with $Y_{i,j}$ the random variable uniformly distributed in the unit square $\{y_j\mathbf{e}_j-y_i\mathbf{e}_i\in\R^d: 0\le y_i,y_j\le 1\}$. It follows then from (\ref{O*3}), (\ref{P*3}) that
\be \label{Q*3}
u(x+z,t+1)-u(x+z,t)+\na_x^* {\bf a}_{\rm hom}\na_x u(x+z,t) \ = \ h_1(x+z,t)-h_2(x+z,t) \ ,
\ee
where the functions $h_j(\cdot,\cdot), \ j=1,2$ are given by the formulas
\be \label{R*3}
 h_1(y,t) \ = \ E\left[ \ \frac{\pa u(y,t+T)}{\pa  t}\right]-\frac{\pa u(y,t)}{\pa t} \ , \quad y\in\R^d,t>0 ,
\ee
\be \label{S*3}
h_2(y,t)  \  = \   \sum_{i,j=1}^d
  {\bf a}_{\rm hom}(i,j)\left\{ E\left[ \ \frac{\pa^2 u(y+Y_{i,j},t)}{\pa  y_i\pa y_j}\right] -\frac{\pa^2 u(y,t)}{\pa y_i\pa y_j}\right\} \ , \quad y\in\R^d,t>0 .
\ee 
In (\ref{R*3}) the random variable $T$ is uniformly distributed in the interval $0<T<1$. 
 Since $u(x+z,0)=f(x+z), \ x\in\Z^d,$ we conclude from (\ref{Q*3}) that
\begin{multline} \label{T*3}
u(x+z,t) \ = \ \sum_{y\in\Z^d} G^{\rm lattice}_{{\bf a}_{\rm hom}}(x-y,t)f(y+z) +
\sum_{r=1}^t\sum_{y\in\Z^d} G^{\rm lattice}_{{\bf a}_{\rm hom}}(x-y,t-r)h_1(y+z,r-1) \\
-\sum_{r=1}^t\sum_{y\in\Z^d} G^{\rm lattice}_{{\bf a}_{\rm hom}}(x-y,t-r)h_2(y+z,r-1) \  .
\end{multline}
Let $Q_0\subset\R^d$ be the unit cube centered at the origin. Then we have that
\begin{multline} \label{U*3}
\int_{Q_0} dz \ \left[u(x+z,t) - \sum_{y\in\Z^d} G^{\rm lattice}_{{\bf a}_{\rm hom}}(x-y,t)f(y+z)\right] \ = \\
\left[G_{{\bf a}_{\rm hom}}(x,t)-G^{\rm lattice}_{{\bf a}_{\rm hom}}(x,t)\right]\int_{\R^d} f(y) \ dy+
{\rm Error}(x) \ ,
\end{multline}
where $|{\rm Error}(x)|$ is bounded by the RHS of (\ref{J*3}). 

Next  observe from (\ref{F1}), (\ref{R*3}), (\ref{S*3}) that
\be \label{V*3}
\int_{Q_0} dz \ \sum_{y\in\Z^d} h_j(y+z,t) \ = \ 0 \quad {\rm for \ } j=1,2.
\ee
It follows from (\ref{V*3})  that if we integrate the third term on the RHS of (\ref{T*3}) with respect to $z\in Q_0$ it is equal to
\be \label{W*3}
\int_{Q_0} dz \ \sum_{r=1}^t\sum_{y\in\Z^d} [G^{\rm lattice}_{{\bf a}_{\rm hom}}(x-y,t-r)-G^{\rm lattice}_{{\bf a}_{\rm hom}}(x,t-r)]h_2(y+z,r-1)  \ .
\ee
Using the fact that the distribution of $Y_{j,i}$ is the same as the distribution of $-Y_{i,j}$ we see from (\ref{S*3}) that $h_2(\cdot,t)$ is bounded by the fourth derivative of $u(\cdot,t)$, whence we conclude that there are constants $\ga,C$  depending only on $d$ such that
\be \label{X*3}
|h_2(y,t)| \ \le \ \frac{C\La\|f\|_\infty}{[\La t+1]^{(d+4)/2}} \exp\left[- \frac{\ga |y|^2}{\La t+1}\right] \ , \quad y\in\R^d,t>0.
\ee
We also have that there are constants $\ga,C$  depending only on $d$ such that
\be \label{Y*3}
 |\na G^{\rm lattice}_{{\bf a}_{\rm hom}}(y,t)|  \le \frac{C}{[\La t+1]^{(d+1)/2}} \exp\left[-\ga\min\left\{|y|, \ \frac{ |y|^2}{\La t+1}\right\}\right] \ , \quad y\in\Z^d, t\in\Z^+.
 \ee
Using (\ref{X*3}), (\ref{Y*3}) we can estimate (\ref{W*3}) and see that it is bounded by the RHS of (\ref{J*3}). Since we can do a similar estimate with the function $h_2$ replaced by $h_1$ we conclude from (\ref{U*3}) that (\ref{J*3}) holds. We can obtain the bounds (\ref{K*3}), (\ref{L*3}) by taking the gradient of (\ref{T*3}) with respect to $x$ and following the previous argument. 
\end{proof}
The inequality (\ref{U1}) of Theorem 1.3 is a consequence of Lemma 4.3 and the following:
\begin{theorem}
Assume Hypothesis 4.1 holds and $4d\La\le 1$.  Then there exist positive constants $\al,\ga,$ with  $\al\le 1$, depending only on $d,\La/\la$ and a constant $C$ such that for $x\in\Z^d, \ t\in\Z,t\ge0,$ 
\be \label{Z*3}
|\na\na G_{{\bf a}}(x,t)-\na\na G^{\rm lattice}_{{\bf a}_{\rm hom}}(x,t)|  \le  \frac{C}{[\La t+1]^{(d+2+\alpha)/2}} \exp\left[-\ga\min\left\{|x|, \ \frac{ |x|^2}{\La t+1}\right\}\right] \ .
 \ee
\end{theorem}
\begin{proof}
Let $\chi:\R^{d+1}\ra\R$ be a $C^\infty$ function with compact support such that the integral of $\chi(\cdot)$ over $\R^{d+1}$ equals $1$.  We write
\be \label{AA*3}
G_{{\bf a}}(x,t) \ = \ \chi_L*G_{{\bf a}}(x,t) +[G_{{\bf a}}(x,t)-\chi_L*G_{{\bf a}}(x,t)] \ ,
\ee
where $\chi_L(x,t)= \La L^{-(d+2)}\chi(x/L,\La t/L^2), \ x\in\R^d,t\in\R$, and $*$ denotes convolution on $\Z^{d+1}$.  Let $\hat{\chi}_L(\zeta,\theta), \ \zeta\in[-\pi,\pi]^d, \ \theta\in[-\pi,\pi]$   be the Fourier transform (\ref{AE2}) of $\chi_L(\cdot,\cdot)$ restricted to the $\Z^{d+1}$ lattice. Since   $\chi_L(\cdot,\cdot)$  has compact support $\hat{\chi}_L(\cdot,\cdot)$ has an analytic continuation to $\C^{d+1}$. Furthermore for $L\ge 1$ there is a  constant $C$ such that
\be \label{AB*3}
|\hat{\chi}_L(0,0)-1|\le C/L, \quad \left| \hat{\chi}_L(\zeta+ia,\theta-i\La|a|^2)\right|
\le  \ C\exp[C|a|^2L^2] \quad a\in\R^d.
\ee
There also exists for positive integers $n$ constants $C_n$ such that
\be \label{AC*3}
\left| \hat{\chi}_L(\zeta+ia,\theta-i\La|a|^2)\right|
\le \frac{C_n}{[1+L|\zeta|+L^2|\theta|/\La]^n} \quad {\rm if \ } |a|L\le 1.
\ee
We assume now that $R<\sqrt{\La t+1}<2R$ and choose $L=R^{1-\del}$ for some $\del>0$.  Then
from (\ref{D2}) we see that
\be \label{AD*3}
\chi_L*\na_k\na_jG_{{\bf a}}(x,t) \ = 
 \frac{1}{(2\pi)^{d+1}}\exp\left[a\cdot x/C+\La|a|^2(t+1)\right]  
 \int_{[-\pi,\pi]^{d}}\int_{-\pi}^\pi 
f_a(\zeta,\theta) \ d\theta \ d\zeta \ ,
\ee
where $a$ is given by (\ref{H3}) and $f_a(\zeta,\theta)$ is defined by
\be \label{AE*3}
f_a(\zeta,\theta) \ = \ \frac{e_k(\zeta+ia/C)e_j(\zeta+ia/C)\hat{\chi}_L(\zeta+ia/C,\theta-i\La|a|^2)
e^{-i\zeta.x+i\theta(t+1)}}{e^{\La|a|^2+i\theta}-1+e(\zeta-ia/C)^*q(\zeta+ia/C,\La|a|^2+i\theta)e(\zeta+ia/C)}  \ .
\ee
It follows from Corollary 2.2 and  the second inequality of (\ref{AB*3})  that if $|a|L\ge 1$  there is a constant $C_1$ such that 
\begin{multline} \label{AF*3}
\exp\left[a\cdot x/C+\La|a|^2(t+1)\right]  
 \int_{[-\pi,\pi]^{d}}\int_{-\pi}^\pi 
|f_a(\zeta,\theta)| \ d\theta \ d\zeta  \ \le \\
\frac{C_1}{[\La t+1]^{(d+3)/2}} \exp\left[-\ga\min\left\{|x|, \ \frac{ |x|^2}{\La t+1}\right\}\right] \ .
\end{multline}
If $|a|L<1$ we also have from Corollary 2.2 and (\ref{AC*3}) that
\begin{multline} \label{AG*3}
\int_{[-\pi,\pi]^{d}\cap\{|\zeta|>1/R^{1-2\del}\}}\int_{-\pi}^\pi 
|f_a(\zeta,\theta)| \ d\theta \ d\zeta+ \\
\int_{[-\pi,\pi]^{d}}\int_{[-\pi,\pi]\cap\{\sqrt{|\theta|/\La}>1/R^{1-2\del}\}}
|f_a(\zeta,\theta)| \ d\theta \ d\zeta  \ \le  \  C_1/[\La t+1]^{(d+3)/2}
\end{multline}
for some constant $C_1$.

To estimate the integral of $f_a(\zeta,\theta)$ over the set $\{(\zeta,\theta): |\zeta|<1/R^{1-2\del}, \ 
\sqrt{|\theta|/\La}<1/R^{1-2\del}\}$ we use the H\"{o}lder continuity (\ref{A3}) of the function $q(\cdot,\cdot)$.  Thus let $g_a(\cdot,\cdot)$ be defined similarly to the function $f_a(\cdot,\cdot)$ by
\be \label{AH*3}
g_a(\zeta,\theta) \ = \ \frac{e_k(\zeta+ia/C)e_j(\zeta+ia/C)\hat{\chi}_L(\zeta+ia/C,\theta-i\La|a|^2)
e^{-i\zeta.x+i\theta(t+1)}}{e^{\La|a|^2+i\theta}-1+e(\zeta-ia/C)^*q(0,0)e(\zeta+ia/C)}  \ .
\ee
Then from (\ref{A3}) we see that
\be \label{AI*3}
\int_{\{|\zeta|<1/R^{1-2\del}\}}\int_{\sqrt{|\theta|/\La}<1/R^{1-2\del}}
|f_a(\zeta,\theta)-g_a(\zeta,\theta)| \ d\theta \ d\zeta\le C_2/[\La t+1]^{d+2+\al)(1-2\del)}
\ee
for some constant $C_2$.  We choose now $\del>0$ in (\ref{AI*3}) sufficiently small so that $(d+2+\al)(1-2\del)>d+2$.  It follows then from (\ref{AF*3}), (\ref{AG*3}), (\ref{AI*3}) that 
$|\chi_L*\na_k\na_jG_{{\bf a}}(x,t)-\chi_L*\na_k\na_jG^{\rm lattice}_{{\mathbf{a}_{\rm hom}}}(x,t)|$ is bounded by the RHS of (\ref{Z*3}). 

To complete the proof of the inequality (\ref{Z*3}) we use the H\"{o}lder continuity result of \cite{dd}.
Thus from the first inequality of (\ref{AB*3}) and \cite{dd} we see that $|\na_k\na_jG_{{\bf a}}(x,t)-\chi_L*\na_k\na_jG_{{\bf a}}(x,t)]|$ is bounded by the RHS of (\ref{Z*3}) for some $\al>0$.  The result follows. 
\end{proof} 
We can essentially repeat the foregoing arguments for the continuous time averaged Green's function  $G_{{\bf a}}(x,t), \ x\in\Z^d, t\ge 0,$ for (\ref{D1}).  In the continuius time case our hypothesis is: 
\begin{hypothesis}
Let $T_{\xi,\eta}$ be the operator (\ref{AN2}) on the Hilbert space $\mathcal{H}(\Om)$ and  $T^*_{\xi,\eta}$denote its adjoint.  Then for $ k\ge 1, \ p_2=p_3=\cdots=p_k=1,$  and $S_{\xi,\eta}=T_{\xi,\eta}$ or $S_{\xi,\eta}=T^*_{\xi,\eta}$, there exists $p_0(\La/\la)>1$ depending only on $d,\La/\la$  and a constant $C(k)$ such that
\begin{multline}  \label{AJ*3}
\left\|\sum_{x_1,...x_k\in\Z^{d+1}}\int_{\R^k} dt_1\cdots dt_k \left\{\prod_{j=1}^k g_j(x_j,t_j)\tau_{x_j,-t_j}P\mathbf{b}(\cdot)[I-PS_{\xi,\eta}\mathbf{b}(\cdot)]^{-1}\right\}v\right\| \\
 \le \  C(k)\prod_{j=1}^k \|g_j\|_{p_j}|v| \quad  {\rm for \ }g_j\in L^{p_j}(\Z^{d}\times\R,\C^d\otimes\C^d), \ j=1,..,k, \ v\in\C^d,
\end{multline}
provided $1\le p_1\le p_0(\La/\la)$ and  $\xi\in\C^d, \ \eta\in\C$ satisfy  $0<\Re\eta<\La,  \ |\Im\xi|\le C_1\sqrt{\Re\eta/\La}$, with $C_1$ depending only on $d,\La/\la$.  
\end{hypothesis}
Assuming Hypothesis 4.2 holds, we can prove the analogues of Proposition 4.1, Theorem 4.1 and Theorem 4.2 for the continuous case. Theorem 1.3 therefore follows in the continuous time case once we are able to establish Hypothesis 4.2.

\vspace{.1in}

\section{Independent Variable Environment}
Our goal in this section will be to prove Hypothesis 3.1 and its generalized form Hypothesis 4.1 in the case when the variables ${\bf a}(\tau_{x,t}\cdot), \ x\in\Z^d, t\in\Z$, are independent.  Following \cite{cn1} we first consider the case of a Bernoulli environment.  Thus for each $x\in\Z^{d}, t\in\Z,$ let $Y_{x,t}$ be independent Bernoulli variables, whence $Y_{x,t}=\pm 1$ with equal probability. The probability space $(\Om,\mathcal{F},P)$ is then the space generated by all the variables $Y_{x,t}, \ (x,t)\in\Z^{d+1}$.  A point $\om\in\Om$ is a set of configurations $\{(Y_n,n) \ : n\in\Z^{d+1}\}$. For $(x,t)\in\Z^{d+1}$ the translation operator $\tau_{x,t}$ acts on $\Om$ by taking the point $\om=\{(Y_n,n): n\in\Z^{d+1}\}$ to $\tau_{x,t}\om=\{(Y_{n+(x,-t)},n): n\in\Z^{d+1}\}$. The random matrix ${\bf a}(\cdot)$ is then defined by 
\be \label{A4}
{\bf a}(\om) \ = \ (1+\ga Y_0) I_d, \quad \om=\{(Y_n,n): n\in\Z^{d+1}\} \ ,
\ee
where $0\le \ga<1$.

In \cite{cn1} we defined for $1\le p<\infty$ Fock spaces $\mathcal{F}^p(\Z^{d+1})$  of complex valued functions, and observed that $\mathcal{F}^2(\Z^{d+1})$ is unitarily equivalent to $L^2(\Om)$. We can similarly define Fock spaces $\mathcal{H}_\mathcal{F}^p(\Z^{d+1})$ of vector valued functions with range $\C^{d}$, such that $\mathcal{H}_\mathcal{F}^2(\Z^{d+1})$  is unitarily equivalent to  $\mathcal{H}(\Om)$. Hence we can regard the operator $T_{\xi,\eta}$ of (\ref{G2}) as acting on $\mathcal{H}_\mathcal{F}^2(\Z^{d+1})$, and by unitary equivalence it is a bounded operator satisfying $\|T_{\xi,\eta}\|\le 1$ for $\xi\in\R^d, \ \Re\eta>0$.  From (\ref{G2}) we see that $T_{\xi,\eta}$ acts as a convolution operator on $N$ particle wave functions $\psi_N(\cdot)$ in $\mathcal{H}_\mathcal{F}^2(\Z^{d+1})$ as
\begin{multline} \label{B4}
T_{\xi,\eta} \psi_N(x_1,t_1,...,x_N,t_N) \ = \\
   \La\sum_{t'=1}^\infty e^{-\eta t'}\sum_{x\in \Z^d} \left\{\nabla\nabla^* G_{\La}(x',t'-1)\right\}^*\exp[-ix'\cdot\xi]  \ \psi_N(x_1-x',t_1-t',..,x_N-x',t_N-t') \ .
\end{multline}
Note that for all  $N$ particle wave functions,  $T_{\xi,\eta}$ acts as a convolution operator on functions on $\Z^{d+1}$. Hence its action is determined by its action on $1$ particle wave functions. Let  $\hat{\psi}_1(\zeta,\theta), \ \zeta\in[-\pi,\pi]^d, \  \theta\in[-\pi,\pi],$ be the Fourier transform (\ref{AE2})  of the $1$ particle wave function $\psi_1(x,t), \ x\in\Z^d,t\in\Z$.
 We see from (\ref{B4}) that for $\xi\in\C^d, \ \Re\eta>0,$  the action of $T_{\xi,\eta}$ in Fourier space is given by
\be \label{D4}
\hat{T}_{\xi,\eta} \hat{\psi}_1(\zeta,\theta) \ = \ \frac{\La e(\xi-\zeta)e(\bar{\xi}-\zeta)^*}{e^{\eta-i\theta}-1+\La e(\bar{\xi}-\zeta)^*e(\xi-\zeta)} \ \hat{\psi}_1(\zeta,\theta)  \  , \quad \zeta\in[-\pi,\pi]^d, \  \theta\in[-\pi,\pi].
\ee
Hence the result of Lemma 2.1 for the Bernoulli case follows from:
\begin{lem} Assume $4d\La\le 1$. Then there exist positive constants $C_1,C_2$ depending only on $d$ such that for $(\xi,\eta)$ in the region 
 $\{(\xi,\eta)\in\C^{d+1}: 0<\Re\eta<\La, \ |\Im\xi|<C_1\sqrt{\Re\eta/\La}\}$, there is the inequality 
 \be \label{E4}
 \La \max[ \  |e(\xi)|^2, \ |e(\bar{\xi})|^2] \ \le \ \left(1+C_2 |\Im\xi|^2/[
\Re\eta/\La]\right)|e^{\eta}-1+\La e(\bar{\xi})^*e(\xi)|  \  .
 \ee
\end{lem} 
\begin{proof}
We have that
\begin{multline} \label{F4}
|e^{\eta}-1+\La e(\bar{\xi})^*e(\xi)|  \ \ge \ e^{\Re\eta}-|1-\La e(\bar{\xi})^*e(\xi)| \\
\ge \  e^{\Re\eta}-1+\La e(\Re\xi)^*e(\Re\xi) -\La|e(\Re\xi)^*e(\Re\xi)-e(\bar{\xi})^*e(\xi)| \ ,
\end{multline}
where we have used the fact that $4d\La\le 1$. Observe that there is a constant $C$ depending only on $d$ such that
\begin{multline} \label{G4}
\La|e(\Re\xi)^*e(\Re\xi)-e(\bar{\xi})^*e(\xi)|  \ \le \ C\La[|\Im\xi|^2+|\Im\xi||e(\Re\xi)|] \\
\le \ C^2\{\La e(\Re\xi)^*e(\Re\xi)\} |\Im\xi|^2/[\Re\eta/\La] +[1/4+CC_1^2]\Re\eta \ .
\end{multline}
We conclude then from (\ref{F4}), (\ref{G4}) that
\be \label{H4}
|e^{\eta}-1+\La e(\bar{\xi})^*e(\xi)|  \ \ge \ [3/4-CC_1^2]\Re\eta+[1-C^2|\Im\xi|^2/[\Re\eta/\La] \La e(\Re\xi)^*e(\Re\xi) \  .
\ee
The inequality (\ref{E4}) follows from (\ref{H4}) by observing similarly to (\ref{G4}) that
\be \label{I4}
\La |e(\xi)|^2 \ \le \  \La e(\Re\xi)^*e(\Re\xi)+C\La[|\Im\xi|^2+|\Im\xi||e(\Re\xi)|]  \ .
\ee 
\end{proof}
\begin{lem} Assume $4d\La\le 1$. Then there exist positive constants $C_1,C_2,C_3$ depending only on $d$ such that for $(\xi,\eta)$ in the region  $\{(\xi,\eta)\in\C^{d+1}: 0<\Re\eta<\La, \ |\Im\xi|<C_1\sqrt{\Re\eta/\La}\}$,  the operator $T_{\xi,\eta}$ of (\ref{B4}) is bounded on $\mathcal{H}_\mathcal{F}^p(\Z^{d+1})$ for $p=3/2$ or $p=3$, and the norm  $\|T_{\xi,\eta}\|_p$ of $T_{\xi,\eta}$ satisfies the inequality
$\|T_{\xi,\eta}\|_p\le C_3\left(1+C_2 |\Im\xi|^2/[\Re\eta/\La]\right)$.
\end{lem}
\begin{proof}
It will be sufficient for us to prove the theorem on the space of $1$ particle wave functions.  To do this we follow the argument of Jones \cite{j}, which adapts the methodology of Calderon-Zygmund \cite{cz} to Fourier multipliers associated with parabolic PDE.  A more general theory of Fourier multipliers can be found in Chapter IV of \cite{stein}, but because of the generality it is hard to estimate the values of constants using this theory.

For a set $E\subset\Z^{d+1},$ we denote by $|E|$  the number of lattice points of $\Z^{d+1}$ contained in $E$. Let $\psi(x,t), \ x\in\Z^d,t\in\Z,$ be a $1$ particle wave function with finite support. We shall show that for any $\ga>0$, the set $E_\ga=\{(x,t)\in\Z^{d+1} \ : \ |T_{\xi,\eta}\psi(x,t)|>\ga\}$ satisfies the inequality
\be \label{J4}
|E_\ga| \ \le \ C_4\left(1+C_2 |\Im\xi|^2/[\Re\eta/\La\right)\ga^{-2}\sum_{(x,t)\in\Z^{d+1}} \min[|\psi(x,t)|, \ \ga]^2+C_5 \beta^\psi(\ga) \ ,
\ee
where $C_2$ is the constant of Lemma 5.1 and $C_4,C_5$ depend only on $d$. The function $\beta^\psi(\cdot)$ is defined in \cite{cz,j} in terms of the distribution function of $\psi(\cdot,\cdot)$. Once (\ref{J4}) is proved the result follows from the argument of \cite{cz}, which shows that $\|T_{\xi,\eta}\|_p$ is simply bounded in terms of the constants occurring in (\ref{J4}). 

We use a Calderon-Zygmund decomposition to prove (\ref{J4}). Recalling that $1/\La\ge 4d,$ let $N_0\ge 2$ be the integer which satisfies  $2^{N_0}\le 1/\La <2^{N_0+1}$.  We choose $a_1,..a_d,b\in\Z$ and sufficiently large integer $N_1$ such that the rectangle $R=\{(x,t)=(x_1,.,x_d,t)\in\R^{d+1} \ : \ a_j+1/2\le x_j\le 2^{N_1}+a_j+1/2,  \ j=1,..,d, \ {\rm and \ }\ b+1/2\le t\le 2^{2N_1+N_0}+b+1/2 \ \}$ contains the support of $\psi(\cdot,\cdot)$  and
\be \label{K4}
\frac{1}{|R|}\sum_{(x,t)\in R\cap\Z^{d+1}} |\psi(x,t)| \ \le \ \ga \ .
\ee
Note that the length of the side of $R$ in the $t$ direction is $2^{N_0}$ times the square of the length of a side in an $x_j$ direction for all $1\le j\le d$.  We subdivide $R$ into $2^d\times 4$ sub-rectangles with the same property and continue to similarly subdivide until we reach  a set of disjoint rectangles $R_m, \ m=1,..,M_1,$ with side in the $x_j, \ 1\le j\le d,$ direction a non-negative power of $2$, which satisfy the inequality
\be \label{L4}
\ga \ < \ \frac{1}{|R_m|}\sum_{(x,t)\in R_m} |\psi(x,t)| \ \le \ 2^{d+2}\ga \ , \quad 1\le m\le M_1,
\ee
together with a set of  rectangles $R'_m, \ m=1,2,...M_2,$ with side in the $x_j, \ 1\le j\le d,$ direction equal to $1$ and equal to $2^{N_0}$ in the $t$ direction which satisfy
\be \label{M4}
\frac{1}{|R'_m|}\sum_{(x,t)\in R'_m} |\psi(x,t)| \ \le \ \ga \ .
\ee 
We  subdivide the rectangles  $R'_m, \ m=1,..,M_2,$ into $2$ rectangles with side in the $t$ direction of length $2^{N_0-1}$, and continue to subdivide until we reach  a set of disjoint rectangles $R_m, \ m=M_1+1,..,M,$   with side in the $t$ direction a non-negative power of $2$, which satisfy the inequality
\be \label{N4}
\ga \ < \ \frac{1}{|R_m|}\sum_{(x,t)\in R_m} |\psi(x,t)| \ \le \ 2\ga \ , \quad M_1+1\le m\le M,
\ee
together with a set of unit cubes centered at lattice points of $\Z^{d+1}$. 
Setting $D_\ga=\cup_{m=1}^M R_m$, one sees that $\R^{d+1}-D_\ga$ is a union of unit cubes centered at lattice points of $\Z^{d+1}$, whence
\be \label{O4}
|\psi(x,t)| \ \le \ \ga \quad {\rm for \ }(x,t)\in \Z^{d+1}\cap[\R^{d+1}-D_\ga] \ .
\ee

We consider the distribution function $\ga\ra |\{(x,t)\in\Z^{d+1} \ : \ |\psi(x,t)|>\ga\}|$ of $\psi(\cdot,\cdot)$ with domain $\{\ga\ge 0\}$, which is a piece-wise constant right  continuous decreasing  function with range $0\le s\le |{\rm supp}[\psi(\cdot,\cdot)]|$. The decreasing rearrangement  $\psi^*(s)$ of $\psi(\cdot,\cdot)$ with domain $s\ge 0$ is also a piece-wise constant right  continuous decreasing  function satisfying  $\psi^*(0)=\sup|\psi(\cdot,\cdot)|$ and $\psi^*(s)=0$ for $s\ge |{\rm supp}[\psi(\cdot,\cdot)]|$. It is the approximate right continuous inverse of the distribution function for $0\le s\le  |{\rm supp}[\psi(\cdot,\cdot)]|$ . In view of (\ref{L4}), (\ref{N4}) we have that
\be \label{P4}
\ga \ < \ \frac{1}{|D_\ga|}\sum_{(x,t)\in D_\ga} |\psi(x,t)| \  \le \ \frac{1}{|D_\ga|}\int_0^{|D_\ga|} \psi^*(s) \ ds  \ = \ \beta_\psi(|D_\ga|)\ ,
\ee
where the function $\beta_\psi(s)$ with domain $s\ge 0$ is decreasing and continuous with range $0<\ga\le   \sup|\psi(\cdot,\cdot)|$. There is a well-defined inverse function $\beta^\psi(\ga)$ for $\beta_\psi(\cdot)$ with domain $0<\ga<\sup|\psi(\cdot,\cdot)|$, and (\ref{P4}) implies that $|D_\ga|\le \beta^\psi(\ga)$. 

We write $\psi(\cdot,\cdot)=\psi_1(\cdot,\cdot)+\psi_2(\cdot,\cdot)$, where the function  $\psi_1(\cdot,\cdot)$ is defined by
\begin{multline} \label{Q4}
\psi_1(x,t) \ =  \ \frac{1}{|R_m|}\sum_{(x',t')\in R_m} \psi(x',t') \ {\rm if \ } (x,t)\in R_m  \ {\rm for \ some \ } m, \ 1\le m\le M,  \\
\psi_1(x,t) \ = \ \psi(x,t) \quad {\rm otherwise} \  .  
\end{multline} 
From Lemma 5.1 and (\ref{O4}) we have then that
\begin{multline}  \label{R4}
|\{(x,t)\in\Z^{d+1} \ : \ |T_{\xi,\eta}\psi_1(x,t)|>\ga/2 \ \}| \ \le  \\
\left(1+C_2 |\Im\xi|^2/[\Re\eta/\La\right)^2\left\{ 4\ga^{-2}\sum_{(x,t)\in\Z^{d+1}} \min[|\psi(x,t)|, \ \ga]^2 +2^{2d+6}|D_\ga| \ \right\} \ .
\end{multline}
To bound the distribution function of $\psi_2(\cdot,\cdot)$ which has support contained in $D_\ga$, we  consider a rectangle $R_m, \ 1\le m\le M,$ with center $(x^m,t^m)\in\Z^{d+1}$ and let $\tilde{R}_m$ be the double of $R_m$. We observe that  similarly to (\ref{I2}) there is a constant $C_d$ depending only on $d$ such that the function $\na\na^*G_\La(x,t)$ satisfies inequalities 
\begin{multline} \label{S4}
|e^{-\Re\eta(t+1)-ix\cdot\Im\xi}\na\na^*G_\La(x,t+1)-e^{-\Re\eta t-ix\cdot\Im\xi}\na\na^*G_\La(x,t)| \\
 \le \frac{C_d}{(t+1)[\La t+1]^{d/2+1}}\exp\left[-\frac{\min\left\{|x|, \ |x|^2/(\La t+1)\right\}}{C_d}\right] \ ,
 \end{multline}
 \begin{multline} \label{T4}
|e^{-\Re\eta t-i(x+\mathbf{e}_j)\cdot\Im\xi}\na\na^*G_\La(x+\mathbf{e}_j,t)-e^{-\Re\eta t-ix\cdot\Im\xi}\na\na^*G_\La(x,t)| \\
 \le \frac{C_d}{[\La t+1]^{d/2+3/2}}\exp\left[-\frac{\min\left\{|x|, \ |x|^2/(\La t+1)\right\}}{C_d}\right] \ ,  \ j=1,.,d,
\end{multline}
provided $\xi\in\C^d, \ \eta\in\C,$ satisfy the conditions in the statement of the lemma. Extending the function $G_\La(x,t), \ x\in\Z^d, \ t=0,1,2,..,$ defined by (\ref{H2})  to have domain $\Z^{d+1}$ by setting $G_\La(x,t)=0$ for $x\in\Z^d, \ t<0$, we conclude from (\ref{S4}), (\ref{T4}) that if $(x',t')\in R_m,$  then there is a constant $C_d$ depending only on $d$ such that
\begin{multline}  \label{U4}
\sum_{(x,t)\in\Z^{d+1}-\tilde{R}_m}\La \  |e^{-\Re\eta(t-t')-i(x-x')\cdot\Im\xi}\na\na^*G_\La(x-x', t-t')\\
-e^{-\Re\eta(t-t^m)-i(x-x^m)\cdot\Im\xi}\na\na^*G_\La(x-x^m,t-t^m)| \ \le \ C_d.
\end{multline}
It follows from (\ref{L4}), (\ref{N4}), (\ref{U4}) that if $\tilde{D}_\ga=\cup_{m=1}^M\tilde{R}_m,$   then
\be \label{V4}
\sum_{(x,t)\in\Z^{d+1}-\tilde{D}_\ga} |T_{\xi,\eta}\psi_2(x,t)| \ \le \   C_d\ga |D_\ga|
\ee
for some constant $C_d$ depending only on $d$. Hence we have that
\be \label{W4}
|\{(x,t)\in\Z^{d+1} \ : \ |T_{\xi,\eta}\psi_2(x,t)|>\ga/2 \ \}| \ \le \  2C_d|D_\ga|+|\tilde{D}_\ga| \ \le \ [2C_d+2^{d+2}]|D_\ga| \ .
\ee
The inequality (\ref{J4}) follows from (\ref{R4}) and (\ref{W4}). 
\end{proof}
\begin{corollary}
Under the assumptions of Lemma 5.2 the operator $T_{\xi,\eta}$ is bounded on $\mathcal{H}^p(\Z^{d+1})$ for $3/2\le p\le 3,$ and  $\|T_{\xi,\eta}\|_p\le [1+\del(p)]\left(1+C_2 |\Im\xi|^2/[\Re\eta/\La]\right)$, where the function $\del(\cdot)$ depends only on $d$ and $\lim_{p\ra 2}\del(p)=0$. 
\end{corollary}
\begin{proof}
The result follows from Lemma 5.1, Lemma 5.2 and the Riesz-Thorin interpolation theorem \cite{sw}.
\end{proof}
 \begin{proof}[Proof of Hypothesis 4.1] 
 We choose $q_0=q_0(\La/\la)$ with $1<q_0<2$ so that $\del(q_0)\le\la/2\La$, where  $\del(\cdot)$ is the function in the statement of Corollary  5.1. It follows then from Young's inequality that Hypothesis 4.1 holds if we choose $p_0=p_0(\La/\la)>1$ with $1/p_0+1/q_0=3/2$.
  It is shown in \cite{cn1} how to extend the argument for the Bernoulli environment corresponding to (\ref{A4}) to general i.i.d. environments  ${\bf a}(\tau_{x,t}\cdot), \ (x,t)\in\Z^{d+1}$. We have therefore proven Hypothesis 4.1 for ${\bf a}(\tau_{x,t}\cdot), \ (x,t)\in\Z^{d+1}$, i.i.d. such that  (\ref{A1}) holds.
 \end{proof}
 
 \vspace{.1in}
 
  \section{Massive Field Theory Environment}
 In this section we show that Hypothesis 3.2 and its generalization Hypothesis 4.2 holds if $(\Om,\mathcal{F},P)$ is given by the massive field theory environment determined by (\ref{J1}), (\ref{K1}). We recall the main features of the construction of this measure. Let $L$ be a positive even integer and $Q=Q_L\subset \Z^d$ be the integer lattice points in the cube centered at the origin with side of length $L$. By a periodic function $\phi:Q\times\R\ra\R$ we mean a function $\phi$ on $Q\times\R$ with the property that $\phi(x,t)=\phi(y,t)$ for all $x,y\in Q, \ t\in\R,$ such that $x-y=L{\bf e}_k$ for some $k, \ 1\le k\le d$. Let $\Om_Q$ be the space of continuous in time periodic functions $\phi:Q\times\R\ra\R$ and $\mathcal{F}_Q$ be the Borel algebra generated by the requirement that the functions $\phi(\cdot,\cdot)\ra\phi(x,t)$ from $\Om_Q\ra\R$ are Borel measurable for all $x\in Q$ and $t$ rational. 
 For $m>0$ we define a probability measure $P_{Q}$ on $(\Om_Q,\mathcal{F}_Q)$ by first defining expectations of functions of the variables $\phi(x,0), \ x\in Q,$ as follows:
\begin{multline} \label{A5}
<F(\phi(\cdot,0))>_{\Om_Q}= \\
 \int_{\R^{L^d}} F(\phi(\cdot))\exp\left[-\sum_{x\in Q} \left\{V(\nabla \phi(x))+\frac{1}{2}m^2\phi(x)^2\right\}\right]  \prod_{x\in Q} d\phi(x)/{\rm normalization} \ ,
 \end{multline}
 where $F:\R^{L^d}\ra\R$ is a continuous function such that $|F(z)|\le C\exp[A|z|], \ z\in\R^{L^d}$, for some constants $C,A$.  
By translation invariance of the measure (\ref{A5}) we see that  $\langle\phi(x,0)\rangle_{\Om_Q}= 0$ for all $x\in Q$ and hence the  Brascamp-Lieb inequality  \cite{bl}  applied to (\ref{A5}) and function $F(\phi(\cdot))=\exp[(f,\phi)]$, where $(\cdot,\cdot)$ is the Euclidean inner product for periodic functions on $Q$,   yields the inequality 
 \be \label{B5}
 \langle  \exp[(f,\phi)]\rangle_{\Om_Q} \ \  \le  \ \  \exp\left[\frac{1}{2} (f, \{-\la\Del+m^2\}^{-1} f)\right] \ .
 \ee
 
 The variables $\phi(x,t), \ x\in Q,t>0,$ are determined from the variables $\phi(x,0), \ x\in Q,$ by solving the stochastic differential  equation
\be \label{C5}
d\phi(x,t) \ = \ -\frac{\pa}{\pa\phi(x,t)}\sum_{x'\in Q} \frac{1}{2}\{V(\na\phi(x',t))+m^2\phi(x',t)^2/2\} \ dt +dB(x,t) \ , \quad x\in Q, t>0, 
\ee
 where $B(x,\cdot), \ x\in Q,$ are independent copies of Brownian motion modulo the periodicity constraint on $Q$.  Since (\ref{A5}) is the invariant measure for the stochastic process $\phi(\cdot,t), \ t\ge 0$, it follows that (\ref{A5}), (\ref{C5}) determine a stationary process for $t\ge 0$, which therefore can be extended to all $t\in\R$.  Furthermore the functions $t\ra\phi(x,t)$ on $\R$ are continuous with probability $1$ for all $x\in Q$.  The probability measure $P_Q$ on $(\Om_Q,\mathcal{F}_Q)$ is the measure induced by the stationary process $\phi(\cdot,t), \ t\in\R$. 
 
 The probability space $(\Om,\mathcal{F},P)$ on  continuous in time fields $\phi:\Z^d\times\R\ra\R$ is obtained as the limit of the spaces $(\Om_Q,\mathcal{F}_Q,P_{Q})$ as $|Q|\ra\infty$. In particular one has from Lemma 2.4 of  \cite{c1} the following result:
\begin{proposition} Assume $m>0$ and let $F:\R^{k}\ra\R$ be a $C^1$ function which satisfies the inequality
\be \label{D5}
            |DF(z)|\le A\exp[ \ B|z| \ ], \quad z\in\R^{k},
 \ee
 for some constants $A,B$. Then for any $x_1,....x_k\in\Z^d,$ and $t_1,..,t_k\in\R,$ the limit
 \be \label{E5}
 \lim_{|Q|\ra\infty} \langle F\left(\phi(x_1,t_1),\phi(x_2,t_2),.....,\phi(x_k,t_k)\right)\rangle_{\Om_Q}= \langle F\left(\phi(x_1,t_1),\phi(x_2,t_2),.....,\phi(x_k,t_k)\right)\rangle
\ee
exists and is finite.
\end{proposition}
From (\ref{B5})  and the Helly-Bray theorem \cite{br,d} one sees that Proposition 6.1 implies the existence of a unique Borel probability measure on $\R^k$ corresponding to the probability distribution of the variables $(\phi(x_1,t_1),..,\phi(x_k,t_k))\in\R^k$, and this measure satisfies (\ref{E5}). The Kolmogorov construction \cite{br,d} then implies the existence of a Borel measure on fields $\phi:\Z^d\times\R\ra\R$ with finite dimensional distribution functions satisfying (\ref{E5}). We have constructed the  probability space $(\Om,\mathcal{F},P)$ corresponding to  (\ref{J1}),(\ref{K1}) for which $\Om$ is the set of continuous in time functions   $\phi:\Z^d\times\R\ra\R$, and it is clear that the translation operators $\tau_{x,t}, \ x\in\Z^d,t\in\R,$ are measure preserving and form a group.

The BL inequality \cite{bl} plays a crucial role in establishing the existence of the limit (\ref{E5})  in \cite{c1,fs}. In particular it yields a Poincar\'{e} inequality for the measure (\ref{A5}). Thus if  $F:\R^{L^d}\ra\R$ is a $C^1$ function such that $|DF(z)|\le C\exp[A|z|], \ z\in\R^{L^d}$, for some constants $C,A$, then
\be \label{F5}
{\rm var}_{\Om_Q}[F(\phi(\cdot,0))]=\langle \ [F(\phi(\cdot,0))-\langle F(\phi(\cdot,0))\rangle]^2\rangle_{\Om_Q} \ \le  \ \frac{1}{m^2} \langle \|dF(\phi(\cdot,0))\|^2\rangle_{\Om_Q} \ ,
\ee
where $dF(\phi(\cdot,0))\in \R^{L^d}$ is the gradient of $F$ at $\phi(\cdot,0)$.  A simple proof of (\ref{F5}) follows from the Helffer-Sj\"{o}strand (HS) representation \cite{hs}
\be \label{G5}
\langle F_1(\phi(\cdot,0))F_2(\phi(\cdot,0))\rangle_{\Om_Q} \ =  \ \langle dF_1(\phi(\cdot,0))[d^*d+\nabla^*V''(\nabla\phi(\cdot))\nabla+m^2]^{-1}dF_2(\phi(\cdot,0))\rangle_{\Om_Q} \ ,
\ee
which holds for $C^1$ functions $F_1,F_2:\R^{L^d}\ra\R$ that satisfy  $|F_j(z)|+|DF_j(z)|\le C\exp[A|z|], \ z\in\R^{L^d}, \ j=1,2$, for some constants $C,A$,  and $\langle F_1(\phi(\cdot,0))\rangle_{\Om_Q}=0$.  In (\ref{G5}) the operator $d^*$ is the adjoint of the gradient operator $d$ with respect to the measure (\ref{A5}), and hence $d^*d$ is a non-negative self-adjoint operator. 

Our first goal here will be to prove strong mixing of the operator $\tau_{\mathbf{e}_1,0}$ on $(\Om,\mathcal{F},P)$. In order to do this we will need a Poincar\'{e} inequality for the measure $(\Om_Q,\mathcal{F}_Q,P_{Q})$, in particular a generalization of (\ref{F5}) to functions $F(\phi(\cdot,t_1),..,\phi(\cdot,t_k))$ depending on values of the field $\phi(\cdot,\cdot)$ at different times.  To do this we follow the development of Gourcy-Wu \cite{gw} who make use of the Malliavin calculus \cite{nu} to prove a log-Sobolev inequality for such measures. The basic insight of the Malliavin calculus is that the Wiener space generated by independent Brownian motions  $B(x,t), \ x\in Q,t>0,$ can be identified with a probability space whose set of configurations is  the Hilbert space $L^2(Q\times\R^+)$, where $\R^+$ is the open interval $(0,\infty)$. We denote the Euclidean inner product on $L^2(Q\times\R^+)$ by $[\cdot,\cdot]$. The measure on $L^2(Q\times\R^+)$ is uniquely determined by the requirement that the variables 
$\psi\ra[\psi,\psi_j], \ j=1,..,k,$ are i.i.d. standard normal for any set of orthonormal vectors $\psi_j, \ j=1,..,k$. We denote this Malliavin probability space by $(\Om_{Q,{\rm Mal}},\mathcal{F}_{Q,{\rm Mal}},P_{Q,{\rm Mal}})$, where $\Om_{Q,{\rm Mal}}=L^2(Q\times\R^+)$ and $\mathcal{F}_{Q,{\rm Mal}}$ is determined by the requirement that the functions $\psi\ra[\psi,\psi_0]$ from $\Om_{Q,{\rm Mal}}$ to $\R$ are Borel measurable for all $\psi_0\in L^2(Q\times\R^+)$. 

The identification of the Wiener space with $(\Om_{Q,{\rm Mal}},\mathcal{F}_{Q,{\rm Mal}},P_{Q,{\rm Mal}})$ follows from the fact that the expectation of a function $F(\psi(\cdot,\cdot))$ with respect to $(\Om_{Q,{\rm Mal}},\mathcal{F}_{Q,{\rm Mal}},P_{Q,{\rm Mal}})$ is the same as the expectation of $F(W(\cdot,\cdot))$ with respect to Wiener space, where $W(\cdot,\cdot)$ is the white noise process corresponding to $B(\cdot,\cdot)$ in (\ref{C5}).  Hence the identification may be summarized as follows:
\be \label{H5}
\psi(x,t)\leftrightarrow W(x,t),   \quad W(x,t)= dB(x,t)/dt,  \quad x\in Q,t>0. 
\ee 
For $t>0$ let $\mathcal{F}_t$ be the $\sig-$field generated by the Brownian motions $B(x,s), \ x\in Q, s<t,$ of (\ref{C5}), so from (\ref{H5}) we can regard $\mathcal{F}_t$ as a sub $\sig-$field of  $\mathcal{F}_{Q,{\rm Mal}}$. We consider next vector fields $G: L^2(Q\times\R^+)\ra L^2(Q\times\R^+)$ on $\Om_{Q,{\rm Mal}}$ which are measurable in the sense that for any $\psi_0\in  L^2(Q\times\R^+)$ the function $\psi(\cdot,\cdot)\ra [ G(\psi(\cdot,\cdot)),\psi_0]$ is $(\Om_{Q,{\rm Mal}},\mathcal{F}_{Q,{\rm Mal}})$ measurable.  The vector field is {\it predictable} if  for any $t, \ 0<t<\infty$, 
$\psi_0$ has support in the interval $Q\times[0,t]$ implies that the function $[ G(\psi(\cdot,\cdot)),\psi_0]$  is $\mathcal{F}_t$ measurable. The Martingale representation theorem \cite{nu} implies that for any function $F\in L^2(\Om_{Q,{\rm Mal}})$ there is a predictable vector field $G: L^2(Q\times\R^+)\ra L^2(Q\times\R^+)$ such that
\begin{eqnarray} \label{I5}
{\rm var}_{\Om_{Q,{\rm Mal}}}[F(\cdot)] \ &=& \ \langle \  \|G(\cdot)\|^2 \ \rangle_{\Om_{Q,{\rm Mal}}} \ , \\
F(\psi(\cdot,\cdot))-\langle F(\cdot)\rangle \ &=& \  [G(\psi(\cdot,\cdot)),\psi(\cdot,\cdot)] \ . \nonumber
\end{eqnarray}
Suppose now that $F\in L^2(\Om_{Q,{\rm Mal}})$ also has a Malliavin derivative $D_{\rm Mal}F: L^2(Q\times\R^+)\ra L^2(Q\times\R^+)$  with the property that  $\langle \  \|D_{\rm Mal}F(\cdot)\|^2 \ \rangle_{\Om_{Q,{\rm Mal}}} <\infty$.  The Clark-Ocone formula \cite{nu} states  that the vector field $G(\psi(\cdot,\cdot))$ in (\ref{I5}) can be expressed in terms of the Malliavin derivative $D_{\rm Mal}F(\psi(\cdot,\cdot))$. Denoting the values of   $G(\psi(\cdot,\cdot)), \ D_{\rm Mal}F(\psi(\cdot,\cdot)),$ at $(x,t)\in Q\times\R^+$  by $G(x,t;\psi(\cdot,\cdot)), \ D_{\rm Mal}F(x,t;\psi(\cdot,\cdot))$ respectively, then
\be \label{J5}
G(x,t;\psi(\cdot,\cdot)) \ = \  \langle \ D_{\rm Mal}F(x,t;\psi(\cdot,\cdot)) \ |  \ \mathcal{F}_t\ \rangle_{\Om_{Q,{\rm Mal}}} \quad x\in Q, t>0.
\ee

We show how the Clark-Ocone formula  (\ref{I5}), (\ref{J5}) implies the  HS formula (\ref{G5}).  Let $\phi(\cdot,T)$ be the solution at time $T>0$ of (\ref{C5}) with initial data $\phi(\cdot,0)=0$ and $f:Q\ra\R$. We can find an expression for the Malliavin derivative of the function $F(\psi(\cdot,\cdot))=(f(\cdot),\phi(\cdot,T))$ by analyzing the first variation equation for (\ref{C5}).  Evidently one has that $D_{\rm Mal}F(x,t;\psi(\cdot,\cdot))=0$ for $x\in Q, t>T$. To get an expression for $D_{\rm Mal}F(x,t;\psi(\cdot,\cdot))$ when $t\le T$ we first note from (\ref{C5}) that 
\be \label{K5}
\frac{d}{dt} (f(\cdot),\phi(\cdot,t)) \ = \ -\frac{1}{2}\left\{ (\na f(\cdot), V'(\na\phi(\cdot,t)))+m^2(f(\cdot),\phi(\cdot,t)) \right\} +(f(\cdot), W(\cdot,t)) \ , \quad t>0.
\ee
It follows from (\ref{K5}) that for $\psi_0\in L^2(Q\times\R^+)$  the function $ \xi(\cdot,t)= [D_{\rm Mal}\phi(\cdot,t)),\psi_0]$  from $Q$ to $\R$ is a solution to the initial value problem
\begin{multline} \label{L5}
\frac{d}{dt}(f(\cdot),\xi(\cdot,t)) \ = \  -\frac{1}{2}\{ (\na f(\cdot), V''(\na\phi(\cdot,t))\na\xi(\cdot,t))+\\
m^2(f(\cdot),\xi(\cdot,t))\}+(f(\cdot),\psi_0(\cdot,t)) \ {\rm for  \ } t>0,  \ f:Q\ra\R; \quad \xi(\cdot,0)=0.
\end{multline}
From (\ref{L5}) we see that $\xi(x,t),  \ x\in Q,t>0,$ is the solution to the initial value problem for the parabolic PDE 
\begin{eqnarray} \label{M5}
\frac{\pa\xi(x,t)}{\pa t}  \ &=&  \ -\frac{1}{2}\{\na^*V''(\na\phi(x,t))\na\xi(x,t)+m^2\xi(x,t)\}+\psi_0(x,t), \\
 \xi(x,0) \ &=& \ 0. \nonumber
\end{eqnarray}
Consider now the terminal value problem for the backwards in time parabolic PDE
\begin{eqnarray} \label{N5}
\frac{\pa u(x,t)}{\pa t}  \ &=&  \ \frac{1}{2}\na^*V''(\na\phi(x,t))\na u(x,t), \   t<T, \\
 u(x,T) \ &=& \ u_0(x), \nonumber
\end{eqnarray}
with solution
\be \label{O5}
u(x,t) \ = \ \sum_{y\in Q} G(x,y,t,T,\phi(\cdot,\cdot)) u_0(y) \ , \quad t<T.
\ee
Then the solution to (\ref{M5}) is given by the formula
\be \label{P5}
\xi(y,T)= \int_0^Te^{-m^2(T-t)/2} dt\sum_{x\in Q} G(x,y,t,T,\phi(\cdot,\cdot)) \psi_0(x,t) \ , \quad y\in Q. 
\ee
We conclude from (\ref{P5}) that
\be \label{Q5}
(f(\cdot),D_{\rm Mal}\phi(x,t;\cdot,T)) \ = \ e^{-m^2(T-t)/2}  (f(\cdot),G(x,\cdot,t,T,\phi(\cdot,\cdot)))  , \quad  x\in Q, t<T,  \  f:Q\ra\R.
\ee

Suppose now that $F:\R^{L^d}\ra\R$ is a $C^1$ function such that $|DF(z)|\le C\exp[A|z|], \ z\in\R^{L^d}$, for some constants $C,A$. Then from (\ref{Q5}) it follows that
\begin{multline} \label{R5}
D_{\rm Mal}F(x,t; \phi(\cdot,T)) \ = \ e^{-m^2(T-t)/2}  (dF(\cdot,\phi(\cdot,T)),G(x,\cdot,t,T,\phi(\cdot,\cdot)))  , \quad  x\in Q, t<T,   \\
D_{\rm Mal}F(x,t; \phi(\cdot,T)) \ = \  0, \quad  x\in Q, t>T.
\end{multline}
Next we observe from (\ref{N5}), (\ref{R5}) that the conditional expectation (\ref{J5}) is given by the formula
\be \label{S5}
\langle \ D_{\rm Mal}F(\cdot,t; \phi(\cdot,T)) \ | \  \mathcal{F}_t \ \rangle_{\Om_{Q,\rm Mal}}  \ = \  e^{-H(T-t)/2} dF(\cdot,\phi(\cdot,t)) \ , \quad t<T,
\ee
where the operator $H$ is as in (\ref{G5}), so  $H \ = \ d^*d+\nabla^*V''(\nabla\phi(\cdot))\nabla+m^2$. Since for any fixed $s\ge 0$ the distribution of $\phi(\cdot,T-s)$ converges as $T\ra\infty$ to the distribution of $\phi(\cdot)$ for the invariant measure (\ref{A5}), it follows that
\be \label{T5}
\lim_{T\ra\infty}\langle \ \big| \langle \ D_{\rm Mal}F(\cdot,T-s; \phi(\cdot,T)) \ | \  \mathcal{F}_{T-s} \ \rangle_{\Om_{Q,\rm Mal}}\big|^2 \ \rangle_{\Om_{Q, \rm Mal}} \ = \ 
\langle \ dF(\cdot,\phi(\cdot)) \  e^{-Hs}dF(\cdot,\phi(\cdot)) \ \rangle_{\Om_Q} \ .
\ee
 Now (\ref{G5}) for $F_1=F_2$ follows from (\ref{I5}), (\ref{T5}) on letting $T\ra\infty$. The identity (\ref{G5}) for general $F_1,F_2$ is then a consequence of the symmetry of the LHS of (\ref{G5}) in $F_1,F_2$. 
 \begin{proposition}
Let $(\Om,\mathcal{F},P)$ be the  massive field theory probability space  defined by Proposition 6.1. Then the  operators $\tau_{\mathbf{e}_j,0}, \ 1\le j\le d,$ on $\Om$ are strong mixing. 
 \end{proposition}
 \begin{proof}
 We proceed as in the proof of Proposition 5.2 of \cite{cs}. It will be sufficient  to prove that for
  $k\ge 1$ and $(x_j,t_j)\in\Z^d\times\R, \ j=1,...,k $,
\begin{multline} \label{U5}
\lim_{n\ra\infty} \langle \ f(\phi(x_1+n{\bf e}_1,t_1),....,\phi(x_k+n{\bf e}_1,t_k)) \  g(\phi(x_1,t_1),....,\phi(x_k,t_k)) \ \rangle \ =  \\
\langle \  f(\phi(x_1,t_1),....,\phi(x_k,t_k)) \ \rangle \  \langle \  g(\phi(x_1,t_1),....,\phi(x_k,t_k)) \ \rangle 
\end{multline}
for all $C^\infty$ functions $f,g:\R^k\ra\R$ with compact support.  Let $Q\subset\Z^d$ be a large cube centered at the origin with side of length an even integer $L$. We define $h_{Q,T}(n)$ for $n\in\Z$ and $T>0$ large by
\begin{multline} \label{V5}
h_{Q,T}(n) \ = \ \langle \ f(\phi(x_1+n{\bf e}_1,t_1+T),....,\phi(x_k+n{\bf e}_1,t_k+T)) \  g(\phi(x_1,t_1+T),....,\phi(x_k,t_k+T)) \ \rangle_{\Om_Q,{\rm Mal}} \ -  \\
\langle \  f(\phi(x_1,t_1+T),....,\phi(x_k,t_k+T)) \ \rangle_{\Om_Q,{\rm Mal}} \  \langle \  g(\phi(x_1,t_1+T),....,\phi(x_k,t_k+T)) \ \rangle_{\Om_Q,{\rm Mal}} \ . 
\end{multline}
The function $h_{Q,T}:\Z\ra\R$ is periodic on the interval $I_L=\Z\cap[-L/2,L/2]$.  We shall show that there is a constant $C$ independent of $L,T$ as $L,T\ra\infty$ such that
\be \label{W5}
\sum_{n\in I_L} |h_{Q,T}(n)|^2 \ \le \ C. 
\ee
 Then (\ref{U5}) follows from (\ref{W5}) and Proposition 6.1 as in Proposition 5.2 .of \cite{cs}. 
 
 To estimate the LHS of (\ref{W5}) we go into Fourier variables, using the Plancherel theorem
 \be \label{X5}
 \sum_{n\in I_L} |h_{Q,T}(n)|^2  \ = \ \frac{1}{2\pi}\int_{\hat{I}_L}|\hat{h}_{Q,T}(\zeta)|^2 \ d\zeta \  .
 \ee
 Let $a(f,\zeta,\phi(\cdot,\cdot))$ be the function
 \be \label{Y5}
 a(f,\zeta,\phi(\cdot,\cdot)) \ = \ \sum_{n\in I_L} f(\phi(x_1+n{\bf e}_1,t_1+T),....,\phi(x_k+n{\bf e}_1,t_k+T))
  \ e^{in\zeta} \ .
 \ee
 Then the Fourier transform of $h_{Q,T}(\cdot)$ is bounded by
 \be \label{Z5}
 |\hat{h}_{Q,T}(\zeta)|^2 \ \le \ \frac{1}{L^2} \ {\rm var}_{\Om_{Q,{\rm Mal}}}[ a(f,\zeta,\phi(\cdot,\cdot))] \ 
  {\rm var}_{\Om_{Q,{\rm Mal}}}[ a(g,\zeta,\phi(\cdot,\cdot))] \ .
 \ee
 From (\ref{Q5}) we see that
 \be \label{AA5}
 |D_{{\rm Mal}} a(x,t; f,\zeta,\phi(\cdot,\cdot)) | \ \le \  \|Df(\cdot\|_\infty \sum_{j=1}^k \sum_{n\in I_L}
 e^{-m^2(T+t_j-t)/2}  G(x,x_j+n\mathbf{e}_1,t,T+t_j,\phi(\cdot,\cdot)) \ ,
 \ee
 where we are using the convention $G(\cdot,\cdot,s,S)=0$ if $s>S$. It follows from (\ref{AA5}) that
 \begin{multline} \label{AB5}
 \sum_{x\in Q}  |D_{{\rm Mal}} a(x,t; f,\zeta,\phi(\cdot,\cdot)) |^2 \ \le  \\
 kL\|Df(\cdot)\|_\infty^2 \sum_{j=1}^k 
 e^{-m^2(T+t_j-t)} \sup_{y\in Q} \sum_{x\in Q} \sum_{n\in I_L}G(x,y+n\mathbf{e}_1,t,T+t_j,\phi(\cdot,\cdot))G(x,y,t,T+t_j,\phi(\cdot,\cdot)) \ .
 \end{multline}
 Observe now that
 \be \label{AC5}
  \sum_{y'\in Q} G(x,y',t,T,\phi(\cdot,\cdot)) =   \sum_{x'\in Q} G(x',y,t,T,\phi(\cdot,\cdot))=1  ,  \quad x,y\in Q, t<T. 
 \ee 
 We conclude from (\ref{I5}), (\ref{AB5}), (\ref{AC5}) that 
 \be \label{AD5}
 {\rm var}_{\Om_{Q,{\rm Mal}}}[ a(f,\zeta,\phi(\cdot,\cdot))]  \ \le \  k^2L\|Df(\cdot)\|_\infty^2/m^2 \ .
 \ee
 The inequality (\ref{W5}) follows from (\ref{X5}), (\ref{Z5}), (\ref{AD5}).
 \end{proof}
 
 To proceed further we need to obtain a more general Poincar\'{e} inequality than was used in Proposition 6.2. In order to do this we consider functions $F(\phi(\cdot,\cdot))$ of continuous in time fields $\phi:Q\times\R\ra\R$. For $h\in L^2(Q\times\R)$, which is continuous in time, we define the {\it directional derivative} of $F(\phi(\cdot,\cdot))$ in direction $h$ by
 \be \label{AE5}
 dF_h(\phi(\cdot,\cdot)) \ = \ \lim_{\ve\ra 0} [F(\phi(\cdot,\cdot)+\ve h(\cdot,\cdot))-F(\phi(\cdot,\cdot))]/\ve \ .
 \ee
  For the functions $F(\phi(\cdot,\cdot))$ we shall be interested in, the directional derivative (\ref{AE5})  can be written as
  \be \label{AF5}
   dF_h(\phi(\cdot,\cdot)) \ = \ \sum_{x\in Q} \int_{-\infty}^\infty dt \ dF(x,t;\phi(\cdot,\cdot)) h(x,t) =[dF(\phi(\cdot,\cdot)), h] \ .
  \ee
  We shall call $dF(\cdot,\cdot;\phi(\cdot,\cdot))$ the {\it field derivative} of $F(\phi(\cdot,\cdot))$.
  One can use the HS formula (\ref{G5}) to obtain a Poincar\'{e} inequality for functions $F(\phi(\cdot,\cdot))$ of the form
  \be \label{AG5}
  F(\phi(\cdot,\cdot)) \ = \ \int_{-\infty}^\infty g(t)G(\phi(\cdot,t)) \ dt \ ,
  \ee
  where $g:\R\ra\C$ is a continuous function of compact support and $G(\phi(\cdot))$ is a complex valued $C^1$ function of fields $\phi:Q\ra\R$ which satisfies $|G(z)|+|DG(z)|\le A\exp[B|z|),  \ z\in\R^{L^d},$ for some constants $A,B$ .  Evidently from (\ref{AF5}) we see that the field derivative of the function (\ref{AG5}) is given by the formula
 \be \label{AH5}
 dF(x,t;\phi(\cdot,\cdot)) \ = \ g(t)dG(x,\phi(\cdot,t)), \quad x\in Q,t\in\R.
 \ee 
 
  Let us define now the correlation function $h:\R\ra\C$ by
  \be \label{AI5}
  h(t) \ = \ \langle \ \overline{G(\phi(\cdot,t))} \  G(\phi(\cdot,0)) \ \rangle_{\Om_Q}-
  \langle \ \overline{G(\phi(\cdot,t))} \ \rangle_{\Om_Q} \ \langle \  G(\phi(\cdot,0)) \ \rangle_{\Om_Q} \ .
  \ee
  Then the variance of $F(\phi(\cdot,\cdot))$ is given in terms of the Fourier transforms of $g(\cdot)$ and  $h(\cdot)$ by
  \be \label{AJ5}
  {\rm var}_{\Om_Q}[ F(\phi(\cdot,\cdot))] \ = \ \frac{1}{2\pi}\int_{-\infty}^\infty |\hat{g}(\zeta)|^2\hat{h}(\zeta)  \ d\zeta \ .
  \ee
 Note that the function $\hat{h}(\cdot)$ is real and non-negative. 
  Observe next that $h(t)$ can be written as an expectation with respect to the measure (\ref{A5}) by using the operator $d^*d$ which occurs in (\ref{G5}). Thus we have that
 \be \label{AK5}
 h(t) \ = \ \langle \ e^{-d^*dt/2}[\bar{G}(\phi(\cdot,0))-\langle\bar{G}(\phi(\cdot,0)\rangle_{\Om_Q}] \  [G(\phi(\cdot,0))-\langle G(\phi(\cdot,0)\rangle_{\Om_Q}] \ \rangle_{\Om_Q} \ ,   \quad t>0,
 \ee 
with a similar formula for $t<0$. For $\zeta\in\R$ let $u(\zeta,\phi(\cdot))$ be the solution to the elliptic PDE
\be \label{AL5}
\left[d^*d/2+i\zeta\right] u(\zeta,\phi(\cdot)) \ = \  [G(\phi(\cdot))-\langle G(\phi(\cdot)\rangle_{\Om_Q}] \ , \quad \phi:Q\ra\R.
\ee
 We conclude from (\ref{AK5}), (\ref{AL5})  that
\be \label{AM5}
\hat{h}(\zeta) \ = \ \langle \ [\bar{G}(\phi(\cdot,0))-\langle\bar{G}(\phi(\cdot,0)\rangle_{\Om_Q}]  \ 
[ u(\zeta,\phi(\cdot,0))+ u(-\zeta,\phi(\cdot,0))] \ \rangle_{\Om_Q} \ .
\ee
 If we apply the gradient operator $d$ to (\ref{AL5}) we obtain the equation
 \be \label{AN5}
\left[d^*d+2i\zeta+\nabla^*V''(\nabla\phi(\cdot))\nabla+m^2\right] du(\cdot,\zeta,\phi(\cdot)) \ = \  2dG(\cdot,\phi(\cdot)) \ ,\quad \phi:Q\ra\R.
 \ee
Hence (\ref{AM5}), (\ref{AN5}) and the HS formula (\ref{G5}) imply that
\begin{multline} \label{AO5}
\hat{h}(\zeta) \ =  \  4\times \ {\rm real \ part \ of} \\
\langle \ d\bar{G}(\cdot,\phi(\cdot,0))\left[d^*d+\nabla^*V''(\nabla\phi(\cdot))\nabla+m^2\right]^{-1} \ \left[d^*d+2i\zeta+\nabla^*V''(\nabla\phi(\cdot))\nabla+m^2\right]^{-1} 
 dG(\cdot,\phi(\cdot,0)) \ \rangle_{\Om_Q} \ .
\end{multline}
Just as (\ref{F5}) follows from (\ref{G5}),  we see from (\ref{AO5}) that 
\be \label{AP5}
0 \ \le \  \hat{h}(\zeta) \ \le \  \frac{4}{m^4} \langle \|dG(\phi(\cdot,0))\|^2\rangle_{\Om_Q} \ .
\ee

It follows from (\ref{AJ5}), (\ref{AP5}) that
\be \label{AQ5}
  {\rm var}_{\Om_Q}[ F(\phi(\cdot,\cdot))] \ \le \  \frac{4}{m^4} \langle \|dG(\cdot,\phi(\cdot,0))\|^2\rangle_{\Om_Q} \ \int_{-\infty}^\infty |g(t)|^2 \ dt .
\ee
Since from (\ref{AH5}) the inequality (\ref{AQ5}) can be rewritten as
\be \label{AR5}
  {\rm var}_{\Om_Q}[ F(\phi(\cdot,\cdot))] \ \le \  \frac{4}{m^4} \langle \|dF(\cdot,\cdot;\phi(\cdot,\cdot))\|^2\rangle_{\Om_Q} \ ,
\ee
we have obtained a Poincar\'{e} inequality for functions $F(\phi(\cdot,\cdot))$ of time dependent fields which are of the form (\ref{AG5}). We generalize this as follows:
\begin{lem}
Let $F(\phi(\cdot,\cdot))$ be a bounded function of continuous in time fields $\phi:Q\times\R\ra\R$ which is $C^1$ with respect to the $L^2(Q\times\R)$ metric, and assume that the field derivative function  $dF(\cdot,\cdot,\phi(\cdot,\cdot))$ with range $L^2(Q\times\R)$ is also  bounded. Then the inequality (\ref{AR5}) holds.
\end{lem}
 \begin{proof}
 Let $T>0$ be large and consider $F(\tau_{0,T}\phi(\cdot,\cdot))$ as a function of solutions $\phi(x,t), \ x\in Q,t>0,$ to the stochastic equation (\ref{C5}).   By the chain rule we have that
 \begin{multline} \label{AS5}
 D_{\rm Mal}F(x,t;\tau_{0,T}\phi(\cdot,\cdot))  \ = \\
  \sum_{y\in Q} \int_{t-T}^\infty ds \ dF(y,s;\tau_{0,T}\phi(\cdot,\cdot)) D_{\rm Mal}\phi(x,t;y,T+s) \  , \quad x\in Q,t>0.
 \end{multline}
 It follows then from (\ref{Q5}) that
  \begin{multline} \label{AT5}
 D_{\rm Mal}F(x,t;\tau_{0,T}\phi(\cdot,\cdot))  \ = \\
  \sum_{y\in Q} \int_{t-T}^\infty ds \ dF(y,s;\tau_{0,T}\phi(\cdot,\cdot))  \ e^{-m^2(T+s-t)/2} G(x,y,t,T+s,\phi(\cdot,\cdot)) \  , \quad x\in Q,t>0.
 \end{multline}
 Hence we have that
 \begin{multline} \label{AU5}
 \sum_{x\in Q}\int_0^\infty dt  \ |D_{\rm Mal}F(x,t;\tau_{0,T}\phi(\cdot,\cdot))|^2 \ = \\
 2 \sum_{x\in Q}\int_{0<t<T+s<T+s'} dt \  ds \ ds' \ e^{-m^2(T+s-t)/2}e^{-m^2(T+s'-t)/2} h(x,s)\overline{h(x,s')} \  ,
 \end{multline}
 where
 \be \label{AV5}
 h(x,s) \ = \ \sum_{y\in Q} G(x,y,t,T+s,\phi(\cdot,\cdot))  \  dF(y,s;\tau_{0,T}\phi(\cdot,\cdot))  \ .
 \ee
 It follows from (\ref{AC5}) that
 \be \label{AW5}
 \sum_{x\in Q} |h(x,s)|^2 \ \le \ \sum_{y\in Q} |dF(y,s;\tau_{0,T}\phi(\cdot,\cdot)|^2 \ ,
 \ee
 and so we conclude from (\ref{AU5}) that
 \be \label{AX5}
  \sum_{x\in Q}\int_0^\infty dt  \ |D_{\rm Mal}F(x,t;\tau_{0,T}\phi(\cdot,\cdot))|^2 \ \le \ 
  \frac{4}{m^4} \sum_{y\in Q}\int_{-\infty}^\infty ds  \ |dF(y,s;\tau_{0,T}\phi(\cdot,\cdot))|^2 \ .
 \ee
 Hence (\ref{I5}), (\ref{J5}) imply that
 \be \label{AY5}
  {\rm var}_{\Om_Q,{\rm Mal}}[ F(\tau_{0,T}\phi(\cdot,\cdot))] \ \le \  \frac{4}{m^4} \langle \|dF(\cdot,\cdot;\tau_{0,T}\phi(\cdot,\cdot))\|^2\rangle_{\Om_Q,{\rm Mal}} \ .
\ee 
The result follows now by observing that the limit of the LHS of (\ref{AY5}) as $T\ra\infty$ is equal to the LHS of (\ref{AR5}). Similarly the RHS of (\ref{AY5}) converges to the RHS of (\ref{AR5}). 
 \end{proof}
 We shall show how the Poincar\'{e} inequality (\ref{AR5}) can be used to improve the most elementary of the inequalities contained in $\S 2$. Thus let us consider an equation which differs from (\ref{AL2}) only  in that the projection operator $P$ has been omitted, 
\be \label{AZ5}
\eta\Phi(\xi,\eta,\om)+\pa\Phi(\xi,\eta,\om)+\pa_\xi^*{\bf a}(\om)\pa_\xi\Phi(\xi,\eta,\om)=-\pa^*_\xi {\bf a}(\om), \quad \eta>0, \ \xi\in \R^d, \ \om\in\Om.
\ee
 For any $v\in\C^d$ we multiply the row vector (\ref{AZ5}) on the right by the column vector $v$ and by the function $\overline{\Phi(\xi,\eta,\om)v}$ on the left.  Taking the expectation we see  that 
\be \label{BA5}
\|P\pa_\xi\Phi(\xi,\eta,\cdot)v\| \ \le \  \|\pa_\xi\Phi(\xi,\eta,\cdot)v\|  \ \le \  \frac{\La|v|}{\la}  \ .
\ee
where $\|\cdot\|$ denotes the norm in $\mathcal{H}(\Om)$. Let  $g:\Z^d\times\R\ra\C^d\otimes\C^d$ be in $L^p(\Z^d\times\R,\C^d\otimes\C^d)$ with norm given  by (\ref{H6}). 
 If $p=1$ then (\ref{BA5}) implies that
\be \label{BC5}
\|P\sum_{x\in\Z^d}\int_{-\infty}^\infty dt \  g(x,t)\pa_\xi\Phi(\xi,\eta,\tau_{x,-t}\cdot)v\|   \ \le \ \frac{\La|v|}{\la} \|g\|_1 \ .
\ee
The Poincar\'{e} inequality (\ref{AR5}) enables us to improve (\ref{BC5}) to allow $g\in L^p(\Z^d\times\R,\C^d\otimes\C^d)$ for some $p>1$.
\begin{proposition}
Suppose ${\bf a}(\cdot)$ in (\ref{AZ5}) is as in the statement of Theorem 1.2.  Then for $\xi\in\R^d, \ \Re\eta>0,$ there exists $p_0(\La/\la)$ depending only on $d$ and $\La/\la$ and satisfying $1<p_0(\La/\la)< 2$, such that for $g\in L^p(\Z^d\times\R,\C^d\otimes\C^d)$ with $1\le p\le p_0(\La/\la)$ and $v\in\C^d$,
\be \label{BD5}
\|P\sum_{x\in\Z^d} \int_{-\infty}^\infty dt \   g(x,t)\pa_\xi\Phi(\xi,\eta,\tau_{x,-t}\cdot)v\|   \ \le \ 
 \frac{C\La_1|v|}{m^2\La^{3/2-1/p}} \|g\|_p \ ,
\ee
where $\La_1$ is the constant in Theorem 1.2 and  $C$ depends only on $d$ and $\La/\la$. 
\end{proposition}
\begin{proof}
We shall first assume that $g(\cdot,\cdot)$ is continuous in time and has compact support in $\Z^d\times\R$. For a cube $Q$ such that $Q\times\R$ contains the support of $g(\cdot,\cdot)$, let $\Phi_Q(\xi,\eta,\cdot)$ be the solution to (\ref{AZ5}) with $\mathbf{a}(\phi)=\tilde{\mathbf{a}}(\phi(0,0)), \ \phi\in\Om_Q,$ so the random environment for (\ref{AZ5}) is $(\Om_Q,\mathcal{F}_Q,P_Q)$.  The inequality (\ref{AR5}) implies that
\begin{multline}  \label{BE5}
\|P\sum_{x\in\Z^d} \int_{-\infty}^\infty dt \   g(x,t)\pa_\xi\Phi(\xi,\eta,\tau_{x,-t}\cdot)v\|^2   \ \le \\
\frac{4}{m^4}\sum_{z\in Q} \int_{-\infty}^\infty ds \    \|\frac{\pa}{\pa\phi(z,s)}\sum_{x\in\Z^d} \int_{-\infty}^\infty dt \   g(x,t)\pa_\xi\Phi(\xi,\eta,\tau_{x,-t}\cdot)v\|^2   \ ,
\end{multline}
where we are using the notation $\pa/\pa\phi(z,s)F(\phi(\cdot,\cdot))$ to denote the value of the field derivative $dF(z,s;\phi(\cdot,\cdot))$  defined by (\ref{AF5}) of a function $F(\phi(\cdot,\cdot))$ at $(z,s)$.

Translation operators $\tau_{x,t}, \ x\in\Z^d,t\in\R$, act on functions $F_Q:\Om_Q\ra\C$  by $\tau_{x,t} F_Q(\phi(\cdot,\cdot))=F_Q(\tau_{x,t}\phi(\cdot,\cdot))$. We shall also need to use translation operators $T_{x,t}, \ x\in\Z^d,t\in\R$, which act on functions $G_Q:Q\times\R\times \Om_Q\ra\C$ by $T_{x,t} G_Q(z,s;\phi(\cdot,\cdot))=G_Q(z+x,s+t;\phi(\cdot,\cdot))$, so $T_{x,t}$ acts on the first two variables of $G_Q(\cdot,\cdot;\phi(\cdot,\cdot))$. The operators  $\tau_{x,t}, \ x\in\Z^d,t\in\R$, act on the third variable of $G_Q(\cdot,\cdot;\phi(\cdot,\cdot))$, and it is clear that they commute with the $T_{x,t}, \ x\in\Z^d,t\in\R$.
Let $F_Q:\Om_Q\ra\C$ be a function which is $C^1$ with respect to the $L^2(Q\times\R)$ metric as in Lemma 6.1. One easily sees from (\ref{AE5}), (\ref{AF5}) that 
\be \label{BF5}
d[\tau_{x,t} F_Q] \ = \ T_{-x,-t}\tau_{x,t} dF_Q, \quad x\in \Z^d, \ t\in\R,
\ee
whence it follows from (\ref{A2})  that
\be \label{BG5}
d[\pa_{j,\xi}\tau_{x,-t} F_Q] \ = \ [e^{-i{\bf e}_j.\xi}T_{-{\bf e}_j,0}\tau_{{\bf e}_j,0}-1]T_{-x,t}\tau_{x,-t} dF_Q, \qquad 1\le j\le d, \ x\in \Z^d, \ t\in\R.
\ee
Hence if we define a function $G_Q:Q\times\R\times \Om_Q\ra\C$ by
\be \label{BH5}
G_Q(y,r;\phi(\cdot,\cdot)) \ = \ e^{-iy\cdot\xi}dF_Q(-y,r;\tau_{y,-r}\phi(\cdot,\cdot)), \quad y\in Q, \ r\in\R,
\ee
then (\ref{BG5}) implies that
\be \label{BI5}
d[\pa_{j,\xi}\tau_{x,-t} F_Q](z,-s;\phi(\cdot,\cdot)) \ = \ e^{i(x-z)\cdot\xi} \ \na_jG_Q(x-z,t-s,\tau_{z,-s}\phi(\cdot,\cdot)),  \qquad 1\le j\le d, \ x,z\in \Z^d, \ t,s\in\R.
\ee
On taking $F_Q(\phi(\cdot,\cdot))=\Phi_Q(\xi,\eta,\phi(\cdot,\cdot))v$ and defining $G_Q$ by (\ref{BH5}), we conclude from (\ref{BI5}) that (\ref{BE5}) is the same as
\begin{multline} \label{BJ5}
\|P\sum_{x\in\Z^d}\int_{-\infty}^\infty dt \  g(x,t)\pa_\xi\Phi_Q(\xi,\eta,\tau_{x,-t}\cdot)v\|^2   \ \le \\
\frac{4}{m^4}\sum_{z\in Q}\int_{-\infty}^\infty ds \  \|  \ \sum_{x\in\Z^d}\int_{-\infty}^\infty dt \  g(x,t) e^{i(x-z)\cdot\xi} \  \na G_Q(x-z,t-s,\phi(\cdot,\cdot))\|^2 \ . 
\end{multline}

We can find an equation for $G_Q(\cdot,\cdot;\phi(\cdot,\cdot))$ by applying the operator $\pa/\pa\phi(\cdot,\cdot)$ to (\ref{AZ5}).   To see this let $h\in L^2(Q\times\R)$ be $C^1$ as a function of time and of compact support. Then (\ref{AZ5}) holds for $\om=\phi(\cdot,\cdot)$ and $\om=\phi(\cdot,\cdot)+\ve h(\cdot,\cdot)$. On subtracting the equations (\ref{AZ5}) for the different values of $\om$, dividing by $\ve$ and letting $\ve\ra 0$, we have from (\ref{AE5}), (\ref{AF5}) that the first term on the LHS of (\ref{AZ5}) converges to  $\eta[d\Phi(\xi,\eta,\phi(\cdot,\cdot))v,h]=\eta[dF_Q(\phi(\cdot,\cdot)),h]$. To find a similar expression for the limit as $\ve\ra 0$ of the second term on the LHS of (\ref{AZ5}), we observe that for $\del>0$, 
\begin{multline} \label{BK5}
\lim_{\ve\ra0} \frac{F_Q(\tau_{0,\del}[\phi(\cdot,\cdot)+\ve h(\cdot,\cdot)])-F_Q(\tau_{0,\del}\phi(\cdot,\cdot))}{\ve} \ = \\
 [dF_Q(\tau_{0,\del}\phi(\cdot,\cdot)),T_{0,\del} h] \ = \  [T_{0,-\del}dF_Q(\tau_{0,\del}\phi(\cdot,\cdot)), h] \ .
\end{multline}
Hence, assuming one can interchange the  limits $\ve\ra 0$ and $\del\ra 0$, we see from (\ref{BK5}) that the second term on the LHS of the difference of the two equations (\ref{AZ5}) converges to
\be \label{BL5}
\lim_{\del\ra 0} \left[ \frac{T_{0,-\del}dF_Q(\tau_{0,\del}\phi(\cdot,\cdot))-dF_Q(\phi(\cdot,\cdot))}{\del} \ , h\right] \ = \ [D_0dF_Q(\phi(\cdot,\cdot)),h] \ .
\ee
To find the limit as $\ve\ra 0$ of the term on the RHS of (\ref{AZ5})  we use the fact that $\mathbf{a}(\phi(\cdot,\cdot))=\tilde{\mathbf{a}}(\phi(0,0))$. Thus we obtain the expression
\be \label{BM5}
\lim_{\ve\ra 0}\frac{\pa^*_{\xi} \mathbf{a}(\phi(\cdot,\cdot))v-\pa^*_{\xi} \mathbf{a}(\phi(\cdot,\cdot)+\ve h(\cdot,\cdot))v}{\ve} \ =  \ -[D_\xi^*\{\ \del(\cdot,\cdot) D\tilde{{\bf a}}(\phi(0,0))v\},h] \ ,
\ee
where the operators  $D_{\xi}=(D_{1,\xi},..,D_{d,\xi})$ and $D^*_{\xi}=(D^*_{1,\xi},..,D^*_{d,\xi})$ are given by the formulae
\be \label{BN5}
D_{j,\xi}=[e^{-i{\bf e}_j.\xi}T_{-{\bf e}_j,0}\tau_{{\bf e}_j,0}-1], \quad 
D^*_{j,\xi}=[e^{i{\bf e}_j.\xi}T_{{\bf e}_j,0}\tau_{-{\bf e}_j,0}-1] \ , \quad 1\le j\le d.
\ee
The  function $\del:Q\times\R\ra\R$ in (\ref{BM5})  is the  delta function, $\del(0,t)=\del(t), \ \del(z,t)=0, \ z\ne 0$, where $\del(\cdot)$ is the Dirac delta function.  The limit as $\ve\ra 0$ of the third term on the LHS
of (\ref{AZ5}) can be expressed by a  similar formula. Thus we have
\begin{multline} \label{BO5}
\lim_{\ve\ra 0}\frac{\pa^*_{\xi} \mathbf{a}(\phi(\cdot,\cdot)+\ve h(\cdot,\cdot))\pa_\xi\Phi(\xi,\eta,\phi(\cdot,\cdot)+\ve h(\cdot,\cdot))v-\pa^*_{\xi} \mathbf{a}(\phi(\cdot,\cdot))\pa_\xi\Phi(\xi,\eta,\phi(\cdot,\cdot))v}{\ve}  \\
 = \ \left[D_\xi^*\tilde{{\bf a}}(\phi(0,0))D_\xi  \ dF_Q(\phi(\cdot,\cdot)) 
 +D_\xi^*\{ \del(\cdot,\cdot) D\tilde{{\bf a}}(\phi(0,0))\pa_\xi F_Q(\phi(\cdot,\cdot))\},h \right] \ . 
\end{multline}
 It follows from (\ref{AZ5}) and (\ref{BL5})-(\ref{BO5}) that $dF_Q(\phi(\cdot,\cdot))$ satisfies the equation
\begin{multline} \label{BP5}
\eta dF_Q(\phi(\cdot,\cdot))+D_0dF_Q(\phi(\cdot,\cdot))+D_\xi^*\tilde{{\bf a}}(\phi(0,0))D_\xi  \ dF_Q(\phi(\cdot,\cdot))  \\ 
= \ -D_\xi^*[ \del(\cdot,\cdot) D\tilde{{\bf a}}(\phi(0,0))\{v+\pa_\xi F_Q(\phi(\cdot,\cdot))\}] \ .
\end{multline}
Evidently for any $(y,-r)\in\Z^d\times\R$ we can replace $\phi(\cdot,\cdot)$ in (\ref{BP5}) by $\tau_{y,-r}\phi(\cdot,\cdot)$.
If we now evaluate (\ref{BP5}) with $\tau_{y,-r}\phi(\cdot,\cdot)$ substituted for $\phi(\cdot,\cdot)$ and  with the first variable of $dF_Q(\cdot,\cdot;\tau_{y,-r}\phi(\cdot,\cdot))$ equal to $-y$ and the second variable equal to $r$, we obtain an equation for the function $G_Q(\cdot,\cdot;\phi(\cdot,\cdot))$ of (\ref{BH5}),
\begin{multline} \label{BQ5}
\eta G_Q(y,r;\phi(\cdot,\cdot))- \frac{\pa  G_Q(y,r;\phi(\cdot,\cdot))}{\pa r} +\na^*_y\tilde{{\bf a}}(\phi(y,-r))\na_yG_Q(y,r;\phi(\cdot,\cdot)) \\
= \  -\na^*_y[ e^{-iy\cdot\xi}\del(-y,r) D\tilde{{\bf a}}(\phi(y,-r))\{v+\pa_\xi F_Q(\tau_{y,-r}\phi(\cdot,\cdot))\}] \ . 
\end{multline}

We define an operator $T_{\eta}$ on functions $g: \Z^d\times\R\times\Om\ra\C^d$ as follows: Let $u(y,r;\phi(\cdot,\cdot))$ be the solution to the equation
\be \label{BR5}
\eta u(y,r;\phi(\cdot,\cdot))- \frac{\pa  u(y,r;\phi(\cdot,\cdot))}{\pa r} +\La\na^*_y\na_yu(y,r;\phi(\cdot,\cdot))  \ = \  \La \na^*_y g(y,r;\phi(\cdot,\cdot)) \ .
\ee
Then $T_{\eta} g(y,r;\phi(\cdot,\cdot))=\na_yu(y,r;\phi(\cdot,\cdot)), \ y\in\Z^d,r\in\R$. It is easy to see that $T_{\eta}$ is a bounded operator on $L^2(\Z^d\times\R\times\Om,\C^d)$ with norm $\|T_{\eta}\|$ satisfying $\|T_{\eta}\|\le 1$. We can obtain a formula for $T_{\eta}$ which is similar to (\ref{AN2}).
Thus we have that
\be \label{BS5}
T_{\eta} g(y,r;\phi(\cdot,\cdot)) \ = \   \La\int_{0}^\infty e^{-\eta t} \ dt\sum_{x\in \Z^d} \left\{\nabla\nabla^* G_{\La}(x,t)\right\}  \ g(y-x,r+t;\phi(\cdot,\cdot)) \ ,
\ee 
with $G_\La(x,t)=G(x,\La t), \ x\in\Z^d,t>0,$ and $G(\cdot,\cdot)$ the Green's function (\ref{AM2}).  We can similarly define operators $T_{\eta,Q}$ on periodic functions $g_Q: Q\times\R\times\Om\ra\C^d$ by extending $g_Q$ periodically to the function  $g_Q: \Z^d\times\R\times\Om\ra\C^d$  and setting $T_{\eta,Q}g_Q=T_\eta g_Q$.  If we now take $g_Q$  to be given by the RHS of (\ref{BQ5}), so $\La g_Q(y,r;\phi(\cdot,\cdot))=e^{-iy\cdot\xi}\del(-y,r) D\tilde{{\bf a}}(\phi(y,-r))\{v+\pa_\xi F_Q(\tau_{y,-r}\phi(\cdot,\cdot))\}$, then $T_{\eta,Q} g_Q(y,r;\phi(\cdot,\cdot))=e^{r\eta}h_Q(y,r;\phi(\cdot,\cdot))$ where 
\begin{multline} \label{BT5}
h_Q(y,r;\phi(\cdot,\cdot)) \ = \  \sum_{n\in\Z^d}e^{-i(y+nL)\cdot\xi}\left\{\nabla\nabla^* G_{\La}(y+nL,-r)\right\} D\tilde{{\bf a}}(\phi(0,0))\{v+\pa_\xi F_Q(\phi(\cdot,\cdot))\} \ , \\
{\rm if \ } y\in Q,  \ r<0, \quad h_Q(y,r,\phi(\cdot,\cdot)) \ = \ 0 \ {\rm if \ } y\in Q, \  r>0,
\end{multline}
where $L$ is the length of the side of $Q$.

We can rewrite (\ref{BQ5}) using the function $h_Q$ of (\ref{BT5}). Thus let $u_Q(y,r;\phi(\cdot,\cdot)), \ y\in Q,r\in\R,$ be the solution to the periodic equation (\ref{BR5}) with $g=g_Q$.  Then (\ref{BQ5}), (\ref{BR5}) imply that $v_Q=G_Q+u_Q$ is the solution to the equation 
\begin{multline} \label{BU5}
\eta v_Q(y,r;\phi(\cdot,\cdot))- \frac{\pa  v_Q(y,r;\phi(\cdot,\cdot))}{\pa r} +\na^*_y\tilde{{\bf a}}(\phi(y,-r))\na_yv_Q(y,r;\phi(\cdot,\cdot)) \\
= \  -e^{r\eta}\La \na^*_y[\tilde{{\bf b}}(\phi(y,-r))h_Q(y,r;\phi(\cdot,\cdot))] \ , 
\end{multline}
where $\tilde{{\bf a}}(\cdot)=\La[I_d-\tilde{{\bf b}}(\cdot)]$.  It follows from (\ref{BA5}) that $\pa_\xi F_Q(\phi(\cdot,\cdot))$ is in $\mathcal{H}(\Om)$ and  $\|\pa_\xi F_Q(\phi(\cdot,\cdot))\|\le \La|v|/\la $. Since $|\na\na^*G_\La(x,t)|, \ x\in\Z^d,t>0,$ is bounded by $1/(\La t+1)$ times the RHS of (\ref{I2}), it follows from (\ref{BT5}) that $h_Q$ is in $ L^2(Q\times\R\times\Om,\C^d)$ and $\|h_Q\|\le C\sqrt{\La}\La_1|v|/\la$, where $C$ is a constant depending only on $d$, and  $\La_1$ is the constant in the statement of Theorem 1.2.  Since from (\ref{BU5}) we see  that $\|\na v_Q\|\le \La\|h_Q\|/\la$, we conclude that $\|\na G_Q\|\le C\La_1|v|(\La/\la)^2/\sqrt{\La}$.   It follows now from (\ref{BJ5}) and Young's inequality that (\ref{BD5}) holds for $p=1$ provided we can show that the LHS of (\ref{BJ5}) converges as $Q\ra\Z^d$ to the LHS of (\ref{BD5}).  To see this note that we are assuming that the function $g(\cdot,\cdot)$ in (\ref{BJ5}) has compact support and that $\Re\eta>0$. Hence we can use the perturbation expansion obtained from (\ref{J2}) and Proposition 6.1 to prove the convergence. 

We can also show that  the RHS of (\ref{BJ5}) converges as $Q\ra\Z^d$ by generating the function $\na G_Q$ from a perturbation expansion. Thus  let $\mathbf{B}:\Z^d\times\R\times\Om\ra\C^d\otimes\C^d$ be defined by $\mathbf{B}(y,r;\phi(\cdot,\cdot))=\tilde{\mathbf{b}}(\phi(y,-r)), \ y\in\Z^d,r\in\R$.  It follows from (\ref{BR5}), (\ref{BU5}) that $\na v_Q$ is the solution to the equation
\be \label{BV5}
\na v_Q(\cdot,\cdot;\phi(\cdot,\cdot)) \ = \  T_{\eta,Q}[\mathbf{B }(\cdot,\cdot;\phi(\cdot,\cdot))\{\na v_Q(\cdot,\cdot;\phi(\cdot,\cdot)) - e^{r\eta} h_Q(\cdot,\cdot;\phi(\cdot,\cdot))\}] \ .
\ee
Since $g\in L^1(\Z^d\times\R)$  it follows by the uniform in $Q$ estimates of the previous paragraph that it is sufficient to prove convergence as $Q\ra\Z^d$ for any finite number of terms in the Neumann series expansion of (\ref{BV5}).  The convergence for a finite number of terms  follows from Proposition 6.1 using the  fact that the function $g(\cdot,\cdot)$ in (\ref{BJ5}) has compact support and that $\Re\eta>0$. 
We have shown now that  
\begin{multline} \label{BW5}
\|P\sum_{x\in\Z^d}\int_{-\infty}^\infty dt \  g(x,t)\pa_\xi\Phi(\xi,\eta,\tau_{x,-t}\cdot)v\|^2   \ \le \\
\frac{4}{m^4}\sum_{z\in \Z^d}\int_{-\infty}^\infty ds \  \|  \ \sum_{x\in\Z^d}\int_{-\infty}^\infty dt \  g(x,t) e^{i(x-z)\cdot\xi} \  \na G(x-z,t-s,\phi(\cdot,\cdot))\|^2 \ , 
\end{multline}
where $ \na G=\na v-h$  with
\begin{multline} \label{BX5}
h(y,r;\phi(\cdot,\cdot)) \ = \  \left\{\nabla\nabla^* G_{\La}(y,-r)\right\}^* D\tilde{{\bf a}}(\phi(0,0))\{v+\pa_\xi F_Q(\phi(\cdot,\cdot))\} \ , \\
{\rm if \ } y\in \Z^d,  \ r<0, \quad h(y,r,\phi(\cdot,\cdot)) \ = \ 0 \ {\rm if \ } y\in \Z^d, \  r>0,
\end{multline}
and $\na v$ is the solution to the equation 
\be \label{BY5}
\na v(\cdot,\cdot;\phi(\cdot,\cdot)) \ = \  T_{\eta}[\mathbf{B }(\cdot,\cdot;\phi(\cdot,\cdot))\{\na v(\cdot,\cdot;\phi(\cdot,\cdot)) - e^{r\eta} h(\cdot,\cdot;\phi(\cdot,\cdot))\}] \ .
\ee

We can now  easily extend the previous argument by using the continuous time version of the Calderon-Zygmund theorem, Corollary  5.1, to prove (\ref{BD5}) for a range of $p>1$.  
Define  for $q\ge 1$ the Banach space $L^q(\Z^d\times\R\times\Om,\C^d)$ of  functions $g:\Z^d\times\R\times\Om\ra\C ^d$ with norm $\|g\|_q$ given by
\be \label{BZ5}
\|g\|_q^q \ = \ \sum_{y\in \Z^d}\int_{-\infty}^\infty dt \  \|g(y,r;\phi(\cdot,\cdot))\|^q \ ,
\ee
where $ \|g(y,r;\phi(\cdot,\cdot))\|$ is the norm of  $g(y,r;\phi(\cdot,\cdot))\in\mathcal{H}(\Om)$.
By following the argument of Lemma 5.2, we see that $T_{\eta}$ is bounded on $L^q(\Z^d\times\R\times\Om,\C^d)$  for $q>1$ with norm $\|T_{\eta}\|_q\le 1+\del(q)$, where $\lim_{q\ra 2}\del(q)=0$. Noting that $\|h\|_q\le C_q\La^{1-1/q}\La_1|v|/\la$  for a constant $C_q$ depending only on $d,q$, we conclude from (\ref{BY5}) and the Calderon-Zygmund theorem  that there exists $q_0(\La/\la)<2$ depending only on $d,\La/\la$, such that $\na G$ is in  $L^q(\Z^d\times\R\times\Om,\C^d)$   for $q_0(\La/\la)\le q\le 2$, and
$\|\na G\|_q  \le  C\La^{-1/q}\La_1|v|$ where the constant $C$ depends only on $d,\La/\la$.  The inequality (\ref{BD5}) with $p=2q/(3q-2)$ follows from (\ref{BW5}) and Young's inequality. 
\end{proof}
In order to establish Hypothesis 4.2 for the massive field theory environment $(\Om,\mathcal{F},P)$ we shall need a refinement of the Poincar\'{e} inequality (\ref{AR5}).  We can see what this refinement should be by considering again functions of the form (\ref{AG5}), for which (\ref{AJ5}) and (\ref{AO5}) hold. It follows from (\ref{AO5}) that $\hat{h}(\zeta)$ satisfies the inequality
\be \label{CM5}
0 \ \le \hat{h}(\zeta) \ \le \ \frac{4}{m^4+4\zeta^2} \langle \|dG(\phi(\cdot,0))\|^2\rangle_{\Om_Q} \ .
\ee
Substituting the RHS of (\ref{CM5}) into (\ref{AJ5}) we obtain the  inequality
\begin{multline} \label{CN5}
 {\rm var}_{\Om_Q}[ F(\phi(\cdot,\cdot))] \ \le \   \frac{1}{m^2} \int_{-\infty}^\infty\int_{-\infty}^\infty g(t)\overline{g(s)}e^{-m^2|t-s|/2} \ dt \ ds  \ \langle \|dG(\phi(\cdot,0))\|^2\rangle_{\Om_Q} \\
 \le \  \frac{4}{m^4}\int_{-\infty}^\infty |g(t)|^2 \ dt \  \langle \|dG(\phi(\cdot,0))\|^2\rangle_{\Om_Q}  \ = \ 
 \frac{4}{m^4} \langle \|dF(\cdot,\cdot;\phi(\cdot,\cdot))\|^2\rangle_{\Om_Q} \ .
\end{multline}
Observe now that the first integral on the  RHS of (\ref{CN5})  can be written as a convolution  $[g,f^*g]$ where $f(t)=m^{-2}e^{-m^2|t|/2}, \ t\in\R$.  Hence it follows from Young's inequality that for $1\le p\le2$, 
\be \label{CO5}
{\rm var}_{\Om_Q}[ F(\phi(\cdot,\cdot))] \ \le \  \frac{C}{m^{2(3-2/p)}}\|g\|_p^{2} \langle \|dG(\phi(\cdot,0))\|^2\rangle_{\Om_Q}  \ ,
\ee
where $\|g\|_p$ denotes the $L^p$ norm of $g(\cdot)$ and $C$ is a universal constant.  The Poincar\'{e} inequality (\ref{AR5}) only implies (\ref{CO5}) for $p=2$. 

We shall also need a continuous time version of Corollary 5.1, as we already did in the proof of Proposition 6.3. Thus  let $T_{\xi,\eta}$ act on functions  $g:\Z^d\times\R\ra\C^d$ as
\be \label{DI5}
T_{\xi,\eta} g(y,r) \ = \   \La\int_{0}^\infty e^{-\eta t} \ dt\sum_{x\in \Z^d} \left\{\nabla\nabla^* G_{\La}(x,t)\right\}  \ e^{ix\cdot\xi} g(y-x,r+t) \ .
\ee 
Comparing the operator $T_{\xi,\eta}$ of (\ref{DI5}) to the operator $T_{\xi,\eta}$ of (\ref{B4}), we see that one can easily extend the argument of $\S4$ to obtain a continuous time version of Corollary 5.1:
\begin{corollary}
For $(\xi,\eta)$ satisfying the assumptions of Lemma 5.2 the operator $T_{\xi,\eta}$ of (\ref{DI5}) is bounded on $\mathcal{H}^p(\Z^{d}\times\R)$ for $3/2\le p\le 3,$ and  $\|T_{\xi,\eta}\|_p\le [1+\del(p)]\left(1+C_2 |\Im\xi|^2/[\Re\eta/\La]\right)$, where the function $\del(\cdot)$ depends only on $d$ and $\lim_{p\ra 2}\del(p)=0$. 
\end{corollary}
\begin{proof}[Proof of Hypothesis 4.2]
We shall first prove Hypothesis 3.2.
 We assume  $g:\Z^{d}\times\R\ra\C^d\otimes \C^d$ has compact support and for $k=1,2,...,$ denote by $a_k(g,\xi,\eta)$ the random $d\times d$ matrix  
\be \label{CC5}
a_k(g,\xi,\eta) \ = \   \sum_{x\in\Z^{d}}\int_{-\infty}^\infty  dt \  g(x,t)\tau_{x,-t}P{\bf b}(\cdot)\left[ PT_{\xi,\eta}{\bf b}(\cdot)\right]^{k-1} \ .
\ee
Evidently Hypothesis 3.2 will follow if we can show there is a constant $C$ such that
\be \label{CD5}
\sum_{k=1}^\infty \|a_k(g,\xi,\eta)v\| \ \le  \ C\|g\|_p|v| \quad {\rm for \ } 1\le p\le p_0(\La/\la), \ v\in\C^d.
\ee
We establish (\ref{CD5}) by obtaining a bound  $\|a_k(g,\xi,\eta)v\|\le C_k\|g\|_p|v|$ where $C_k$ decays exponentially in $k$ as $k\ra\infty$. 

In the case $k=1$ we have from the Poincar\'{e} inequality  (\ref{AR5}) that 
\begin{multline} \label{CF5}
\|a_1(g,\xi,\eta)v\|^2 \ \le \  \frac{4}{m^4}\sum_{x\in\Z^d}\int_{-\infty}^\infty dt \ 
\|g(x,t)D\tilde{\mathbf{b}}(\phi(0,0))v\|^2 \\
 \le \  \left(\frac{2\La_1}{m^2\La}\|g\|_2|v|\right)^2 \ .
\end{multline}
We also have that $a_1(g,\xi,\eta)=\hat{g}(A,B)P\mathbf{b}(\cdot)$, where $\hat{g}$ is the Fourier transform (\ref{AO2}) of $g$ and $A,B$ are the self-adjoint operators (\ref{AQ2}), (\ref{AR2}). Hence we have that $\|a_1(g,\xi,\eta)v\|\le \|g\|_1|v|$. We conclude therefore from (\ref{CF5}) and  the Riesz-Thorin interpolation theorem \cite{sw} there is a constant $C_1$ such that  $\|a_1(g,\xi,\eta)v\|\le C_1\|g\|_p|v|$ for $1\le p\le 2$ . 

When $k>1$ we  write
\be \label{CG5}
a_k(g,\xi,\eta)v \ = \ P\sum_{x\in\Z^{d}}\int_{-\infty}^\infty  dt \  g(x,t)\tau_{x,-t}{\bf b}(\cdot)\pa_\xi F_k(\phi(\cdot,\cdot)) \ ,
\ee
where the functions $F_k(\phi(\cdot,\cdot))$ are defined inductively.  For  $\xi\in\R^d$ the $F_k(\phi(\cdot,\cdot))$ satisfy the recurrence equations
\begin{multline} \label{CH5}
[\eta+\pa]F_2(\phi(\cdot,\cdot))+\La\pa _\xi^*\pa_\xi  F_2(\phi(\cdot,\cdot)) \ = \ \La P\pa^*_\xi[\tilde{\mathbf{b}}(\phi(0,0))v] \ , \\
[\eta+\pa]F_k(\phi(\cdot,\cdot))+\La\pa _\xi^*\pa_\xi  F_k(\phi(\cdot,\cdot)) \ = \ \La P\pa^*_\xi[\tilde{\mathbf{b}}(\phi(0,0))\pa_\xi F_{k-1}(\phi(\cdot,\cdot))]   \ {\rm if   \ } k>2.
\end{multline}
The $F_k(\phi(\cdot,\cdot))$  for $\xi\in\C^d$ are defined by analytic continuation from the values of $F_k(\phi(\cdot,\cdot))$ when $\xi\in\R^d$. 
Similarly to (\ref{BH5}) we define for $k\ge 2$ functions $G_k:\Z^d\times\R\times\Om\ra \C$ by
\be \label{CI5}
G_k(y,r;\phi(\cdot,\cdot)) \ = \ e^{-iy\cdot\xi}dF_k(-y,r;\tau_{y,-r}\phi(\cdot,\cdot)), \quad y\in \Z^d, \ r\in\R.
\ee
Then from (\ref{CH5}) we see that the $G_k(y,r;\phi(\cdot,\cdot))$ satisfy the equations
\begin{multline} \label{CJ5}
\eta G_2(y,r;\phi(\cdot,\cdot))- \frac{\pa  G_2(y,r;\phi(\cdot,\cdot))}{\pa r} +\La\na^*_y\na_yG_2(y,r;\phi(\cdot,\cdot))  \ = \  \La P \na^*_y[e^{-iy\cdot\xi}\del(-y,r)D\tilde{\mathbf{b}}(\phi(y,-r))v] \ ,
\\
\eta G_k(y,r;\phi(\cdot,\cdot))- \frac{\pa  G_k(y,r;\phi(\cdot,\cdot))}{\pa r} +\La\na^*_y\na_yG_k(y,r;\phi(\cdot,\cdot))  \ = \\
\La P\na^*_y[e^{-iy\cdot\xi}\del(-y,r)D\tilde{\mathbf{b}}(\phi(y,-r)))\pa_\xi F_{k-1}(\tau_{y,-r}\phi(\cdot,\cdot))+\tilde{\mathbf{b}}(\phi(y,-r))\na_yG_{k-1}(y,r;\phi(\cdot,\cdot))] \quad {\rm if  \ } k>2.
\end{multline}

Instead of estimating the norm of the function $a_k(g,\xi,\eta)v$  of (\ref{CG5}) directly by using  the Poincar\'{e} inequality as in (\ref{BJ5}), we begin with the Clark-Okone formula (\ref{I5}).  Let $\phi(\cdot,t), \ t>0,$ be the solution of (\ref{C5}) with initial condition $\phi(\cdot,0)=0$.  We extend the function $\phi(\cdot,t)$ to $t<0$ by setting $\phi(\cdot,t)=0$ for $t<0$.  It is then easy to see that
\be \label{CK5}
\|a_k(g,\xi,\eta)v\|^2 \ = \
\lim_{T\ra\infty}{\rm var}_{\Om_Q,{\rm Mal}}[ \  H(\tau_{0,T}\phi(\cdot,\cdot)) \ ] \ ,
\ee
where the function $H(\phi(\cdot,\cdot))$ is given by the formula
\be \label{CL5}
H(\phi(\cdot,\cdot)) \ = \ \sum_{y\in\Z^{d}}\int_{-\infty}^\infty  ds \  g(y,s)\tilde{\mathbf{b}}(\phi(y,-s))\pa_\xi F_k(\tau_{y,-s}\phi(\cdot,\cdot))   \ .
\ee
We have now from (\ref{AS5})  that for $x\in Q, \ t>0,$ the Malliavin derivative $D_{\rm Mal} H(x,t;\tau_{0,T}\phi(\cdot,\cdot))=\sig_{1,T}(x,t;\phi(\cdot,\cdot))+\sig_{2,T}(x,t;\phi(\cdot,\cdot))$, where
\begin{multline} \label{CP5}
\sig_{1,T}(x,t;\phi(\cdot,\cdot)) \ = \\
 \sum_{y\in\Z^{d}}\int_{-\infty}^{T-t}  ds \  g(y,s)e^{-m^2(T-t-s)/2}G(x,y,t,T-s,\phi(\cdot,\cdot))D\tilde{\mathbf{b}}(\phi(y,T-s))\pa_\xi F_k(\tau_{y,T-s}\phi(\cdot,\cdot))   \ ,
\end{multline}
with $G(\cdot,\cdot,\cdot,\cdot,\phi(\cdot,\cdot))$ being the Green's function (\ref{O5}). 
The function $\sig_{2,T}(x,t;\phi(\cdot,\cdot))$ is given by the formula
\be \label{CQ5}
\sig_{2,T}(\cdot,\cdot;\phi(\cdot,\cdot)) \ = \ \sum_{y\in\Z^{d}}\int_{-\infty}^\infty  ds \  g(y,s)\tilde{\mathbf{b}}(\phi(y,T-s))D_{\rm Mal}[\pa_\xi F_k(\tau_{y,T-s}\phi(\cdot,\cdot)) ] \ .
\ee
It follows from  (\ref{I5}) that
\begin{multline} \label{CR5}
\|a_k(g,\xi,\eta)v\| \ \le \ \lim_{T\ra\infty}
 \left[\sum_{x\in Q}\int_0^\infty dt  \  |\langle \ \sig_{1,T}(x,t;\phi(\cdot,\cdot))  \ | \ \mathcal{F}_t \ \rangle_{\Om_{Q,{\rm Mal}}}|^2 \ \right]^{1/2} \\
+ \lim_{T\ra\infty}\left[\sum_{x\in Q}\int_0^\infty dt  \  |\langle \ \sig_{2,T}(x,t;\phi(\cdot,\cdot))  \ | \ \mathcal{F}_t \ \rangle_{\Om_{Q,{\rm Mal}}} |^2 \ \right]^{1/2} \ .
\end{multline}

To estimate the first term on the RHS of (\ref{CR5}) we argue as in Lemma 6.1. Thus from (\ref{CP5}) we have that 
\begin{multline} \label{CS5}
\sum_{x\in Q}\int_0^\infty dt  \  |\sig_{1,T}(x,t;\phi(\cdot,\cdot))  \ |^2 \ = \\ 
2 \sum_{x\in Q}\int_{0<t<T-s<T-s'} dt \  ds \ ds' \ e^{-m^2(T-t-s)/2}e^{-m^2(T-t-s')/2} h(x,s)\cdot\overline{h(x,s')} \  ,
 \end{multline}
 where
 \be \label{CT5}
 h(x,s) \ = \ \sum_{y\in Q} G(x,y,t,T-s,\phi(\cdot,\cdot)) g(y,s)D\tilde{\mathbf{b}}(\phi(y,T-s))\pa_\xi F_k(\tau_{y,T-s}\phi(\cdot,\cdot)) \ .
 \ee
 It follows from (\ref{AC5}) that
 \be \label{CU5}
 \sum_{x\in Q} |h(x,s)|^2 \ \le \ \sum_{y\in Q}| g(y,s)D\tilde{\mathbf{b}}(\phi(y,T-s))\pa_\xi F_k(\tau_{y,T-s}\phi(\cdot,\cdot))|^2 \ ,
 \ee
 and so we conclude from (\ref{CS5}) that
 \be \label{CV5}
 \sum_{x\in Q}\int_0^\infty dt  \  |\sig_{1,T}(x,t;\phi(\cdot,\cdot))  \ |^2 \ \le \frac{1}{m^2}\int_{t-T}^\infty\int_{t-T}^\infty ds \  ds'  \  e^{-m^2|s-s'|/2} k(s,\tau_{0,T}\phi(\cdot,\cdot)) \  k(s',\tau_{0,T}\phi(\cdot,\cdot)) \ .
 \ee
 The function $k(s,\phi(\cdot,\cdot))$ is given by the formula
 \be \label{CW5}
 k(s,\phi(\cdot,\cdot)) \ =  \ \left[\sum_{y\in Q}| g(y,s)D\tilde{\mathbf{b}}(\phi(y,-s))\pa_\xi F_k(\tau_{y,-s}\phi(\cdot,\cdot))|^2\right]^{1/2} \ = \|g(\cdot,s)\|_2 \  k_1(s,\phi(\cdot,\cdot)) ,
 \ee
 where $\|g(\cdot,s)\|_2$ denotes the $L^2$ norm of the $d\times d$ matrix valued function $g(y,s), \ y\in\Z^d$.  Observe next that  (\ref{CV5}), (\ref{CW5}) and the Schwarz inequality imply that
 \be \label{CX5}
  \sum_{x\in Q}\int_0^\infty dt  \  |\sig_{1,T}(x,t;\phi(\cdot,\cdot))  \ |^2 \ \le \frac{1}{m^2}\int_{t-T}^\infty\int_{t-T}^\infty ds \  ds'  \  e^{-m^2|s-s'|/2}\|g(\cdot,s)\|_2\|g(\cdot,s')\|_2 k_1(s,\tau_{0,T}\phi(\cdot,\cdot))^2 \ . 
 \ee
 Hence on using the fact that for any $s\in\R$, one has 
 \be \label{CY5}
\lim_{T\ra\infty} \langle  \  k_1(s,\tau_{0,T}\phi(\cdot,\cdot))^2 \ \rangle_{\Om_{Q,{\rm Mal}}} \ \le \ 
\frac{C\La_1^2}{\La^2} \langle \ |\pa_\xi F_k\phi(\cdot,\cdot))|^2 \ \rangle_{\Om_Q} \ ,
 \ee
 where the constant $C$ depends only on $d$, we conclude as in the argument  showing (\ref{CO5}) that for $1\le p\le 2$ there is a constant $C$ depending only on  $d$ such that
 \begin{multline} \label{CZ5}
 \lim_{T\ra\infty}
 \sum_{x\in Q}\int_0^\infty dt  \  |\langle \ \sig_{1,T}(x,t;\phi(\cdot,\cdot))  \ | \ \mathcal{F}_t \ \rangle_{\Om_{Q,{\rm Mal}}}|^2   \ \le \\
 \frac{C\La_1^2}{\La^2m^{2(3-2/p)}}\left[ \int_{-\infty}^\infty \|g(\cdot,s)\|_2^p \ ds  \right]^{2/p}  \langle \ |\pa_\xi F_k\phi(\cdot,\cdot))|^2 \ \rangle_{\Om_Q}  \ .
\end{multline}
From (\ref{CH5}) we see that for $\xi\in\R^d$
\be \label{DA5}
\langle\ |\pa_\xi F_k(\phi(\cdot,\cdot))|^2 \ \rangle_{\Om_Q} \ \le \  (1-\la/\La)^{2(k-1)}|v|^2\ ,
\ee
and so (\ref{CZ5}) implies that for $\xi\in\R^d$ the first term on the RHS of (\ref{CR5}) is bounded as
 \be \label{DB5}
 \lim_{T\ra\infty}
 \sum_{x\in Q}\int_0^\infty dt  \  |\langle \ \sig_{1,T}(x,t;\phi(\cdot,\cdot))  \ | \ \mathcal{F}_t \ \rangle_{\Om_{Q,{\rm Mal}}}|^2   \ \le \  
  \left\{\frac{C\La_1}{\La m^{3-2/p}} \|g(\cdot,\cdot)\|_p \  (1-\la/\La)^{k-1}|v| \ \right\}^2 \ ,
 \ee
where the $p$ norm of $g(\cdot,\cdot)$ is given by (\ref{H6}). 

We can estimate the second term on the RHS of (\ref{CR5}) by following the argument of Proposition 6.3.
Thus we have that
\begin{multline} \label{DC5}
 \lim_{T\ra\infty}
 \sum_{x\in Q}\int_0^\infty dt  \  |\langle \ \sig_{2,T}(x,t;\phi(\cdot,\cdot))  \ | \ \mathcal{F}_t \ \rangle_{\Om_{Q,{\rm Mal}}}|^2 \ \le  \\
\frac{4}{m^4}\sum_{z\in Q}\int_{-\infty}^\infty ds \  \left\langle \ \left|
\sum_{x\in \Z^d}\int_{-\infty}^\infty dt \  g(x,t)  \  \tilde{\mathbf{b}}(\phi(x,-t))e^{i(x-z)\cdot\xi}\na G_k(x-z,t-s;\tau_{z,-s}\phi(\cdot,\cdot)) \  \right|^2 \ \right\rangle_{\Om_Q}    \\  
=  \ \frac{4}{m^4}\sum_{z\in Q}\int_{-\infty}^\infty ds \\
 \left\langle \ \left|
\sum_{x\in \Z^d}\int_{-\infty}^\infty dt \  g(x,t)  \  \tilde{\mathbf{b}}(\phi(x-z,s-t)) e^{i(x-z)\cdot\xi}\na G_k(x-z,t-s;\phi(\cdot,\cdot)) \  \right|^2 \ \right\rangle_{\Om_Q}   \ ,
\end{multline}
where we have used  the invariance of the operators $\tau_{z,-s}, \ z\in\Z^d,s\in\R,$ on $(\Om_Q,\mathcal{F}_Q,P_Q)$. As in Proposition 6.3 we are justified in taking the limit $Q\ra\Z^d$ in (\ref{DC5}), and hence (\ref{BR5}), (\ref{BS5}),  (\ref{CJ5}) imply that $\na G_2(\cdot,\cdot;\phi(\cdot,\cdot))$ is given by the formula
\begin{multline} \label{DE5}
\na G_2(y,r;\phi(\cdot,\cdot)) \ = \ \La e^{\eta r}\na\na^*G_\La(y,-r)P[D\tilde{\mathbf{b}}(\phi(0,0))v] \ , \\
{\rm  if \ } y\in\Z^d, \ r<0, \quad \na G_2(y,r;\phi(\cdot,\cdot)) \ = \ 0  \ {\rm  if \ } y\in\Z^d, \ r>0.
\end{multline}
We similarly have that for $k>2$
\be \label{DF5}
\na G_k(y,r;\phi(\cdot,\cdot)) \ = \ e^{\eta r}h_k(y,r;\phi(\cdot,\cdot)) +PT_\eta[\mathbf{B}(\cdot,\cdot;\phi(\cdot,\cdot))\na G_{k-1}(\cdot,\cdot;\phi(\cdot,\cdot))] \ ,
\ee
where the function $\mathbf{B}(\cdot,\cdot;\phi(\cdot,\cdot))$ is as in (\ref{BV5}) and 
\begin{multline} \label{DG5}
h_k(y,r;\phi(\cdot,\cdot)) \ = \ \La \na\na^*G_\La(y,-r)P[D\tilde{\mathbf{b}}(\phi(0,0))\pa_\xi F_{k-1}(\phi(\cdot,\cdot))] \ , \\
{\rm  if \ } y\in\Z^d, \ r<0, \quad  h_k(y,r;\phi(\cdot,\cdot)) \ = \ 0  \ {\rm  if \ } y\in\Z^d, \ r>0.
\end{multline}
Defining the function $A_k:\Z^d\times\R\times\Om\ra\C^d$ by $A_k(y,r;\phi(\cdot,\cdot))=e^{iy\cdot\xi}\na G_k(y,r;\phi(\cdot,\cdot)) $,  we see from (\ref{DF5}) that for $k>2$ the function  $A_k(\cdot,\cdot;\phi(\cdot,\cdot))$ satisfies the equation
\be \label{DH5}
A_k(\cdot,\cdot;\phi(\cdot,\cdot)) \ = \  D_k(\cdot,\cdot;\phi(\cdot,\cdot))+PT_{\xi,\eta}[\mathbf{B}(\cdot,\cdot;\phi(\cdot,\cdot)) A_{k-1}(\cdot,\cdot;\phi(\cdot,\cdot))] \ ,
\ee
where $D_k(y,r;\phi(\cdot,\cdot))=e^{iy\cdot\xi+\eta r}h_k(y,r;\phi(\cdot,\cdot))$  and from (\ref{BS5}) it follows that the operator $T_{\xi,\eta}$ is given by (\ref{DI5}).  

Just as in Proposition 6.3 we see that if $\|\cdot\|_q$ denotes the $q$ norm (\ref{BZ5}) then for $\xi\in\R^d$ and  $1<q\le 2$
\be \label{DJ5}
\|A_2\|_q \ \le \  C_q\La_1|v|/\La^{1/q} \ , \quad \|D_k\|_q \ \le \  C_q(1-\la/\La)^{k-2}\La_1|v|/\La^{1/q} \  \ {\rm for \ } k>2,
\ee
where the constant $C_q$ depends only on $d,q$ and diverges as $q\ra 1$.  It follows then from (\ref{DH5}), (\ref{DJ5}) that for $\xi\in\R^d$ one has the inequality $\|A_k\|_2\le Ck(1-\la/\La)^{k-2}\La_1|v|/\La^{1/2} $ for some constant $C$ depending only on $d$. We can extend this inequality by using Corollary 6.1. Thus for $(\xi,\eta)$ satisfying the conditions of Lemma 5.2 for sufficiently small constant $C_1$ depending only on $d$, there exists $q_0(\La/\la)<2$ such that for some constant $C$ depending only on $d,\La/\la$  one has the inequality
 \be \label{DK5}
 \|A_k\|_q \ \le \  Ck(1-\la/\La)^{k-2}\left(1+C_2 |\Im\xi|^2/[\Re\eta/\La]\right)^{k-2}\La_1|v|/\La^{1/q} \  \ {\rm for \ } k\ge 2,
\ee
provided $q_0(\La/\la)\le q\le 2$.  We can bound now the RHS of (\ref{DC5})  in terms of  $\|g\|_p$ with $p=2q/(3q-2)$ by using (\ref{DK5}) and Young's inequality. If we combine this with the  inequality (\ref{DB5}) then we conclude from (\ref{CR5})  that
\begin{multline} \label{DL5}
\|a_k(g,\xi,\eta)v\| \ \le  \frac{C\La_1}{\La m^{3-2/p}} \|g(\cdot,\cdot)\|_p \  (1-\la/\La)^{k-1}|v| \\
+\frac{Ck\La_1}{ m^{2}\La^{3/2-1/p}} \|g(\cdot,\cdot)\|_p \  (1-\la/\La)^{k-2}\left(1+C_2 |\Im\xi|^2/[\Re\eta/\La]\right)^{k-2}|v|  \ .
\end{multline}
Evidently the inequality (\ref{DL5}) implies that (\ref{CD5}) holds provided $(\xi,\eta)$ satisfy the conditions of Lemma 5.2 for sufficiently small $C_1$ depending only on $d,\La/\la$.  Restricting $\xi$ to be in $\R^n$  then  (\ref{I6}) follows, and hence we have proven that   Hypothesis 3.2 holds in the massive field theory case. 

To complete the proof of Hypothesis 4.2 we first observe that the above argument immediately applies to the situation where the functions  $g_2,..,g_k$ are delta functions $g_j(x,t)=\del(x-x_j,t-t_j), \ j=2,..,k$.  The inequality (\ref{AJ*3}) then follows for general $g_2,..,g_k\in L^1(\Z^d\times\R,\C^d\otimes\C^d)$  from the triangle inequality.
\end{proof}

 \thanks{ {\bf Acknowledgement:} The authors would like to thank Tom Spencer for helpful conversations.

\end{document}